\documentclass[12pt]{article}

\usepackage{amsmath,amssymb,amsthm,mathtools}
\usepackage{graphicx,tikz}

\def\nfrac#1#2{{\textstyle\frac{#1}{#2}}}
\def\dfrac#1#2{\lower0.15ex\hbox{\large$\frac{#1}{#2}$}}

\usepackage[colorlinks=true,citecolor=black,linkcolor=black,urlcolor=blue]{hyperref}

\allowdisplaybreaks

\def\non{\nonumber}
\def\nonum{\nonumber}

\def\cF{\mathcal{F}}
\def\cG{\mathcal{G}}

\def\lm#1{#1_\times}

\def\E{\mathbb{E}}

\def\disp{\displaystyle}
\newtheorem{firstthm}{Proposition}[section]
\newtheorem{thm}[firstthm]{Theorem}

\newtheorem{lemma}[firstthm]{Lemma}
\newtheorem{cor}[firstthm]{Corollary}

\theoremstyle{definition}

\newtheorem{remark}[firstthm]{Remark}

\numberwithin{equation}{section}

\title{A threshold result for loose Hamiltonicity\\ in random regular uniform hypergraphs}

\author{
Daniel Altman\\
\small Mathematical Institute\\[-0.8ex]
\small University of Oxford\\[-0.8ex]
\small Oxford, OX2 6GG, U.K.\\
\small \texttt{daniel.h.altman@gmail.com}
\and
Catherine Greenhill\thanks{Supported by the Australian Research Council grant DP140101519.}\\
\small School of Mathematics and Statistics\\[-0.8ex]
\small UNSW Sydney\\[-0.8ex]
\small NSW 2052, Australia\\
\small \texttt{c.greenhill@unsw.edu.au} 
\and
Mikhail Isaev${}^*$\\
\small School of Mathematical Sciences\\[-0.8ex]
\small Monash University\\[-0.8ex]
\small VIC 3800, Australia\\
\small Moscow Institute of Physics and Technology\\[-0.8ex]
\small Dolgoprudny, 141700, Russia\\
\small \texttt{isaev.m.i@gmail.com}
\and
Reshma Ramadurai\\
{\small School of Mathematics and Statistics}\\[-0.8ex]
{\small Victoria University}\\[-0.8ex]
{\small Wellington, New Zealand }\\
{\small \texttt{reshma.ramadurai@vuw.ac.nz}}
}

\date{1 November 2019}

\begin{document}

\maketitle

\begin{abstract}
Let $\mathcal{G}(n,r,s)$ denote a uniformly random 
$r$-regular $s$-uniform hypergraph on $n$ vertices,
where $s$ is a fixed constant and $r=r(n)$ may grow with $n$.
An $\ell$-\emph{overlapping} Hamilton cycle 
is a Hamilton cycle in which successive edges
overlap in precisely $\ell$ vertices, and 1-overlapping Hamilton
cycles are called \emph{loose} Hamilton cycles.  

When $r,s\geq 3$ are fixed integers, we establish
a threshold result for the property of containing a loose Hamilton cycle.
This partially verifies a conjecture of
Dudek, Frieze, Ruci{\' n}ski and {\v S}ileikis (2015).
In this setting, we also find the asymptotic distribution of the number of loose Hamilton
cycles in $\mathcal{G}(n,r,s)$. 

Finally we 
prove that for $\ell = 2,\ldots, s-1$ and for $r$ growing moderately as
$n\to\infty$,
the probability that 
$\mathcal{G}(n,r,s)$ has a $\ell$-overlapping Hamilton cycle tends
to zero. 
\end{abstract}

\section{Introduction}\label{s:intro}

A \emph{hypergraph} $G=(V,E)$ consists of a finite set $V$ of vertices and a multiset
$E$ of multisubsets of $V$, which we call edges.
We say that $H$ is \emph{simple} if $E$ is a set of sets: that is, there are
no repeated edges and no edge contains a repeated vertex.
Given a fixed integer $s \geq 2$, the hypergraph $G$ is said to be $s$-\emph{uniform} 
if every edge contains precisely $s$ vertices, counting multiplicities. Uniform hypergraphs have been well-studied, as they generalise graphs (which are 2-uniform hypergraphs). 
Let $r\geq 1$ be an integer. 
A hypergraph is said to be $r$-\emph{regular} if every vertex has degree $r$, counting multiplicities.  
(For more background on hypergraphs, see~\cite{duchet})

For integers $r,s\geq 2$, let
$\mathcal{S}(n,r,s)$ be the set of all simple, $r$-regular, $s$-uniform hypergraphs on the vertex set 
$\{1,2,\ldots, n\}$.  To avoid trivialities, assume that $s$ divides $rn$
(as any hypergraph in $\mathcal{S}(n,r,s)$ has $rn/s$ edges).
We write $\mathcal{G}(n,r,s)$ to denote a random
hypergraph chosen uniformly from $\mathcal{S}(n,r,s)$.
Here $s$ is fixed, though we sometimes allow $r=r(n)$ to grow with $n$.

A 1-\emph{cycle} (or \emph{loop}) in a hypergraph is an edge which contains a 
repeated vertex, and a 2-\emph{cycle}
is an unordered pair of edges which intersect in at least 2 vertices.  
Hence a hypergraph is simple if and only if it contains no 1-cycle and 
no 2-cycle consisting of two identical edges.
For $k\geq 3$,
a set of $k$ edges forms a $k$-\emph{cycle} if for some ordering 
$e_0,e_1,\ldots, e_{k-1}$ of the edges, there exist \emph{distinct} vertices
$v_0,\ldots, v_{k-1}$ such that
$v_j\in e_j\cap e_{j+1}$ for $j=0,\ldots, k-1$ (identifying $e_k$ with $e_0$).

We particularly focus on loose cycles.
A 1-cycle $e$ is \emph{loose} if $e$ contains $s-1$ distinct vertices, and 
a 2-cycle $e_1, e_2$ is  \emph{loose} if $|e_1\cap e_2| = 2$ and $e_1$, $e_2$ are not loops.  For
$k\geq 3$, a $k$-cycle is \emph{loose} if it contains no loops and, for some ordering
$e_0,\ldots, e_{k-1}$ of its edges,
\[ |e_i\cap e_j| = \begin{cases} 1 & \text{ if $i-j\equiv \pm 1 \, (\operatorname{mod}\,  k)$,}\\ 0 
& \text{ otherwise}\end{cases}
\]
for all distinct $i,j\in \{ 0,\ldots, k-1\}$.
A loose $k$-cycle $C$ contains precisely $k(s-1)$ distinct vertices,
for any positive integer $k$.
Figure~\ref{hamilton}
shows a loose 6-cycle in a 3-uniform hypergraph. 

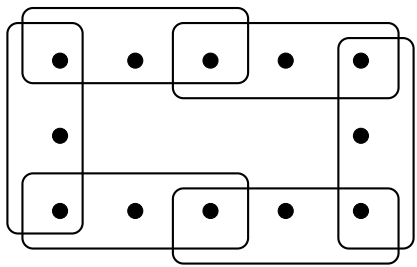
\begin{figure}[ht!]
\begin{center}
\begin{tikzpicture}[scale=1.0]
\draw [fill] (0.0,0.0) circle (0.1);
\draw [fill] (1.0,0.0) circle (0.1);
\draw [fill] (2,0.0) circle (0.1);
\draw [fill] (3,0) circle (0.1);
\draw [fill] (4.0,0.0) circle (0.1);
\draw [fill] (4.0,1) circle (0.1);
\draw [fill] (4,2) circle (0.1);
\draw [fill] (3,2) circle (0.1);
\draw [fill] (2,2) circle (0.1);
\draw [fill] (1,2) circle (0.1);
\draw [fill] (0,2) circle (0.1);
\draw [fill] (0,1) circle (0.1);
\draw [thick,rounded corners,-] (-0.5,2.7) rectangle (2.5,1.7);
\draw [thick,rounded corners,-] (1.5,2.5) rectangle (4.5,1.5);
\draw [thick,rounded corners,-] (-0.5,0.5) rectangle (2.5,-0.5);
\draw [thick,rounded corners,-] (1.5,0.3) rectangle (4.5,-0.7);
\draw [thick,rounded corners,-] (3.7,2.3) rectangle (4.7,-0.5);
\draw [thick,rounded corners,-] (-0.7,2.5) rectangle (0.3,-0.3);
\end{tikzpicture}
\label{hamilton}
\caption{A loose 6-cycle in a 3-uniform hypergraph.}
\end{center}
\end{figure}

Define $t = n/(s-1)$ and let $G$ be a hypergraph on $n$ vertices.
Observe that a loose $t$-cycle of $G$ covers all $(s-1)t = n$ vertices of $G$.
From now on, we refer to a loose $t$-cycle of $G$ as a 
\emph{loose Hamilton cycle}. 
A necessary condition for the existence of a loose Hamilton cycle in a
hypergraph $G\in\mathcal{S}(n,r,s)$ is that $s-1$ divides $n$. 

More generally, an $\ell$-\emph{overlapping} Hamilton cycle is
a set of $t_\ell = n/(s-\ell)$ edges which can be labelled $e_0, e_1,\ldots, e_{t_\ell-1}$
such that for some ordering $v_0,\ldots, v_{n-1}$ of the vertices we have
\begin{equation}
\label{edge-def}
 e_i = \{ v_{i(s-\ell)}, \, v_{i(s-\ell) + 1},\, \cdots, v_{i(s-\ell) + s-1}\} \quad
               \text{ for \,\, $i=0,\ldots, t_\ell-1$}.
\end{equation}
Here the vertex labels are also interpreted cyclically, so that for example
\[ e_{t_\ell-1} = \{ v_{n-s+\ell}, v_{n-s+\ell+1},\ldots, v_{n-1},\, v_0,\, v_1,\ldots, v_{\ell-1}\}.\]
A necessary condition for an $\ell$-overlapping Hamilton cycle to exist in an $s$-uniform hypergraph
on $n$ vertices is that $s-\ell$ divides $n$.
An $(s-1)$-overlapping Hamilton cycle is also called a \emph{tight} Hamilton cycle, and 
a 1-overlapping Hamilton cycle is just a loose Hamilton cycle.

In this paper, asymptotic results for $\ell$-overlapping Hamilton cycles hold as $n\to\infty$ 
restricted to the set
\[  \mathcal{I}^{(\ell)}_{(r,s)} = \{ n\in\mathbb{Z}^+ \, : \, s\mid rn \,\, \text{ and }
            \,\, (s-\ell) \mid n\}. \]
If the probability of an event tends to 1 as $n\to\infty$ along this set, then we say that
the event holds \emph{asymptotically almost surely} (a.a.s.).
Since our main focus is on loose
Hamilton cycles, we write $\mathcal{I}_{(r,s)}$ instead of $\mathcal{I}^{(1)}_{(r,s)}$
when $\ell=1$.

The case $s=2$ (graphs) has been extensively studied.
In order to prove that random $r$-regular graphs are a.a.s.\ Hamiltonian, for fixed $r\geq 3$,
Robinson and Wormald~\cite{RW92,RW94} 
used an analysis of variance technique now known as the
\emph{small subgraph conditioning method}. 
In~\cite{RW92}, Robinson and Wormald proved that random cubic graphs
are a.a.s.\ Hamiltonian, 
but their generalisation~\cite{RW94} to higher degrees used an inductive argument based
on the a.a.s.\ presence of perfect matchings in random regular graphs of
degree $r\geq 4$.  The ideas of~\cite{RW92,RW94} were further developed 
by Frieze et al.~\cite{fjmrw} and by Janson~\cite{janson}.
Frieze et al.~\cite{fjmrw} provided algorithmic results for the construction,
generation and counting of Hamilton cycles in random regular graphs,
while Janson~\cite{janson} applied small subgraph conditioning to
give the asymptotic distribution
of the number of Hamilton cycles in random $r$-regular graphs. 
Janson stated this distribution in~\cite[Theorem 2]{janson}, and noted
that in particular, the expected number of Hamilton cycles in random
$r$-regular graphs is asymptotically equal to
\begin{equation}
\label{exps2}
 e\, \sqrt{\frac{\pi}{2n}}\, \left( \frac{(r-2)^{(r-2)/2}\, (r-1)}{r^{(r-2)/2}}\right)^n.\end{equation}
Janson also observed that~\cite[Theorem 2]{janson} directly implies that
random $r$-regular graphs are a.a.s.\ Hamiltonian, when $r\geq 3$.

Our aim in this paper is to extend the results of Frieze et al.~\cite{fjmrw}
and Janson~\cite{janson} by using small subgraph conditioning to study loose 
Hamilton cycles in random $s$-uniform $r$-regular hypergraphs, for any $r,s\geq 2$.
Our work is also motivated by
two conjectures stated by Dudek et al.~\cite{cfmr}, as discussed
in Section~\ref{s:related}.
Where possible, we state our results so that they also cover the 
known results for graphs ($s=2$), though we stress that our proofs do not cover
this case.

\begin{thm}\label{main2}
Let $s\geq 2$ be a fixed integer. 
There exists a positive constant $\rho(s)$ such that for any fixed integer $r\geq 2$, 
as $n\to\infty$ along $\mathcal{I}_{(r,s)}$,
\[ 
    \Pr(\, \mathcal{G}(n,r,s) \, \text{ contains a loose Hamilton cycle})\,
  \longrightarrow \begin{cases} 1 & \text{ if $r >  \rho(s)$,}\\
  0 & \text{ if $r \leq  \rho(s)$.} \end{cases} 
\]
Specifically, $\rho = \rho(s)$ is the unique real number in $(2,\infty)$ such that
\[ (\rho-1)(s-1)\left(\frac{\rho s-\rho-s}{\rho s-\rho}\right)^{(s-1)(\rho s-\rho-s)/s}=1.\]
We note that $\rho(2) \in (2,3)$ and $\rho(3)=3$, while if $s\geq 4$ then
\begin{equation}
\rho^{-}(s) < \rho(s) < \rho^+(s)
\label{trap}
\end{equation}
where
\begin{align*}
 \rho^{-}(s) &= \frac{e^{s-1}}{s-1} -\frac{s-2}{2} - \frac{(s^2-s+1)^2}{(s-1)e^{s-1}},\\ 
  \rho^+(s) &= \frac{e^{s-1}}{s-1} - \frac{s-2}{2}.
\end{align*}
\end{thm}

In Table~\ref{rhobounds} we give the values of $\rho(s)$  for $s=2,\ldots, 10$,
and for $s = 4,\ldots, 10$ we compare $\rho(s)$ with the
lower and upper bounds $\rho^-(s)$, $\rho^+(s)$ given in (\ref{trap}).
\begin{table}[ht!]
\renewcommand{\arraystretch}{1.4}
\begin{center}
\begin{tabular}{|c|c|c|c|c|c|c|c|c|c|c|}
\hline
$s$ &  2 & 3 & 4 & 5 & 6 & 7 & 8 & 9 & 10 \\
\hline
$\rho^{-}(s)$  
         & --  & -- & 2.891 & 10.130 &  26.388 & 63.974 & 153.239 & 368.896 & 896.229  \\
\hline
$\rho(s)$ & 2.488 & 3 & 5.501 & 11.998 & 27.580 & 64.675 & 153.625 & 369.100 & 896.332  \\
\hline
$\rho^+(s)$ & -- & -- & 5.695 & 12.150 & 27.683 & 64.738 & 153.662 & 369.120 & 896.342 \\
\hline
\end{tabular}
\label{rhobounds}
\end{center}
\caption{Values of $\rho(s)$ for small $s$, together with our bounds (for $s\geq 4$).}
\end{table}
All values are rounded to three decimal places, except $\rho(3)=3$,
since it is an integer.  (Three decimal places are required to see
that $\rho(5) < 12$.)
We see that
$\rho(s)$ is closely approximated by the upper bound $\rho^+(s)$,
except at very small values of $s$. 

Furthermore, we establish the asymptotic distribution of the number of
loose Hamilton cycles in $\mathcal{G}(n,r,s)$.   We present this result
in Theorem~\ref{distribution} below, as it requires the definition of
several crucial parameters which will be specified later.

The following result concerns $\ell$-overlapping Hamilton cycles
and allows the degree $r$ to grow moderately with $n$.  

\begin{thm}
Let $s\geq 3$ be a fixed integer and let
\[ \kappa = \kappa(s) = \begin{cases} 1 & \text{ if $s\geq 4$,}\\       
                            \nfrac{1}{2} & \text{ if $s=3$}.
           \end{cases}
\]
Suppose that $r=r(n)$ with $2\leq r = o(n^\kappa)$.
Then as $n\to\infty$ along $\mathcal{I}^{(\ell)}_{(r,s)}$, a.a.s.\ 
$\cG(n,r,s)$ has no $\ell$-overlapping Hamilton cycle for $\ell=2,\ldots, s-1$.
\label{new-theorem}
\end{thm}

To prove these results, as is usual in this area,
 we will work in a related probability model known as the
\emph{configuration model}. 
After discussing some related results and extensions in Section~\ref{s:related},
we review the configuration model for hypergraphs in Section~\ref{s:mainideas}
and prove Theorem~\ref{new-theorem} in Section~\ref{ss:expectation}.
To prove Theorem~\ref{main2} we will apply
the small subgraph conditioning method, which is discussed in Section~\ref{ss:partitions}. 
The structure of the rest of the paper will be described in
Section~\ref{s:structure}.

\subsection{Extensions and related results}\label{s:related}

The small subgraph conditioning method has been applied to prove
many a.a.s.\ structural theorems (contiguity results)
for regular graphs. See Wormald~\cite{wormald}
or Janson~\cite{janson} for more detail.  For uniform regular hypergraphs,
we only know of one application of the method:
Cooper et al.~\cite{cfmr} used small subgraph conditioning to investigate 
perfect matchings in random regular uniform hypergraphs.
They proved a threshold result for 
existence of a perfect matching in a random $r$-regular $s$-uniform hypergraph,
where $r, s\geq 2$ are fixed integers. Specifically, in~\cite[Theorem 1]{cfmr}
they proved that as $n\to\infty$, this probability tends to 0 
if $s > \sigma_r$ and tends to 1 if $s < \sigma_r$, where
\[ \sigma_r = \frac{\ln r}{(r-1)\ln\left(\frac{r}{r-1}\right)} + 1.\]
Defining
$r_0(s) = \min\{ r : s < \sigma_r\}$, Cooper et al.~\cite{cfmr} remark that
$r_0(s)$ is approximately $e^{s-1}$.
It is interesting to observe that, in contrast to graphs ($s=2$),
the threshold for the
existence of perfect matchings in $\cG(n,r,s)$ is higher than the threshold for the existence of
loose Hamilton cycles when $s\geq 3$;  that is, $r_0(s) > \rho(s)$.

Recently, Dudek et al.~\cite{dfrs-sandwich} established a relation between $\cG(n,r,s)$ and 
the uniform probability model
on the set of $s$-uniform hypergraphs on $n$ vertices with $m$ edges (when $s\geq 3$). 
Using known results about 
the existence of loose Hamilton cycles in the latter model, they showed that a.a.s.\ $\cG(n,r,s)$ 
contains a loose Hamilton cycle when $r \gg \ln n $ 
(or $r = \Omega(\ln n)$, if $s=3$) and $r = O(n^{s-1})$.  

Dudek et al.\ made the following conjecture~\cite[Conjecture 1]{dfrs},
rewritten here in our notation:
\begin{quote} 
For every $s \geq 3$ there exists a constant $\rho = \rho(s)$ such that for any $r \geq \rho$, $\cG(n,r,s)$ contains a loose Hamilton cycle a.a.s.
\end{quote}
We have partially verified this conjecture with Theorem~\ref{main2}, which gives a threshold 
result for constant values of $r$.  This leaves a gap for degrees $r=r(n) = O(\ln n)$.
Intuitively, it seems that increasing the degree should make the
existence of a loose Hamilton cycle more likely, but it appears that other ideas are required.  

Aldosari and Greenhill~\cite{AG} 
applied a switching argument to provide an asymptotic formula for the expected number
of loose Hamilton cycles in $\cG(n,r,s)$ when $r,s$ are slowly-growing.    We will use this
formula in the proof of Corollary~\ref{lem:looseHam}.

There has been much work on $\ell$-overlapping Hamilton cycles in the binomial model 
$\cG_{n,p}^{(s)}$ of
$s$-uniform hypergraphs, where each $s$-set is an edge with probability $p$,
independently.  
In particular, loose ($\ell=1$) and tight ($\ell=s-1$) Hamilton cycles
are well studied, see for example~\cite{ABKP,df,ferber} and references therein.
Dudek and Frieze~\cite[Theorem 3(i)]{df} proved that for all fixed integers $s > \ell\geq 2$ and 
fixed $\epsilon > 0$, if $p \leq (1-\varepsilon) e^{s-\ell}/n^{s-\ell}$ then a.a.s.\
$\cG_{n,p}^{(s)}$ has no $\ell$-overlapping Hamilton cycle.
This motivated the second conjecture of Dudek et al.~\cite[Conjecture 2]{dfrs},
written here in our notation:
\begin{quote}
For every $s > \ell \geq 2$, if $r \gg n^{\ell-1}$ then a.a.s.\ $\cG(n,r,s)$ contains an
$\ell$-overlapping Hamilton cycle. 
\end{quote}
Dudek et al.~\cite[Theorem 5]{dfrs-sandwich} 
proved most of this conjecture, showing that if
$s> \ell\geq 2$ then there exists a constant $C>0$ such that if
\[ \ell = 2 \,\,\, \text{ and } \,\,\, n \ll r \leq \frac{n^{s-1}}{C},\]
or if
\[ \ell \geq 3 \,\,\, \text{ and }\,\,\,  C n^{\ell-1} \ll r \leq \frac{n^{s-1}}{C},\]
then a.a.s.\ $\cG(n,r,s)$ contains an $\ell$-overlapping Hamilton cycle.
Dudek et al.~\cite{dfrs-sandwich} conjectured that in each of these
situations, the lower bounds for $r$ is
a threshold for $\ell$-overlapping Hamiltonicity.
Recently, Espuny D{\' i}az et al.~\cite[Corollary~3.13]{EJKO} proved that 
when $\ell\geq 2$,
the probability that $\cG(n,r,s)$ contains an $\ell$-overlapping Hamilton cycle tends to~0
if $r=o(n^{\ell-1})$ and tends to~1 if $r=\omega(n^{\ell-1})$ and $r=o(n^{s-1})$. 
Their arguments require $r$ to tend to infinity, and they rely on our Theorem~\ref{new-theorem}
to handle the case of constant~$r$.
This solves the conjecture of Dudek et al.~\cite{dfrs-sandwich} for all $r \leq n^{s-1}/C$.
For larger values of $r$,
we believe that complex-analytic methods such as those presented in~\cite{mother}
may allow further progress.

\section{Main ideas}\label{s:mainideas}

To study properties of  random regular uniform hypergraphs it is convenient to  
work in the configuration model, which we briefly review here. This is the same model used by Cooper et al.~\cite{cfmr}.
We use the notation $[n] = \{ 1,2,\ldots, n\}$.
Throughout the paper, we use the convention that $0^0=1$.

Let $B_1, B_2,\ldots, B_n$ be disjoint sets of size $r$, which we call \emph{cells}, and
define $\mathcal{B} = \cup_{i=1}^n B_i$.  Elements of $\mathcal{B}$ are called
\emph{points}.   We assume that there is a fixed ordering on the $rn$ points of $\mathcal{B}$
(so that different points in the same cell are distinguishable). 

Assume that $s$ divides $rn$. 
Let $\Omega(n,r,s)$ be the set of all unordered partitions $F = \{ U_1,\ldots, U_{rn/s}\}$
of $\mathcal{B}$ into $rn/s$ parts, where each part has exactly $s$ points. 
Each partition $F\in \Omega(n,r,s)$ defines a hypergraph $G(F)$ on the vertex set $[n]$
in a natural way: vertex $i$ corresponds to the cell $B_i$, and each part $U\in F$ gives
rise to an edge $e_U$ such that the multiplicity of vertex $i$ in $e_U$ equals $|U\cap B_i|$,
for $i=1,\ldots, n$. Then $G(F)$ is an $s$-uniform $r$-regular hypergraph.

The partition
$F\in\Omega(n,r,s)$ is called \emph{simple} if $G(F)$ is simple. More generally, we will often describe $F \in \Omega(n,r,s)$ as having a particular hypergraph property  if $G(F)$ has that property.
A subpartition $X$ of $F$ is called a
\emph{cycle} if the hypergraph corresponding to $X$, denoted $G(X)$, is a cycle. 
Hence we may speak of a 1-cycle, loose $k$-cycle or $\ell$-overlapping Hamilton cycle in a
partition $F$.

For any nonnegative integer $\ell$ which is divisible by $s$, define
\begin{equation}
\label{p-def}
p(\ell) = \frac{\ell!}{(\ell/s)!\, (s!)^{\ell/s}}.
\end{equation}
Then 
\begin{equation} 
|\Omega(n,r,s)| = p(rn) = \frac{(rn)!}{(rn/s)!\, (s!)^{rn/s}}. \label{sizeomega}
\end{equation}
More generally, if $k$ parts in a partition have already been specified then
there are $p(rn-sk)$ elements of $\Omega(n,r,s)$ which contain the specified $k$ parts.

Every hypergraph in $\mathcal{S}(n,r,s)$
corresponds to precisely $(r!)^n$ partitions $F\in\Omega(n,r,s)$. (This is only true for simple hypergraphs.) 
Therefore
\begin{equation}
 |\mathcal{S}(n,r,s)| = \frac{|\Omega(n,r,s)|\, \Pr(\cF(n,r,s) \text{ is simple})}{(r!)^n},
\label{simpleratio}
\end{equation}
where $\mathcal{F}(n,r,s)$ denotes a random partition chosen uniformly from $\Omega(n,r,s)$. 
Observe also that
$\cG(n,r,s)$ has the same distribution as $G(\cF(n,r,s))$, conditioned on the event that
$\cF(n,r,s)$ is simple.

Now we explain how to translate asymptotic properties  for the random regular uniform hypergraphs from  the configuration model.
First, for any event $\mathcal{A}\subseteq \Omega(n,r,s)$, we can bound 
\begin{align}
  \Pr(\cF(n,r,s)\in \mathcal{A}\mid &\cF(n,r,s)  \text{ is simple}) 
  \leq \frac{\Pr(\cF(n,r,s)\in \mathcal{A})}{\Pr(\cF(n,r,s) \text{ is simple})}.
\label{conditioning}
\end{align}
Cooper et al. showed in \cite{cfmr} that when $r, s\geq 3$ are fixed,
\begin{equation*}
 \lim_{n \to \infty} \Pr(\cF(n,r,s) \text{ is simple}) = e^{-(r-1)(s-1)/2}.
\end{equation*} 
They also~remark in~\cite[Section 2]{cfmr}, the probability
that two parts in $\cF(n,r,s)$ give rise to a repeated edge is $o(1)$.
 The next lemma follows from these  bounds.

\begin{lemma}\label{translate}
Fix integers $r\geq 2$ and $s\geq 3$.  For any positive integer $n$ such that $s$ divides $rn$,
let
$\widehat{\cF}$ be a uniformly random partition in $\Omega(n,r,s)$ with no 1-cycles (loops) and let
$\cF_{{\mathcal{S}}}$ be a uniformly random simple partition in $\Omega(n,r,s)$.
Write $\cF=\mathcal{F}(n,r,s)$, for ease of notation.
Finally, let $h: \Omega(n,r,s) \rightarrow \mathbb{Z}$. 
Then as $n \rightarrow \infty$ along integers such that $s$ divides $rn$, the following
two properties hold.
\begin{itemize}	
\item[\emph{(a)}] If  \, $\Pr(h(\cF) \in A) = o(1)$ \, then  \,
  $\Pr(h(\cF_{\mathcal{S}}) \in A) =o(1)$  \, for any \, $A \subset \mathbb{Z}$.
\item[\emph{(b)}] $\Pr(h(\widehat{\cF}) \in A) - 	\Pr(h(\cF_{\mathcal{S}}) \in A)  = o(1)$  
for any  $A \subset \mathbb{Z}$. 
\end{itemize}
\end{lemma}

\begin{proof}
Property (a) follows immediately from (\ref{conditioning}).
For (b), observe that $h(F)\in A$ if and only if $F\in h^{-1}(A)$, for any set
$A\subseteq \mathbb{Z}$ and any $F\in\Omega(n,r,s)$.  
Let $\mathcal{R}\subseteq \Omega(n,r,s)$ be the set of partitions which give
rise to hypergraphs with repeated edges. 
Now for any $B\subseteq \Omega(n,r,s)$,
\begin{align*}
 \big| \Pr(\widehat{\cF}\in B) - \Pr(\cF_{\mathcal{S}}\in B) \big|
 &= \big| \Pr(\widehat{\cF}\in B) - \Pr(\widehat{\cF}\in B\mid \widehat{\cF}\in \overline{\mathcal{R}}) \big| \\
 &= \Pr(\widehat{\cF}\in \mathcal{R})\,
      \big| \Pr(\widehat{\cF}\in B\mid \widehat{\cF}\in \mathcal{R}) - 
   \Pr(\widehat{\cF}\in B\mid \widehat{\cF}\in \overline{\mathcal{R}})\big|\\
 &\leq \Pr(\widehat{\cF}\in\mathcal{R})\\
 &\leq \frac{\Pr(\cF \in\mathcal{R})}{\Pr(\cF \text{ is simple})}\\
 &= o(1),
\end{align*}
completing the proof.
\end{proof}

\noindent Property (a) can be described as asymptotic absolute continuity of 
$h(\cF_{\mathcal{S}})$ with respect to
$h(\cF)$, and property (b) can be described as asymptotic equivalence in distribution of
$h(\widehat{\cF})$ and $h(\cF_{\mathcal{S}})$.

\subsection{Expected value in configuration model}\label{ss:expectation}

For $\ell=1,\ldots, s-1$, let $Y^{(\ell)}$ be the number of $\ell$-overlapping Hamilton cycles
in $\cF(n,r,s)$.   We now find an expression and an upper bound for the
expected value of $Y^{(\ell)}$.
Since $r=r(n)$ may depend on $n$ in this section, interpret $\mathcal{I}_{(r,s)}^{(\ell)}$
to be the set of all positive integers $n$ such that $s\mid r(n)\, n$ and $(s-\ell)\mid n$.
When $r, s$ are both constant, the set $\mathcal{I}_{(r,s)}^{(\ell)}$ is guaranteed to be infinite.  
To cover the case of non-constant $r=r(n)$, we include this condition in the lemma statement.

\begin{lemma}
Let $\ell\in \{ 1,\ldots, s-1\}$, where $s\geq 3$ is fixed,
and let $r=r(n)$ be a function of $n$ which satisfies $r\geq 2$ and $r(s-\ell)\geq s$.
Suppose that $\mathcal{I}^{(\ell)}_{(r,s)}$ contains infinitely many positive integers $n$.
Write $\ell = q(s-\ell) + c$
where $q,c$ are nonnegative integers and $c\in \{ 0,1,\ldots, s-\ell-1\}$.
Then as $n\to\infty$ along $\mathcal{I}^{(\ell)}_{(r,s)}$,
\begin{align*}
\E Y^{(\ell)}
&= \frac{n!}{2t_\ell}\, 
    \left(\frac{\left((r)_{q+1}\right)^{s-\ell}\, (r-q-1)^c \, (s)_{\ell+c} }
               {c!}\, \right)^{t_\ell} \,
           \frac{(rn/s)_{t_\ell} }{(rn)_{st_\ell}}. 
\end{align*}
When $\ell\geq 2$ and $r= o(n^{\ell-1})$ we have $\E Y^{(\ell)} \leq e^{-n}$.
\label{nr-l-overlap}
\end{lemma}

\begin{proof}
Recall the notation $t_\ell = n/(s-\ell)$.
There are $n!$ ways to fix an ordering $v_0,\ldots, v_{n-1}$ of the vertices.  
This gives rise to an $\ell$-overlapping
Hamilton cycle $H$ with edges $e_0,\ldots, e_{t_\ell-1}$ defined by (\ref{edge-def}).    
For $j=0,\ldots, t_\ell-1$ define the set $S_j = e_j\setminus e_{j-1}$,
with index arithmetic performed cyclically (so $S_0 = e_0 \setminus e_{t_\ell-1}$).
The sets $S_0,\ldots, S_{t_{\ell}-1}$ partition $[n]$.

First, observe that
$s-\ell-c$ vertices in $S_i$ have degree $q+1$ in the Hamilton cycle $H$,
and the remaining $c$ vertices of $S_i$ have degree $q+2$ in $H$, 
for all $i\in \{ 0,1,\ldots, t_\ell-1\}$.
It follows that the chosen Hamilton cycle $H$ (as a hypergraph) corresponds to precisely
\[ 2t_\ell\, \left( c!\, (s-\ell-c)!\right)^{t_\ell}\]
orderings of the vertices.  

The number of ways to embed the edges of $H$ as parts in a partition $F$ is
\[ \left(\left((r)_{q+1}\right)^{s-\ell-c} \, \left((r)_{q+2}\right)^{c}\right)^{t_\ell}
   = \left( \left( (r)_{q+1}\right)^{s-\ell} (r-q-1)^c\right)^{t_\ell},\]
This completely specifies the $t_\ell$ parts of the partition which correspond to $H$.
Finally, we must multiply by
\[ \frac{p(rn-st_\ell)}{p(rn)}\]
for the probability that a randomly chosen partition contains these $t_\ell$ specified parts.
We have shown that
\begin{align}
\E Y^{(\ell)} &= \frac{n!}{2t_\ell\, \left(c!\, (s-\ell-c)!\right)^{t_\ell}}\,
    \left(\left((r)_{q+1}\right)^{s-\ell}\, (r-q-1)^c \right)^{t_\ell} \,
           \frac{p(rn-st_\ell)}{p(rn)} \nonumber \\
&= \frac{n!}{2t_\ell}\,
    \left(\frac{\left((r)_{q+1}\right)^{s-\ell}\, (r-q-1)^c \, s!}
 {c!\, (s-\ell-c)!}\, \right)^{t_\ell} \,
           \frac{(rn-s t_{\ell})!\, (rn/s)!\,  }{(rn/s - t_\ell)!\,  \,
                 (rn)!}, \label{pre-Stirling}
\end{align}
which proves the first statement of the lemma, after some cancellation. 
(Observe that $rn\geq st_{\ell}$, by assumption.)

For the remainder of the proof, suppose that $\ell \geq 2$ and $r=o(n^{\ell-1})$.
Recall Stirling's inequalities 
\[ \sqrt{2\pi a} \, (a/e)^a \leq a!\leq e^{1/(12 a)}\, \sqrt{2\pi a} \, (a/e)^a\]
which hold for all positive integers $a$.  If $rs-r\ell-s > 0$ then applying Stirling's inequalities to
(\ref{pre-Stirling}) gives
\begin{align}
 &  \E Y^{(\ell)} \nonumber \\
 & \leq  e^{o(n)}\, \left(\frac{e^{\ell-1}\, ((r)_{q+1})^{s-\ell}\, (r-q-1)^c\, (s-1)_{c+\ell-1}\,
   (rs-r\ell-s)^{(s-1)(rs-r\ell-s)/s}}{n^{\ell-1}\, c!\, r^{r(s-1)(s-\ell)/s}\, (s-\ell)^{(s-1)(rs-r\ell-s)/s}}\right)^{t_\ell}.
\label{firstcase}
\end{align}
When $rs-r\ell-s=0$ the factors $(rn-st_\ell)!$ and $(rn/s-t_\ell)!$ both equal~1,
and applying Stirling's inequalities to $n!$, $(rn)!$ and $(rn/s)!$  in (\ref{pre-Stirling}) leads to
\begin{align*}
 \E Y^{(\ell)} 
 &\leq  e^{o(n)}\, \left(\frac{e^{(rs-r-s)(s-\ell)/s}\, ((r)_{q+1})^{s-\ell}\, (r-q-1)^c\, s! }
    {n^{(rs-r-s)(s-\ell)/s}\, c!\, (s-\ell-c)!\, r^{r(s-1)(s-\ell)/s}\, 
      s^{r(s-\ell)/s}}\right)^{t_\ell}.
\end{align*}
But in this case $\ell-1=(rs-r-s)/r$ and $s-\ell=s/r$.  Substituting these identities into the above
expression shows that (\ref{firstcase}) holds in all cases, since $0^0=1$, and hence
\[ \E Y^{(\ell)} = \left(\frac{O(r)}{n^{\ell-1}}\right)^{t_\ell}. \]
This proves the second statement of the lemma.
\end{proof}

We are now in a position to prove Theorem~\ref{new-theorem}.

\begin{proof}[Proof of Theorem~\ref{new-theorem}]\
Recall from the statement of Theorem~\ref{new-theorem} 
that $\kappa=\kappa(s)$ equals 1 if $s\geq 4$, and equals $\nfrac{1}{2}$ when $s=3$.
Fix $\ell=2,\ldots, s-1$, where $s\geq 3$ is constant and
$r=r(n)$ may grow with $n$, such that $2\leq r = o(n^\kappa)$.
Dudek et al.~\cite[Theorem 1]{dfrs-enum} proved that when $r=o(n^{\kappa})$,
\[
  |\mathcal{S}(n,r,s)| 
          = \frac{(rn)!}{(rn/s)!\, (s!)^{rn/s}\, (r!)^n}\,
         \exp\Big( - \nfrac{1}{2}(r-1)(s-1) + O\big((r/n)^{1/2} + r^2/n\big)\, \Big).\]
Combining this with (\ref{sizeomega}) and (\ref{simpleratio}), 
we conclude that when $r=o(n^{\kappa})$,
\begin{align*}
 \Pr( \cF(n,r,s) \text{ is simple}) 
        &= \exp\big( - \nfrac{1}{2} (r-1)(s-1)(1 + o(1))\, \big) = \Omega(\exp( -\hat{c}_s\, r ))
\end{align*}
for some positive constant $\hat{c}_s$ (independent of $r$).
By (\ref{conditioning}), Lemma~\ref{nr-l-overlap} and Markov's Inequality, 
it follows that the probability
that $G\in\cG(n,r,s)$ contains an $\ell$-overlapping Hamilton cycle is bounded
above by 
\begin{align*} \frac{\E Y^{(\ell)}}{\Pr(\cF(n,r,s) \text{ is simple})} 
                  &=  O\left(\exp(\hat{c}_s\, r -n)\right) 
\end{align*}
which is $o(1)$ for any $r=o(n)$.  
This completes the proof, by (\ref{conditioning}). 
\end{proof}

We are particularly interested in loose Hamilton cycles ($\ell=1$) for fixed $r$.
Let $Y=Y^{(1)}$ denote the number of loose Hamilton cycles in $\cF(n,r,s)$.  

\begin{cor}
Let $r\geq 2$ and $s\geq 3$ be fixed integers and let $t=n/(s-1)$.
The expected value of $Y$ satisfies
\begin{align}
\E Y &=  \frac{n! }{2 t\, {((s-2)!)^t}}\,\cdot\, r^n (r-1)^t\,
\cdot \, \frac{p(rn-st)}{p(rn)} \label{EY-raw} \\
 &\sim \sqrt{\frac{\pi}{2n}}\, (s-1)\, \left(
    (r-1)(s-1)\, \left(\frac{rs-r-s}{rs-r}\right)^{(s-1)(rs-r-s)/s} \right)^{n/(s-1)}\nonum
\end{align}
as $n\to\infty$ along $\mathcal{I}_{(r,s)}$.
Furthermore, if $Y_{\cG}$ denotes the number of loose Hamilton cycles in $\cG(n,r,s)$ then
\[ \E Y_{\cG} \sim \exp\left( \frac{(s-1)(rs-s-2)}{2(rs-r-s)}\right)\, \E Y.\]
\label{lem:looseHam}
\end{cor}

\begin{proof}
The first statement follows from substituting $\ell=1$ into 
Lemma~\ref{nr-l-overlap} and using Stirling's approximation.
The proof is completed by comparing this asymptotic expression with the asymptotic
expression for $\E Y_{\cG}$ given in~\cite[Corollary~3.2]{AG}.
\end{proof}

\medskip

In Lemma~\ref{rho-val} we characterise pairs $(r,s)$ for which
$\E Y$ tends to infinity, leading to the definition of the threshold
function $\rho(s)$.
Combining this result with (\ref{conditioning}), we obtain the negative part of
the threshold result Theorem~\ref{main2}, as explained in Section~\ref{s:proof-main2}.
In order to complete the proof of Theorem~\ref{main2}, 
we require more information about the asymptotic distribution of the number of
loose Hamilton cycles in $\cF(n,r,s)$.  This information is obtained using the 
small subgraph conditioning method.

\subsection{Small subgraph conditioning for hypergraphs}\label{ss:partitions}

The following statement of the small subgraph conditioning method is adapted 
from \cite[Theorem~1]{janson}. 
A similar theorem is given in \cite[Theorem 4.1]{wormald}.

\begin{thm}[\cite{janson}]\label{thm:janson}
Let $\lambda_k > 0$ and $\delta_k \ge -1$, $k = 1,2, \dots,$ be constants and 
suppose that for each $n$ there are random variables $X_{k,n}$, $k = 1,2, \dots$, and 
$Y_n$ (defined on the same probability space) such that $X_{k,n}$ is nonnegative integer 
valued and $\E Y_n \ne 0$ and furthermore the following conditions are satisfied:
\begin{enumerate}
\item[\emph{(A1)}] $X_{k,n} \overset{d}{\rightarrow}  Z_k$ as $n \to \infty$, 
jointly for all k, where $Z_k \sim \text{Po}(\lambda_k)$ are independent Poisson random 
variables;
\item[\emph{(A2)}] For any finite sequence $x_1, x_2, \dots, x_m$ of nonnegative integers,
\[ \disp \frac{\E(Y_n | X_{1,n} = x_1, X_{2,n} = x_2, \dots, X_{m,n} = x_m)}{\E Y_n} 
   \to \prod_{k=1}^m \left( 1 + \delta_k \right)^{x_k} e^{-\lambda_k \delta_k} 
  \,\, \text{ as } \,\,n \to \infty;\]
\item[\emph{(A3)}] $\disp \sum_{k\geq 1} \lambda_k \delta_k^2 < \infty$;
\item[\emph{(A4)}] $\disp \frac{\E( Y_n^2)}{(\E Y_n)^2} \to \exp\left(\sum_{k\geq 1} \lambda_k \delta_k^2 \right)$ as $n \to \infty$.
\end{enumerate}
Then
\begin{equation}
\label{Wdef}
 \disp \frac{Y_n}{\E Y_n}  \stackrel{d}{\longrightarrow} W = \prod_{k=1}^\infty \left(1 + \delta_k\right)^{Z_k} e^{-\lambda_k \delta_k}\,\, \text{ as } \,\,n \to \infty;
\end{equation}
moreover, this and the convergence (A1) hold jointly. 
The infinite product defining $W$ converges a.s. and in $L^2$, with 
\[ \E W = 1 \text{ and } 
  \E W^2 = \exp\left(\sum_{k\geq 1} \lambda_k \delta_k^2 \right) = 
    \lim_{n \to \infty} \frac{\E( Y_n^2)}{(\E Y_n)^2}.\]
Furthermore, if $\delta_k > -1$ for all $k$ then a.a.s.\ $Y_n > 0$.
\end{thm}

Janson remarks in~\cite{janson} that in the asymptotics, the index set
$\mathbb{Z}^+$ may be replaced by any other countably-infinite set. The same is
true for the other results stated in this section.

We emphasise that for the remainder of the paper, $r\geq 3$ is a fixed integer.
Recall that $t = n/(s-1)$ is the number of edges in a loose Hamilton cycle.
We will apply Theorem~\ref{thm:janson} to the random variables defined as follows.
In order to distinguish our specific random variables from the general random variables
used in Theorem~\ref{thm:janson}, we do not include the subscript $n$ in our notation.
\begin{itemize}
\item  Let $Y$ be the number of subsets $F_H$ of $\cF(n,r,s)$ consisting
of $t$ parts such that $G(F_H)$ is a loose Hamilton cycle. 
\item For $k\geq 2$ let $X_k$ be the number of subsets $F_C$ of $\cF(n,r,s)$
consisting of $k$ parts such that $G(F_C)$ is a loose $k$-cycle. 
\item Let $X_1$ be the number of parts $U$ in $\cF(n,r,s)$ such that $U$ gives rise
to an edge which contains a repeated vertex. 
That is, $X_1$ is the number of parts $U$ in $\cF(n,r,s)$ such that
$|U\cap B_j| > 1$ for some $j\in [n]$.  
\end{itemize}
Note that $X_1$ counts parts which correspond to 1-cycles, not just
loose 1-cycles.
We define $X_1$ in this way so that $X_1=0$ if and only if 
no edge of $G(F)$ contains a repeated vertex.
(Our definition of $X_1$ agrees with that used in~\cite{cfmr}.)

Cooper et al. proved in~\cite[Section 5]{cfmr} that $X_k\rightarrow Z_k$ 
as $n\to\infty$, jointly
for $k\geq 1$, where $Z_k\sim \operatorname{Po}(\lambda_k)$ are asymptotically 
independent Poisson random variables with mean 
\begin{equation}
\label{lambdak}
\lambda_k = \frac{\left((r-1)\,(s-1)\right)^k}{2k}.
\end{equation}
This verifies that (A1) of Theorem~\ref{thm:janson} holds.
In fact, Cooper et al.~\cite{cfmr} worked with the random variable $X_k'$ for $k\geq 1$
which counts the number of $k$-cycles (not necessarily loose). 
Note that $X_1'=X_1$.
Calculations from~\cite[Section 5]{cfmr} show that $X_k'\stackrel{d}{\sim} X_k$
jointly for $k\geq 1$, since a.a.s.\ the contribution to $X_k$ from non-loose 
$k$-cycles forms only a negligible fraction of $X_k$.
Here we write $A_n \stackrel{d}{\sim} B_n$ to mean that two sequences
of random variables $(A_n)$ and $(B_n)$ have the same asymptotic distribution,
recalling that both $X_k$ and $X_k'$ depend on $n$.
Hence Theorem~\ref{thm:janson} (A1) holds with $\lambda_k$ as in (\ref{lambdak}).

\bigskip

In order to establish (A2) of Theorem~\ref{thm:janson},
the following result (for general random variables) is convenient.

\begin{lemma}[{\cite[Lemma 1]{janson}}]
\label{lemma:janson}
Let $\lambda_k' \ge 0$, $k=1,2,\dots$, be constants. Suppose that \emph{(A1)} holds, that $Y_n \ge 0$ and that
\begin{enumerate}
	\item[\emph{(A2${}^\prime$)}] for every finite sequence $x_1, x_2, \dots, x_m$ of nonnegative integers
	$$\disp \frac{\E\left(Y_n (X_{1,n})_{x_1}( X_{2,n})_{x_2} \dots (X_{m,n})_{x_m}\right))}{\E Y_n} \to 
  \prod_{k=1}^m \left( \lambda_k'\right)^{x_k} \text{ as } n \to \infty.$$
\end{enumerate}
Then \emph{(A2)} holds with $\lambda_k(1 + \delta_k) = \lambda_k'$ for all $k\geq 1$.
\end{lemma}
The arguments of
Section~\ref{s:framework} and Section~\ref{s:shortcycles} can be extended
to show that, for fixed integers $r,s\geq 2$,
\begin{equation}
\label{omit}
  \frac{\E\left(Y (X_{1})_{x_1}( X_{2})_{x_2} \dots (X_{m})_{x_m}\right))}{\E Y} = (1+o(1))
   \prod_{k=1}^m \left(\frac{\E( Y X_{k})}{\E Y}\right)^{x_k},
\end{equation} 
similarly to the case of Hamilton cycles in cubic graphs~\cite[equation (2.8)]{RW92}.  
Roughly, (\ref{omit}) holds because fixing a constant-length cycle has asymptotically negligible
effect on the number of ways to choose subsequent cycles, and noting that overlapping
cycles also give negligible relative contribution.
For completeness, 
the proof of (\ref{omit}) is given in Appendix~\ref{SomeRoutineUsuallyIgnored};
see Lemma~\ref{more-than-one-cycle}.

The distribution of the number of loose Hamilton cycles in
$\mathcal{G}(n,r,s)$  is asymptotically equivalent to the conditional distribution of the random variable $Y$  on the event that $\cF(n,r,s)$ has no 1-cycles (loops), as described  in Lemma \ref{translate}(b). Therefore, the following 
(general) corollary of Theorem~\ref{thm:janson} will be very useful for us (namely, in the proof of Theorem~\ref{distribution}).

\begin{cor}
Suppose that $Y_n$ and $X_{k, n}$ satisfy conditions \emph{(A1)--(A4)} of Theorem~\emph{\ref{thm:janson}}.
Let $\widehat{Y}_n$ be the random variable obtained from $Y_n$ by conditioning
on the event $X_{1,n}=0$.  
Then
\[ \frac{\widehat{Y}_n}{\E Y_n} \stackrel{d}{\longrightarrow } 
e^{-\lambda_1\delta_1}\,
  \prod_{k=2}^\infty \left(1 + \delta_k\right)^{Z_k} e^{-\lambda_k \delta_k}\,\, 
                     \text{ as } \,\,n \to \infty.
\]
Moreover, if $\delta_k > -1$ for all $k\geq 1$ then a.a.s.\ $\widehat{Y}_n > 0$.
\label{W-conditional}
\end{cor}

\begin{proof}
Let $\widehat{W}$ be the random variable obtained from $W$ by conditioning on the 
event that $Z_1 = 0$.
Observe that
\[ \widehat{W} = e^{-\lambda_1\delta_1}\,\prod_{k=2}^\infty \left(1 + \delta_k\right)^{Z_k} e^{-\lambda_k \delta_k}.
\]
For any continuity set   $\mathcal{E}\subset \mathbb{R}$ of $\widehat{W}$,  
we have
 \begin{equation}
\label{no-Ham}
 \Pr\left(\frac{\widehat{Y}_n}{\E Y_n}\in \mathcal{E}\right) = \frac{\Pr\big((\frac{Y_n}{\E Y_n}\in \mathcal{E})\wedge (X_{1,n}=0)\big)}
             {\Pr(X_{1,n}=0)}
              = 
             \frac{\Pr\big( (\frac{Y_n}{\E Y_n}, X_{1,n}) \in \mathcal{E}'\big)}
             {\Pr(X_{1,n}=0)},
\end{equation}
where $\mathcal{E}' = \mathcal{E}\times (-\nfrac{1}{2},\nfrac{1}{2})$. 
Note that  $\Pr(X_{1,n}=0)$ tends to $e^{-\lambda_1}>0$ and that
$\mathcal{E}'$  is a continuity set for the random vector $(W, Z_1)$.
By Theorem~\ref{thm:janson}, the convergence of (\ref{Wdef}) 
and the convergence of (A1) holds jointly. Therefore, as $n\rightarrow \infty$
\[
             \frac{\Pr\big( (\frac{Y_n}{\E Y_n}, X_{1,n}) \in \mathcal{E}'\big)}
             {\Pr(X_{1,n}=0)} \longrightarrow  
             \frac{ \Pr\big((W, Z_1) \in \mathcal{E}'\big)}{\Pr(Z_{1}=0)} 
       =  \Pr(\widehat{W}\in \mathcal{E}),
\]	
proving the first statement.

Next, suppose that $\delta_k > -1$ for all $k\geq 1$.
Then a.a.s.\ $Y_n > 0$, by the final statement of Theorem~\ref{thm:janson}.
Applying (\ref{no-Ham}) to  $\mathcal{E} = \{ 0\}$, we find that  
\[ 
\Pr(\widehat{Y}_n=0) \leq \frac{\Pr(Y_n=0)} {\Pr(X_{1,n}=0)} = o(1),
\]
 completing the proof.
\end{proof}

\subsection{Structure of the rest of the paper}\label{s:structure}

We assume that $r,s\geq 3$, since the results for $s=2$ were proved by
Frieze et al.~\cite{fjmrw} and Janson~\cite{janson}.
It remains to investigate the second moment of $Y$ and the interaction of
$Y$ with short cycles.

Section~\ref{s:preliminary} contains some terminology and preliminary results,
and describes a common framework which we will use for the calculations in
the following two sections.  

In Section~\ref{s:shortcycles} we calculate $\E(Y X_k)/\E Y$, 
using a generating function to assist in our calculations.
For each $k\geq 1$ this determines the value of $\delta_k$ such that $\E(Y X_k)/\E Y$ tends to
$\lambda_k(1+\delta_k)$, where $\lambda_k$ is defined in (\ref{lambdak}).
Standard arguments imply that condition~(A2) of Theorem~\ref{thm:janson} also holds. 

The remainder of the paper is devoted to
completing the small subgraph conditioning argument to prove Theorem~\ref{main2}.
In Section~\ref{s:consequences} we calculate $\sum_{k=1}^\infty \lambda_k \delta_k^2$,
proving that assumption (A3) of Theorem~\ref{thm:janson} holds.
Section~\ref{s:variance} contains the analysis of the second moment $\E(Y^2)$.
Here we use Laplace summation to find an asymptotic expression for the second
moment, proving that (A4) of Theorem~\ref{thm:janson} holds. This involves proving 
that a certain 4-variable real function has a unique
maximum in a certain bounded convex domain.  
The proof of Theorem~\ref{main2} is completed in Section~\ref{s:proof-main2}.

There are two appendices: the optimisation argument required for Section~\ref{s:variance}
is performed in Appendix~\ref{s:maximum}, and the deferred proof of (\ref{omit})
is presented in Appendix~\ref{SomeRoutineUsuallyIgnored}.

\section{Terminology and common framework}\label{s:preliminary} 

For the small subgraph conditioning method, we need to calculate
the second moment of $Y$ and establish condition (A2) of Theorem~\ref{thm:janson}.
We now describe a common framework which we will use for these
calculations, which will be completed in Sections~\ref{s:shortcycles} 
and~\ref{s:variance}.

Suppose that $F_C$, $F_H$ are
both subpartitions of some partition in $\Omega(n,r,s)$,
such that $G(F_H)$ is a loose Hamilton cycle and
$G(F_C)$ is a loose $k$-cycle. 
In particular, $|F_H|=t = n/(s-1)$ and $|F_C|=k$.  Write $H$ for $G(F_H)$ and
write $C$ for $G(F_C)$.  
We will be particularly interested in two extreme cases, namely,
when $k$ is constant or when $k=t$. (In the latter case, $C$ is also a
loose Hamilton cycle.)
In order to describe the common framework we will use for our calculations
in these cases, we need some terminology.  We will use
Figure~\ref{f:Cexternal-cases}
as a running example: it shows a 12-cycle $C$ in a 5-uniform hypergraph,
and some edges of a Hamilton cycle $H$.
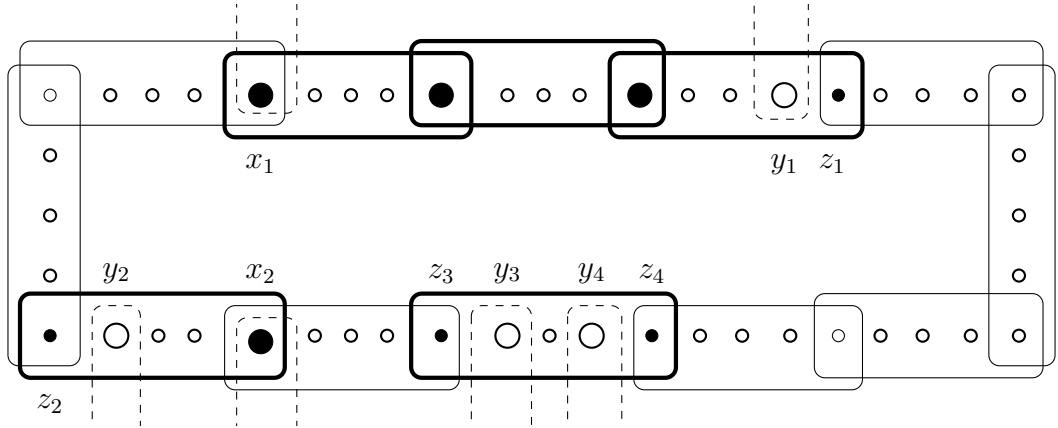
\begin{figure}[ht!]
\begin{center}
\begin{tikzpicture}[scale=0.8]
\begin{scope}
\clip (-1.8, 1.5) rectangle (17.0, 8.5);
\draw [ultra thick, rounded corners] (2.4,6.3) rectangle (6.5,7.7);
\draw [ultra thick, rounded corners] (5.5,6.5) rectangle (9.7,7.9);
\draw [ultra thick, rounded corners] (8.8, 6.3) rectangle (13.0,7.7);
\draw [ultra thick, rounded corners] (5.5,2.3) rectangle (9.9,3.7);
\draw [ultra thick, rounded corners] (-1.0,2.3) rectangle (3.4,3.7);
\draw [dashed, rounded corners] (2.6,6.7) rectangle (3.6,9);
\draw [dashed, rounded corners] (11.2,6.6) rectangle (12.1,9);
\draw [dashed, rounded corners] (0.2,3.5) rectangle (1.0,1.0);
\draw [dashed, rounded corners] (2.6,3.3) rectangle (3.6,1.0);
\draw [dashed, rounded corners] (6.5,3.5) rectangle (7.5,1.0);
\draw [dashed, rounded corners] (8.1,3.5) rectangle (9.0,1.0);
\draw [rounded corners] (-1.0, 6.5) rectangle (3.4,7.9);
\draw [rounded corners] (12.3, 6.5) rectangle (16.0,7.9);
\draw [rounded corners] (2.4, 2.1) rectangle (6.3,3.5);
\draw [rounded corners] (9.2, 2.1) rectangle (13.0,3.5);
\draw [rounded corners] (15.1, 2.5) rectangle (16.2,7.5);
\draw [rounded corners] (-1.2, 2.5) rectangle (0.0,7.5);
\draw [rounded corners] (12.2, 2.3) rectangle (16.0,3.7);
\draw [fill] (3.0,2.9) circle (0.2);  
\node [above] at (3.0,3.7) {$x_2$};
\draw [fill] (3.0,7.0) circle (0.2);  
\node [below] at (3.0,6.2) {$x_1$};
\draw [fill] (6.0,7.0) circle (0.2);  
\draw [fill] (9.3,7.0) circle (0.2);  
\draw [thick] (11.7,7.0) circle (0.2);  
\node [below] at (11.7,6.2) {$y_1$};
\draw [thick] (0.6,3.0) circle (0.2);  
\node [above] at (0.6,3.7) {$y_2$};
\draw [thick] (7.1,3.0) circle (0.2);  
\node [above] at (7.1,3.7) {$y_3$};
\draw [thick] (8.5,3.0) circle (0.2);  
\node [above] at (8.5,3.7) {$y_4$};
\draw [fill] (12.6,7.0) circle (0.1);  
\node [below] at (12.5,6.2) {$z_1$};
\draw [] (-0.5,7.0) circle (0.1);  
\draw [fill] (-0.5,3.0) circle (0.1);  
\node [below] at (-0.5,2.2) {$z_2$};
\draw [fill] (6.0,3.0) circle (0.1);  
\node [above] at (6.0,3.7) {$z_3$};
\draw [fill] (9.5,3.0) circle (0.1);  
\node [above] at (9.5,3.7) {$z_4$};
\draw [] (12.6,3.0) circle (0.1);  
\draw [thick] (0.5,7.0) circle (0.1);  
\draw [thick] (1.2,7.0) circle (0.1);  
\draw [thick] (1.9,7.0) circle (0.1);  
\draw [thick] (3.9,7.0) circle (0.1);  
\draw [thick] (4.5,7.0) circle (0.1);  
\draw [thick] (5.1,7.0) circle (0.1);  
\draw [thick] (7.1,7.0) circle (0.1);  
\draw [thick] (7.7,7.0) circle (0.1);  
\draw [thick] (8.3,7.0) circle (0.1);  
\draw [thick] (10.1,7.0) circle (0.1);  
\draw [thick] (10.8,7.0) circle (0.1);  
\draw [thick] (13.3,7.0) circle (0.1);  
\draw [thick] (14.0,7.0) circle (0.1);  
\draw [thick] (14.8,7.0) circle (0.1);  
\draw [thick] (15.6,7) circle (0.1);  
\draw [thick] (1.3,3.0) circle (0.1);  
\draw [thick] (1.9,3.0) circle (0.1);  
\draw [thick] (3.9,3.0) circle (0.1);  
\draw [thick] (4.5,3.0) circle (0.1);  
\draw [thick] (5.1,3.0) circle (0.1);  
\draw [thick] (7.8,3.0) circle (0.1);  
\draw [thick] (10.3,3.0) circle (0.1);  
\draw [thick] (11.0,3.0) circle (0.1);  
\draw [thick] (11.8,3.0) circle (0.1);  
\draw [thick] (13.3,3.0) circle (0.1);  
\draw [thick] (14.0,3.0) circle (0.1);  
\draw [thick] (14.8,3.0) circle (0.1);  
\draw [thick] (15.6,3) circle (0.1);  
\draw [thick] (-0.5,4.0) circle (0.1);  
\draw [thick] (-0.5,5.0) circle (0.1);  
\draw [thick] (-0.5,6.0) circle (0.1);  
\draw [thick] (15.6,4) circle (0.1);  
\draw [thick] (15.6,5) circle (0.1);  
\draw [thick] (15.6,6) circle (0.1);  
\end{scope}
\end{tikzpicture}
\caption{A 12-cycle $C$ with 5 edges in $G(F_C\cap F_H)$, in three paths.}
\label{f:Cexternal-cases}
\end{center}
\end{figure}

We start by introducing three parameters, $a$, $b$ and $c$, which will be important in our arguments.  
Let $a$  denote the number of parts in $F_C\setminus F_H$.  Then $F_C\cap F_H$ has  $k-a$ parts. 
If $a=0$ or $a=k$ then we set $b=0$.
If $0< a < k$ then we let $b$ denote
the number of connected components of  $G(F_C\cap F_H)$  in $G(H)$. 
Note that  each of these components is a path. 
Finally, we denote by $c$ the number of components of $G(F_C\cap F_H)$ of length at least two.
We will say that $(F_H,F_C)$ \emph{has parameters $(a,b,c)$}.

In our running example from Figure~\ref{f:Cexternal-cases}, 
the edges shown in bold belong to $G(F_C\cap F_H)$.
There are 5 such edges forming 3 paths,  two of which are of length one.
Hence $(F_C,F_H)$ has parameters $(a,b,c)$ where $a=12 - 5 = 7$, $b=3$ and $c=1$. 
The 7 edges of $G(F_C\setminus F_H)$ are shown as thin
rectangles. The dashed lines indicate partial edges 
which belong to $G(F_H\setminus F_C)$.

Let $v$ be a vertex in a loose cycle $C$.  If $v$ has degree 2 in $C$ then
we will say that $v$ is $C$-\emph{external} (or just \emph{external}, if no 
confusion can arise).
Otherwise, $v$ has degree 1 in $C$ and we will say that it is
$C$-\emph{internal} (or just \emph{internal}).
A loose Hamilton cycle $H$ has $t$ external vertices
and $(s-2) t$ internal vertices. 
In Figure~\ref{f:Cexternal-cases}, vertices $x_1$ and $x_2$ (shown as large black circles) 
are $C$-external and $H$-external. Vertices $z_1$, $z_2$, $z_3$ and $z_4$ 
(shown as small black circles) are $C$-external and $H$-internal.
Finally, vertices $y_1$, $y_2$, $y_3$ and $y_4$ (shown as large
white circles) are $H$-external and $C$-internal.
It will be important to know whether a given 
vertex is external or internal in $C$ and/or $H$.

The edges of $G(F_C\cap F_H)$ which start or end a component of $G(F_C\cap F_H)$
will play a special role: we call these \emph{terminal} edges. 
A $C$-\emph{connection vertex} is a $C$-external vertex that belongs to an edge in $G(F_C\cap F_H)$
and an edge of $G(F_C\setminus F_H)$.
If a component of $G(F_C\cap F_H)$ has length at least two
then it has two terminal edges, and each terminal edge contains precisely
one $C$-connection vertex.
In Figure~\ref{f:Cexternal-cases}
there is one component of $G(F_C\cap F_H)$ 
which has more than one edge (and hence has two distinct terminal edges).
The $C$-connection vertices for this component are $x_1$ and $z_1$.
On the other hand, if a component of $G(F_C\cap F_H)$
has length 1 then it has only one terminal edge, containing two
$C$-connection vertices. As mentioned earlier, in Figure~\ref{f:Cexternal-cases} there
are two such components: one has $C$-connection vertices $z_2$ and $x_2$,
and the other has $C$-connection vertices $z_3$ and $z_4$.
We refer to components of $G(F_C\cap F_H)$ of length one as 1-\emph{components}.

As we will see later, the two $C$-connection vertices in components of 
$G(F_C\cap F_H)$ of length at least two
are essentially independent, as far as our counting argument is concerned,
since they belong to \emph{distinct} terminal edges. This is not true 
for the 1-components in $G(F_C\cap F_H)$, so we need to take special care with these,
requiring the introduction of the parameter $c$.

Occasionally we will also need the notion of a $H$-connection vertex:
this is an $H$-external vertex which is incident with exactly one edge of
$G(F_H\setminus F_C)$.  In Figure~\ref{f:Cexternal-cases}, the
$H$-connection vertices are $x_1$, $x_2$, $y_1$, $y_2$, $y_3$, $y_4$.

\medskip

Now we need some definitions for points. 
Every $C$-connection vertex corresponds to two points in $F_C$:  one in $F_C\cap F_H$,
which we call the $C^+$-\emph{connection point}, and one in $F_C\setminus F_H$,
which we call the $C^-$-\emph{connection point}.
We define $H^+$-connection points and $H^-$-connection points
by exchanging the roles of $H$ and $C$ in this definition.

When $b\geq 1$, it will be convenient to consider $F_C$ (or $F_H$) as a 
\emph{sequence} of parts, where 
the edges corresponding to consecutive parts intersect in one vertex, the first edge is terminal and the last edge is not terminal.  Given $F_C$,
these sequences of parts are in 1-1 correspondence with the choice of a $C^+$-connection point, so that the first edge contains this point and the last edge contains the $C^-$-connection point corresponding to the same vertex.  (Similarly for $F_H$.)

Observe that there are $2b$ $C^+$-connection points in $F_C$, and there
are $2b$ $H^+$-connection points in $F_H$.

Given $(F_H,F_C,\Gamma_C)$, where $\Gamma_C$ is a $C^+$-connection point,
we can construct two sequences
$\boldsymbol{\ell}$ and $\boldsymbol{u}$ of \emph{intersection lengths}
and \emph{gap lengths}, defined below.
Let $\boldsymbol{\ell}=(\ell_1,\ldots, \ell_b)$ be the sequence 
of lengths of the components of $G(F_C\cap F_H)$, in the order
determined by the $C^+$-connection point (that is, starting from the
start-vertex of $C$ and in the given direction).
Similarly, let $\boldsymbol{u} = (u_1,\ldots, u_b)$ be the sequence
of lengths of the components of $G(F_C\setminus F_H)$, which also form
$b$ paths, in the order determined by the $C^+$-connection point.
(That is, starting from the start-vertex of $C$ and in the given direction,
measure the gap lengths in order.)
So $u_j$ is the number of edges of $G(F_C\setminus F_H)$ between the $j$'th and $(j+1)$'th 
components
of $G(F_C\cap F_H)$, and $u_b$ is the number of edges between the last and the
first component.
In Figure~\ref{f:Cexternal-cases},
if we take $x_2$ to be the start-vertex and choose
the clockwise direction, then the sequence of intersection lengths is 
$\boldsymbol{\ell}=(1,3,1)$ and the sequence of gap lengths is $\boldsymbol{u}=(2,4,1)$.

\subsection{A common framework}\label{s:framework}

Let $r,s\geq 3$ be fixed integers.
Recall that $X_k$ denotes the number of loose $k$-cycles in $\cF(n,r,s)$,
for $k\geq 2$, and $X_1$ is the number of 1-cycles (parts containing more than
one point from some cell).
We now describe the common framework that we will use when calculating
$\E(Y X_k)$, with $k=O(1)$, and $\E(Y^2) = \E(Y X_t)$. 
These calculations are presented in Sections~\ref{s:shortcycles} and~\ref{s:variance},
respectively. 
We can write
\begin{equation}\label{sum}
\E(Y X_k) = \sum_{(F_H,F_C)} \, \Pr(F_C\cup F_H\subseteq \cF(n,r,s))
\end{equation}
where the sum is over all pairs $(F_H,F_C)$ of subpartitions (not necessarily disjoint)
such that $G(F_H)$ is a loose Hamilton cycle and $G(F_C)$ is a loose $k$-cycle, for $k\geq 2$,
and a 1-cycle for $k=1$.  

To perform the sum in \eqref{sum} over all choices of pairs $(F_H, F_C)$, we first
specialise to sum over all $(F_H,F_C)$  with given parameters $(a,b,c)$. 
Here we consider only triples $(a,b,c)$ of nonnegative integers such that
$0\leq c\leq b\leq a\leq k-1$ and $b\geq 1$.  We call such triples
\emph{valid}. The special cases
where $(a,b) = (0,0)$ or $(a,b)=(k,0)$ will be treated separately.
Note that when $k=1$ we must have $(a,b)=(1,0)$, so 1-cycles fall into the second special
case.

For any given valid triple $(a,b,c)$, we will first sum
over all ways to choose a 4-tuple $(F_H, F_C, \Gamma_H,\Gamma_C)$
such that 
\begin{itemize}
\item $F_H\cup F_C$ is contained in at least one partition in $\Omega(n,r,s)$,
\item $F_H$ corresponds to a loose Hamilton cycle and $F_C$ corresponds
to a loose $k$-cycle,
\item $(F_H,F_C)$ has parameters $(a,b,c)$, 
\item $\Gamma_H$ is a $H^+$-connection point and \item $\Gamma_C$ is a $C^+$-connection point. 
\end{itemize}
Now we fix a valid triple $(a,b,c)$ and show how to construct a 4-tuple
$(F_H,F_C,\Gamma_H,\Gamma_C)$ with parameters $(a,b,c)$, in all possible ways,
using a procedure with five steps.
The case $k=1$ is somewhat special, since $X_1$ counts all 1-cycles,
not just loose 1-cycles. This will be discussed further in Section~\ref{s:shortcycles}.

\begin{enumerate}
\item[] {\bf Step~1: Choose $(F_H,\Gamma_H)$.}\\
We count the number of ways to choose a loose Hamilton cycle $H$, to
choose parts $F_H$ to correspond to the edges of $H$, and
select an external point $\Gamma_H$ for $F_H$.
Once $F_H$ is chosen, there are $2t$ choices for the external point $\Gamma_H$.
Hence, recalling (\ref{p-def}), the number of choices of $(F_H,\Gamma_H)$ is 
\begin{equation}
\label{eq:step1}
\frac{n!}{((s-2)!)^{t}}\, r^n(r-1)^t = \frac{2t\, p(rn)}{p(rn-st)}\, \E Y,
\end{equation}
using (\ref{EY-raw}) and the arguments given in the proof of Lemma~\ref{nr-l-overlap}.
\item[]  {\bf Step~2: Choose the intersection lengths and gap lengths 
     $(\boldsymbol{\ell},\boldsymbol{u})$.}\\  
We must count all possible ways to choose 
$(\boldsymbol{\ell},\boldsymbol{u})$, where
$\boldsymbol{\ell}=(\ell_1,\ldots, \ell_b)$ will be the sequence of intersection
lengths and 
$\boldsymbol{u} =(u_1,\ldots, u_b)$ will be the sequence of gap lengths. 
Denote the overall number of choices of $(\boldsymbol{\ell},\boldsymbol{u})$
for given values of $(a,b,c)$  by $M_2(k,a,b,c)$.
When $k = O(1)$
we calculate $M_2(k,a,b,c)$ in Section~\ref{s:shortcycles}, 
while the case that $k=t$ is analysed in
Section~\ref{s:variance}. 
\item[] {\bf Step~3: Choose a sequence of $(F^{(1)},F^{(2)},\ldots, F^{(b)})$ 
subpaths of $F_H$.}\\
The input to this step is $(F_H,\Gamma_H,\boldsymbol{\ell})$.
We must choose a sequence $(F^{(1)},\ldots, F^{(b)})$ of 
vertex-disjoint connected subpartitions (paths) of $F_H$, 
such that $F^{(j)}$ has length $\ell_j$ for $j\in [b]$, and such that
$\Gamma_H$ is a point corresponding to a vertex in a terminal edge of 
$G(F^{(j)})$, which has degree~1 in $G(F^{(j)})$, for some $j\in [b]$.
(Here ``vertex-disjoint'' means that $G(F^{(1)}),\ldots, G(F^{(b)})$ 
are vertex-disjoint.)
In Lemma~\ref{step3} below, we prove that there are
\begin{equation}
\label{eq:step3} 
\frac{b\, (t-k+a-1)!}{(t-k+a-b)!}
\end{equation}
ways to choose $(F^{(1)},F^{(2)},\ldots, F^{(b)})$. 
These paths will form the components of $F_H\cap F_C$ in order around $C$, as described in Step 5 below.
Note that the expression in (\ref{eq:step3}) does not depend on~$\boldsymbol{\ell}$.
\item[] {\bf Step~4: Choose an ordered pair of points $(\Gamma_{j,1},\Gamma_{j,2})$
for each subpath $F^{(j)}$.}\\
For each $j\in [b]$ we want to select two points, $\Gamma_{j,1}$ and $\Gamma_{j,2}$, which are not used in $F_H$.
These $2b$ points will become the $C^-$-connection points, and define the $C$-connection vertices.
If $F^{(j)}$ has length at least 2 then we choose one vertex from each terminal edge of $G(F^{(j)})$, 
ensuring that the chosen vertex has degree~1 in $G(F^{(j)})$, and order this pair arbitrarily to give
$(v_{j,1},v_{j,2})$. 
If $F^{(j)}$ has length 1 then we choose an ordered pair $(v_{j,1},v_{j,2})$
of vertices from $G(F^{(j)})$. 
Then, for each chosen vertex $v_{j,1}$ (respectively, $v_{j,2}$) we choose a corresponding point $\Gamma_{j,1}$
(respectively, $\Gamma_{j,2}$) which does not belong to any part of $F_H$.

We prove in Lemma~\ref{step4} below that there are
\begin{equation}
\label{eq:step4}
 \left( 2 h(r,s)\right)^b\, \left(\frac{(rs-r-s)^2}{h(r,s)}\right)^c
\end{equation}
ways to  do this, where 
\begin{equation}
h(r,s) = (r-2)^2 + 2(s-2)(r-1)(r-2) + \dfrac{1}{2}(s-2)(s-3)(r-1)^2. 
\label{hdef}
\end{equation}
Observe that $h(r,s) > 0$ whenever $r,s\geq 3$.
\item[] {\bf Step~5:  Complete the specification of $(F_C,\Gamma_C)$. }\\
Let $\Gamma_C$ be the $C^+$-connection point which belongs to the same cell as $\Gamma_{1,1}$. 
Note that $\Gamma_{1,1}$ was chosen in Step~4 and this determines $\Gamma_C$ uniquely.  
(In Figure~\ref{f:Cexternal-cases}, if $\Gamma_{1,1}$ is the point which represents $x_2$ in the part corresponding to the 
edge containing $x_2$ and $z_3$,
then $\Gamma_C$ must be the point which represents $x_2$ in the part corresponding to the edge containing $x_2$ and $y_2$.)
The intersection paths $F^{(1)}$, $F^{(2)},\ldots, F^{(b)}$
will occur around $C$ in this order, with the orientation of each path determined by the choice of
connection vertices. That is, we orient $F^{(j)}$ so that as we move around 
$C$ in the direction determined
by $\Gamma_{C}$, the point
 $\Gamma_{j,2}$ will be joined to $\Gamma_{j+1,1}$ by a path $\widetilde{F}^{(j)}$
in $F_C\setminus F_H$, for $j=1,\ldots, {b-1}$, and 
$\Gamma_{b,2}$ will be joined to $\Gamma_{1,1}$ by a path $\widetilde{F}^{(b)}$ in $F_C\setminus F_H$.

To complete the specification of $F_C$, we must choose points for all the parts in $\widetilde{F}^{(1)},
\ldots, \widetilde{F}^{(b)}$ (other than $\Gamma_{1,1},\ldots, \Gamma_{b,2}$), ensuring that the length of $\widetilde{F}^{(j)}$ is $u_j$ for all $j\in [b]$.
We do this by choosing a sequence of $a-b$ vertices which are not incident with an edge of $G(F_C\cap F_H)$ (these
will be $C$-external vertices), choosing a sequence of $a$ sets
of $s-2$ vertices which are not incident with an edge of $G(F_C\cap F_H)$ (these will be $C$-internal vertices), 
and specifying points for each of these vertices.
Let $M_5(k,a,b)$ denote the number of ways to do this.
It will turn out that this number is independent of $c$.
We calculate $M_5(k,a,b)$ in Section~\ref{s:shortcycles}
when $k=O(1)$, and in Section~\ref{s:variance} when $k=t$.
\end{enumerate}


To complete the calculation of the part of the sum in \eqref{sum} corresponding 
to pairs $(F_H, F_C)$ with parameters $(a,b,c)$, we must
divide by the number of choices of $(\Gamma_H,\Gamma_C)$. 
Since both these points are connection points and are chosen independently,
there are exactly $(2b)^2$ ways to choose $(\Gamma_H,\Gamma_C)$ for
a given $(F_H,F_C)$.  This leads to the following.

\begin{lemma}
\label{5-steps}
Let $(a,b,c)$ be valid parameters. The number of choices of
$(F_H,F_C)$ with parameters $(a,b,c)$ is
\begin{align*}
\frac{1}{2b} \cdot &
\frac{p(rn)\cdot \E Y}{p(rn-st)}\cdot  
     M_2(k,a,b,c)\\
 & {} \times \frac{t(t-k+a-1)!}{(t-k+a-b)!}\cdot
  \big( 2 h(r,s)\big)^b\, \left(\frac{(rs-r-s)^2}{h(r,s)}\right)^c\cdot M_5(k,a,b),
\end{align*}
where $M_2(k,a,b,c)$ denotes the number of ways to perform Step~2 and $M_5(k,a,b)$
denotes the number of ways to perform Step~5.
\end{lemma}

We must also consider the special cases, namely when $(a,b)= (0,0)$ or $(a,b)=(k,0)$.
The first special case can only arise when $C$ is also a loose Hamilton cycle, since 
$a=0$ implies that
$F_C=F_H$.  Once $F_H$ has been specified, we also know $F_C$, so the number of
choices for $(F_H,F_C)$ in this case is
\[ \mathbf{1}_{k=t}\,\,\frac{p(rn)\, \E Y}{p(rn-st)}.\]
The second special case arises when $F_H$ and $F_C$ are disjoint, so $a=k$. Here the number of
choices of $(F_H,F_C)$ can be written as
\begin{align*}
M_0(k)\, \frac{p(rn)\, \E Y}{p(rn-st)}
\end{align*}
where $M_0(k)$ denotes the number of ways to choose a subpartition $F_C$, disjoint
from a given $F_H$, which is a $k$-cycle (and is loose, if $k\geq 2$).
An expression for $M_0(k)$
is calculated in Lemma~\ref{step5} when $k=O(1)$
and in Lemma~\ref{ST5} when $k=t$.

\bigskip

Finally, to obtain an expression for $\E(Y X_k)$ we must multiply each of these
terms by the probability that the subpartition $F_H\cup F_C$ is contained in 
$\cF(n,r,s)$, namely, 
\[
\frac{p(rn-s(t+a))}{p(rn)}, 
\]
then add them all together.
After dividing by $\E Y$, this leads to the resulting expression, using
Lemma~\ref{5-steps}:
\begin{align}
& \frac{\E(Y X_k)}{\E Y}\nonumber \\
 &= \mathbf{1}_{k=t} +
  M_0(k)\cdot \frac{p(rn-s(t+k))}{p(rn-st)}\, \nonum \\
 & \quad {} + 
 \sum_{\substack{0\leq c\leq b\leq a\leq k-1,\\b\geq 1}} 
         \frac{1}{2b} \cdot M_2(k,a,b,c)\cdot \frac{t\,(t-k+a-1)!}{(t-k+a-b)!}\cdot
 \left( 2 h(r,s)\right)^b\cdot \left(\frac{(rs-r-s)^2}{h(r,s)}\right)^c \nonum
  \\ & \hspace*{5cm}  {}\times M_5(k,a,b)\cdot\frac{p(rn-s(t+a))}{p(rn-st)}.
\label{unified}
\end{align}
In particular, when $k=1$ we have
\[ 
\frac{\E(Y X_1)}{\E Y}
 = 
 M_0(1)\cdot \frac{p(rn-s(t+1))}{p(rn-st)}.
\]

To complete this section, we must prove (\ref{eq:step3}) and (\ref{eq:step4}).
First we state the following lemma which contains two useful combinatorial facts 
(the proofs are omitted, as they are standard). 
We adopt the convention that $\binom{0}{0}=1$.

\begin{lemma}
Let $R$, $T$ be positive integers with $R\leq T$,
and let $J$ be a nonnegative integer.
\begin{itemize}
\item[\emph{(i)}]
The number of sequences of $R$ positive integers which sum to $T$ is
\[ \binom{T-1}{R-1}.\]
\item[\emph{(ii)}]
The number of sequences of $R$ positive integers which sum to $T$ and
which contain precisely $J$ entries equal to 1 is
\[ \binom{R}{J}\, \binom{T-R-1}{R-J-1}
\]
if $R<T$, and equals 1 if $R=T$ (in which case also $J=T$).
\end{itemize}
\label{lem:binom}
\end{lemma}

First we calculate the number of ways to perform Step~3 when $b\geq 1$.

\begin{lemma}
Let $r,s\geq 3$ be fixed integers.
Given $k$, valid parameters $(a,b,c)$ and $(F_H,\Gamma_H)$, 
let $\boldsymbol{\ell}=(\ell_1,\ldots, \ell_b)$ be a sequence of 
positive integers which sum to $k-a$.  
Then \emph{(\ref{eq:step3})} equals the number of ways to choose a 
sequence $(F^{(1)},\ldots, F^{(b)})$ of vertex-disjoint induced subpartitions 
(paths) of $F_H$, such that 
\begin{itemize}
\item $F^{(j)}$ has length $\ell_j$ for all $j\in [b]$, 
\item $\Gamma_H$ is a point corresponding to a vertex $w$ which lies
in a terminal edge of $G(F^{(j)})$ for some $j$, and which has 
degree~1 in $G(F^{(j)})$.
\end{itemize}
\label{step3}
\end{lemma}

\begin{proof}
First we choose a permutation $\sigma$ of $[b]$, in $b!$ ways.
This gives us a re-ordering $(\ell_{\sigma(1)},\ell_{\sigma(2)},\ldots, \ell_{\sigma(b)})$
of the entries of $\boldsymbol{\ell}$.  
We will choose vertex-disjoint connected
subpartitions (paths) $(F^{({\sigma(1))}}, F^{({\sigma(2))}},\ldots, F^{({\sigma(b))}})$
of $F_H$ of lengths $\ell_{\sigma(1)}$,
$\ell_{\sigma(2)},\ldots, \ell_{\sigma(b)}$ in this order around $F_H$,
starting from the part of $F_H$ which contains $\Gamma_H$.
(That is, we start from the vertex corresponding to $\Gamma_H$, and set
the direction so that the first edge of $H$ corresponds to the part of $F_H$
which contains $\Gamma_H$.)

To choose $F^{({\sigma(1)})},\ldots, F^{({\sigma(b)})}$
in this order around $F_H$, it suffices to choose a 
sequence of (positive) integers $(g_1,\ldots, g_b)$,
which will be the ``gap lengths'' around $F_H$, in order. 
That is, the first $\ell_{\sigma(1)}$ parts of $F_H$ (in the chosen
direction, starting from the part containing $\Gamma_H$)  will form 
$F^{({\sigma(1)})}$,
then we skip the next $g_1$ parts of $F_H$, to leave a gap between
$F^{({\sigma(1)})}$ and $F^{(\sigma(2))}$;  then the next
$\ell_{\sigma(2)}$ parts of $F_H$ will form $F^{(\sigma(2))}$, and so on.
Once $F^{({\ell(1)})},\ldots, F^{({\ell(b)})}$ have been chosen, we apply 
$\sigma^{-1}$ to produce the desired sequence $(F^{(1)},\ldots, F^{(b)})$.

The positive integers $(g_1,\ldots, g_b)$ must add up to $t-k+a$, as the number
of parts in $\cup_{j=1}^b F^{(j)}$ must be $k-a$.
By Lemma~\ref{lem:binom}(i),
there are $\binom{t-k+ a-1}{b-1}$ ways to choose the sequence $(g_1,\ldots, g_b)$,
which determines $(F^{(1)},\ldots, F^{(b)})$ as described above.
Multiplying these factors together gives
\[ b!\, \binom{t-k+a-1}{b-1} = \frac{b\, (t-k+a-1)!}{(t-k+a-b)!},
\]
completing the proof.
\end{proof}

Next we calculate the number of ways to perform Step~4 when $b\geq 1$.

\begin{lemma}
Let $r,s\geq 3$ be fixed integers.
Given $k$, valid parameters $(a,b,c)$ and $(F_H,\Gamma_H)$, 
suppose that $(F^{(1)},\ldots, F^{(b)})$ is the output of Step~$3$.
Then \emph{(\ref{eq:step4})} counts the number of ways to choose an
ordered pair of distinct $C^{-}$-connection points $(\Gamma_{j,1},\Gamma_{j,2})$
for all $j\in [b]$, as described in Step~4.
\label{step4}
\end{lemma}

\begin{proof}
For $j\in [b]$ we must choose an ordered pair $(\Gamma_{j,1},\Gamma_{j,2})$
of $C^-$-connection points.
Each $C$-connection vertex $v$ is incident with one edge of 
$G(F^{(1)}\cup \cdots \cup F^{(b)})$, which we denote by $e$.

First suppose that $v$ is a $C$-connection vertex in a terminal edge 
$e$ which belongs to a path $F^{(j)}$ of length $\ell_j\geq 2$.
If $v$ equals the $H$-connection vertex in $e$ then there is 1 choice 
for the vertex, and $r-2$ ways to select an unused point corresponding to this 
vertex: this will be a $C^-$-connection point.
Otherwise, there are $s-2$ $H$-internal vertices which can be chosen for $v$,
and $r-1$ ways to assign an unused point corresponding to this vertex.
Overall, this gives $r-2 + (s-2)(r-1) = rs-r-s$ ways to choose the 
$C$-connection vertex $v$ and an unused point corresponding to $v$.
The choice of $v$ has no effect on the number of
choices for the $C$-connection vertex in the other terminal edge of 
$F^{(j)}$, so we can simply square this contribution to take both
connection vertices into account, giving a contribution of 
$(rs-r-s)^2$ in this case.  We multiply this by two to impose an ordering
on these two $C^{-}$-connection points.

Now suppose that $\ell_j=1$.
The two $C$-connection vertices in $e$ may both be $H$-external, giving 
1 choice for the unordered pair of $C$-connection vertices and 
$(r-2)^2$ ways to assign the corresponding points.
There are $2(s-2)(r-1)(r-2)$ choices if one $C$-connection vertex in $e$ is 
$H$-external and the other is $H$-internal. (For example, see the edge 
containing vertices $z_2$ and $x_2$ in Figure~\ref{f:Cexternal-cases}.) 
Finally, if both $C$-connection vertices in $e$ are
$H$-internal then there are $\nfrac{1}{2}(s-2)(s-3)(r-1)^2$ choices
for the $C$-connection vertices (as an unordered pair) and the corresponding points. 
(See the edge containing vertices $z_3$ and $z_4$ in Figure~\ref{f:Cexternal-cases}.)
So the contribution in the second case is $h(r,s)$, as defined in (\ref{hdef}).
Again, we multiply by 2 to impose an ordering on the two $C^{-}$-connection points.

Overall, the number of ways to select the $2b$ $C$-connection vertices, to 
assign a 
point to each, and to orient each component of $G(F_C\cap F_H)$ within $C$ is
\[ 2^b\, (rs-r-s)^{2c}\, h(r,s)^{b-c} = \left( 2 h(r,s)\right)^b\,
         \left(\frac{(rs-r-s)^2}{h(r,s)}\right)^c,\]
as required.
\end{proof}

To apply (\ref{unified}), it remains to calculate $M_2(k,a,b,c)$ and $M_5(k,a,b)$
and perform the summation, in the two extreme cases, namely when  $k=O(1)$ 
(in Section~\ref{s:shortcycles})
and $k=t$ (in Section~\ref{s:variance}).
Several simplifications make the calculations easier when $k$ is constant,
allowing the use of generating functions to assist us with Steps 2 and 4.
When $k=t$ we use Laplace summation to calculate
the sum over all parameters. 
This will involve detailed analysis of a certain real 
function of four variables.

\section{Effect of short cycles}\label{s:shortcycles}

We use an ordinary generating function to perform Step~2 for short cycles.
(For an introduction to generating functions see for example Wilf~\cite{wilf}.)
As is standard, square brackets are used to denote coefficient extraction: that is,
if $F(x) = \sum_{i=0}^\infty a_i x^i$ is the generating function for a
sequence $(a_i)$ then $[x^j] F(x) = a_j$.

\begin{lemma}
\label{T2}
Suppose that $k\geq 2$ is fixed and let $(a,b,c)$ be a valid triple.  
Then the number of ways to choose $(\boldsymbol{\ell},\boldsymbol{u})$
with parameters $(a,b,c)$ is
\begin{align*} 
M_2(k,a,b,c) =
 [x^k\, y^a\, z^b\, w^c] 
 \left( \frac{x^2yz\big(1-x+xw\big)}{(1-x)(1-xy)}\right)^b. 
\end{align*}
\end{lemma}

\begin{proof}
We will use a generating function to keep track of the number of ways
to construct the sequences $\boldsymbol{\ell}$, $\boldsymbol{u}$,
using the following variables:
\begin{itemize}
\item the power of $x$ equals the sum of all entries of $\boldsymbol{\ell}$
and $\boldsymbol{u}$,
\item the power of $y$ equals the sum of all entries of $\boldsymbol{u}$,
\item the power of $z$ equals the number of entries of $\boldsymbol{\ell}$,
\item the power of $w$ marks the number of entries of 
$\boldsymbol{\ell}$ which are strictly bigger than 1.
\end{itemize}
For example, if $\boldsymbol{\ell}=(1,3,1)$ and $\boldsymbol{u}=(2,4,1)$, as in 
Figure~\ref{f:Cexternal-cases}, then the corresponding term in the generating
function is $x^{12} y^7 z^3 w$. 

First we must specify the first intersection length $\ell_1$.
If $\ell_1=1$ then this is stored in the generating function as $xz$,
as it contributes 1 to the total sum and 1 to the number of entries in 
$\boldsymbol{\ell}$.  Otherwise, $\ell_1 = j+2$ for some $j\geq 0$,
so the contribution is $x^{j+2}zw$. Summing these over $j$ gives
$x^2zw/(1-x)$.  Therefore
the contribution of the first entry of $\ell_1$ to the generating
function is
\[ xz + \frac{x^2zw}{1-x} = \frac{xz(1 -x + xw)}{1-x}.\]
Next we must specify the first gap length $u_1$.  If $u_1=j$ then
this contributes $j$ to the total sum and $j$ to the sum of entries of
$\boldsymbol{u}$.  After summing over $j\geq 1$, this is recorded in
the generating function as 
\[ \frac{xy}{1-xy}.\]
To completely specify $\boldsymbol{\ell}$ and $\boldsymbol{u}$
we simply repeat the above procedure $b$ times in total.
Therefore
\[ M_2(k,a,b,c) = [x^k\, y^a\, z^b\, w^c] 
 \left( \frac{x^2yz\big(1-x+xw\big)}{(1-x)(1-xy)}\right)^b ,
\]
completing the proof.
\end{proof}

Next we perform Step~5.  Recall that during Steps 1--4 we have identified
$(F_H,\Gamma_H)$, the sequences $\boldsymbol{\ell}$ and $\boldsymbol{u}$ of
intersection lengths and gap lengths around $C$,
the subpartitions $(F^{(1)},\ldots, F^{(b)})$ and the ordered pairs of
$C^-$-connection points $(\Gamma_{j,1},\Gamma_{j,2})$ for each $j\in [b]$.
The paths $F^{(1)},\ldots, F^{(b)}$ will occur around $C$ in this order, starting
from the $C^+$-connection point $\Gamma_C$ determined by $\Gamma_{1,1}$, and with the orientation
of $F^{(j)}$ determined by the $C^{-}$-connection points $(\Gamma_{j,1},\Gamma_{j,2})$.
That is, as we move around $C$ starting from $\Gamma_C$, the point $\Gamma_{j,2}$
will be joined by a path in $F_C\setminus F_H$ by $\Gamma_{j+1,1}$ for $j=1,\ldots, b-1$, and the point
$\Gamma_{b,2}$ will be joined by a path in $F_C\setminus F_H$ to $\Gamma_{1,1}=\Gamma_C$. 
In Step~5 we count the number of ways to specify the rest of $F_C\setminus F_H$. 

\begin{lemma}
Let $k$ be a fixed integer. If $1\leq b\leq a\leq k-1$ then 
$M_5(k,a,b)$ is asymptotically equal to
\[
   \left( \frac{(r-2)(rs-r-s-1)(rs-r-s)^{s-2}\, t^{s-1}}{(s-2)!}\right)^{a}\,
   \left( (r-2)(rs-r-s-1)\, t\right)^{-b}.
\]
If $k\geq 1$ then $M_0(k)$ is asymptotically equal to
\[
  \frac{1}{2k} \left( \frac{(r-2)(rs-r-s-1)(rs-r-s)^{s-2}\, t^{s-1}}{(s-2)!}\right)^{k}.
\]
\label{step5}
\end{lemma}

\begin{proof}
First suppose that $b\geq 1$, which implies that $k\geq 2$.
We must identify a sequence of $C$-external vertices which are not incident
with an edge of $G(F^{(1)}\cup \cdots \cup F^{(b)})$ (there are $a-b$ of them), and 
a sequence of $a$ sets of $s-2$ $C$-internal vertices
which are not incident with an edge of $G(F^{(1)}\cup \cdots \cup F^{(b)})$.
This will specify all remaining vertices in $G(F_C)$. 
We call the vertices identified in this step the \emph{new} vertices.
Then for each new vertex, we must identify the appropriate number of points to complete
$F_C\setminus F_H$.  Note that the number of parts between $F^{(j)}$ and 
$F^{({j+1})}$ is $u_j$,
for $j\in [b]$,
where $\boldsymbol{u}=(u_1,\ldots, u_b)$ is the sequence of of gap lengths chosen
in Step~2. (Here $F^{({b+1})}$ is identified with $F^{(1)}$.)

Since $k=O(1)$,
as we move around $C$ identifying new vertices (starting from the start-vertex and in the
direction determined by $\Gamma_C$),  
at any point around $C$ there are always $t-O(1)\sim t$ remaining $H$-external vertices 
to choose from, and there are always $n-t-O(1)\sim (s-2)t$
remaining $H$-internal vertices to choose from.
For a new vertex~$v$, the number of choices for points representing it in parts corresponding to the
edges of $G(F_C\setminus  F_H)$ incident with $v$ is
\[ \begin{cases} (r-2)(r-3) & \text{ if $v$ is $C$-external and $H$-external,}\\
                 (r-1)(r-2) & \text{ if $v$ is $C$-external and $H$-internal,}\\
          (r-2) & \text{ if $v$ is $C$-internal and $H$-external,}\\
          (r-1) & \text{ if $v$ is $C$-internal and $H$-internal.}
\end{cases}
\]
First we count the number of ways to identify a sequence of new $C$-external vertices, 
in order around $C$, 
and assign points to these vertices.  This can be done by selecting a sequence of
$a-b$ ordered pairs of points, such that both points in a pair
correspond to the same vertex, and these $a-b$ vertices are distinct and
do not belong to $G(F^{(1)}\cup \cdots \cup F^{(b)})$.
For each new $C$-external vertex
there are
\[ (r-2)_2\, t + (r-1)_2\, (s-2)t - O(1) \sim (r-2)(rs-r-s-1)\, t\]
available choices of pairs of points, avoiding points used in pairs of
$F^{(1)}\cup\cdots\cup F^{(b)}$ as well as points that we have just assigned.
Hence the number of ways to identify the sequence of new $C$-external vertices,
and assign points to them, is asymptotically equal to
\begin{align}
\left( (r-2)(rs-r-s-1)\, t\right)^{a-b}. 
\label{fred1}
\end{align}
Next, the number of ways to identify a sequence of $a$ sets of new $C$-internal vertices,
and assign points to these vertices, is asymptotically equal to
\begin{align}
\left(\frac{(rs-r-s)^{s-2}\, t^{s-2}}{(s-2)!}\right)^{a}.\label{fred2}
\end{align}
To see this, choose a sequence of $(s-2)a$ points, with
$nr-ts-O(1) \sim (rs-r-s)t$ choices
for each, and then divide by $((s-2)!)^a$ since the order of the new $C$-internal
vertices within each edge does not matter.
The expression for $M_5(k,a,b)$ follows by multiplying (\ref{fred1}) and 
(\ref{fred2}).

Now we turn to the second statement of the lemma.
When $k\geq 2$, the argument above determines a sequence of parts of $F_C$, disjoint from $F_H$,
with respect to some given start-vertex and direction. 
Dividing by $2k$ forgets the choice of start-vertex and direction, establishing the equation
for $M_0(k)$.

Finally, suppose that $k=1$. This case is slightly different from other values of $k$, 
since the random variable $X_1$ counts all 1-cycles, not just loose 1-cycles.
Since loose 1-cycles involve $s-1$ distinct cells,
while non-loose 1-cycles involve at most $s-2$ distinct cells,  
the contribution to $M_0(1)$ from non-loose 1-cycles is $O(n^{s-2})$ and
the contribution to $M_0(1)$ from loose 1-cycles is $\Theta(n^{s-1})$.
Hence when $k=1$, it suffices to only consider loose 1-cycles.
Similar arguments as above show that there are
\[ \nfrac{1}{2} (r-2)(rs-r-s-1)t\]
ways to choose a $C$-external vertex $v$ and a \emph{set} of two unused points for $v$,
while setting $a=1$ in (\ref{fred2}) gives the number of choices for a set of $s-2$
$C$-internal vertices and an unused point for each.  Multiplying these together proves
the expression for $M_0(1)$.
\end{proof}

We now have all the information we need in order to
perform the summation in (\ref{unified}).

\begin{lemma}
Let $r,s\geq 3$ be fixed integers.
For any fixed integer $k\geq 1$, 
\[\frac{\E(Y X_k)}{\E Y} 
  \sim  \frac{((r-1)(s-1))^k}{2k} + \frac{\zeta_1^k}{2k} + \frac{\zeta_2^k}{2k}
 - \frac{1}{2k}
\] 
as $n\to\infty$ along $\mathcal{I}_{(r,s)}$,
where $\zeta_1$, $\zeta_2\in\mathbb{C}$ satisfy
\begin{equation}
\label{sumprod}
 \zeta_1 + \zeta_2 = -\frac{rs^2-s^2-2rs+r+2}{rs-r-s}, \quad
 \zeta_1\zeta_2 = \frac{(s-1)(s-2)(r-1)}{rs-r-s}.
\end{equation}
\label{lem:Xk}
\end{lemma}

\begin{proof}
Fix $k\geq 1$.  
Before applying Lemmas~\ref{T2} and~\ref{step5},
we simplify some factors of (\ref{unified}).
Since $k-a=O(1)$, the factor of (\ref{unified})
from Step~3 equals
\begin{equation}
\label{henry1} \frac{t(t-k+a-1)!}{(t-k+a-b)!} \sim t^b.
\end{equation}
Similarly, the factor of (\ref{unified}) from Step~6 
equals
\begin{equation}
\label{henry2}
   \frac{p(rn-st-sa)}{p(rn-st)}
 \sim \left(\frac{(s-1)!}{(rs-r-s)^{s-1}\, t^{s-1}}\right)^{a}
\end{equation}
using the fact that $(m)_p\sim m^p$ whenever $p$ is bounded and
$m\to\infty$.
Combining (\ref{henry1}) and (\ref{henry2}) with
Lemmas~\ref{T2} and~\ref{step5}, the expression (\ref{unified})
becomes 
\begin{align*} 
\frac{\E(Y X_k)}{\E Y} 
&\sim
\frac{\mu_1^k}{2k} + \sum_{\substack{a,c\geq 0,\\ b\geq 1}}
  [x^k\, y^a\, z^b\, w^c] \,\frac{1}{2b}\,
  \left(\frac{x^2yz\big( 1-x + xw\big)}{(1-x)(1-  xy)}\,
 \right)^b \, \mu_1^a\, \mu_2^b\, \mu_3^c\\ 
 &= \, \frac{\mu_1^k}{2k} \, - 
    \,\frac{1}{2}\, \sum_{a,c\geq 0} [x^k\, y^a\, w^c] 
  \, \ln\left( 1 - 
  \frac{\mu_1\mu_2 \, x^2y\big( 1-x + \mu_3 xw\big)}{(1-x)(1-\mu_1 \, xy)}\,
   \right),
\end{align*} 
where
\begin{equation}
\begin{rcases}
\mu_1 &= \displaystyle{\frac{(s-1)(r-2)(rs-r-s-1)}{rs-r-s}}\qquad\\
\mu_2 &= \displaystyle{\frac{2h(r,s)}{(r-2)(rs-r-s-1)}}\qquad\\
 \mu_3 &= \displaystyle{\frac{(rs-r-s)^2}{h(r,s)}}.
\end{rcases}
\label{mus}
\end{equation}
Observe that $\mu_1,\mu_2,\mu_3$ are well-defined when $r,s\geq 3$.
The summation over $a$ and $c$ can be achieved by setting $y=w=1$, giving
\begin{align*}
  \frac{\E(Y X_k)}{\E Y}
 &\sim  \, \frac{\mu_1^k}{2k} \, - \, \frac{1}{2}\, [x^k]\,
  \ln\left( 1 - 
  \frac{\mu_1\mu_2 \, x^2\big(1-(1-\mu_3)x\big)}{(1-x)(1-\mu_1 \, x)}\,
   \right)\\
 &= -\frac{1}{2} \, [x^k]
  \ln\left(\frac{1-(\mu_1+1)\, x - \mu_1(\mu_2-1)\, x^2 - \mu_1\mu_2(\mu_3-1)\, x^3}{1-x}
  \right) \non\\
 &= -\frac{1}{2}\,   [x^k ]\, 
   \ln \Biggl(
\frac{\left(1- (r-1)(s-1)x \right)\left( 1 + \frac{rs^2-s^2-2rs+r+2}{rs-r-s}x + \frac{(s-1)(s-2)(r-1)}{rs-r-s}x^2 \right)}{1-x}\Biggr)
\end{align*}
using (\ref{mus}) for the final equality.
The quadratic factor inside the logarithm factors as
\[
  1 + \frac{rs^2-s^2-2rs+r+2}{rs-r-s}x + \frac{(s-1)(s-2)(r-1)}{rs-r-s}x^2 
=\left( 1 - \zeta_1 x \right) \left(1-  \zeta_2 x \right)
\]
where the roots $\zeta_1,\zeta_2\in\mathbb{C}$ are defined by (\ref{sumprod}).
Using this factorisation we can write
\begin{align*}
  \frac{\E(Y X_k)}{\E Y} &\sim - \frac{1}{2}\, [x^k]\,
\Big( \ln\big(1 - (r-1)(s-1)x\big) + \ln(1-\zeta_1 x) + \ln(1-\zeta_2 x) - \ln(1-x)\Big)\\
  &= \frac{((r-1)(s-1))^k}{2k} + \frac{\zeta_1^k}{2k} + \frac{\zeta_2^k}{2k}
   - \frac{1}{2k},
\end{align*}
as claimed. 
\end{proof}

The following corollary follows from Lemma~\ref{lemma:janson},
(\ref{lambdak}), (\ref{omit}) and Lemma~\ref{lem:Xk}. 

\begin{cor}
Suppose that $r,s\geq 3$ are fixed integers.
Then condition \emph{(A2)} of
Theorem~\emph{\ref{thm:janson} } holds with $\lambda_k$ given by
\emph{(\ref{lambdak})} and $\delta_k$ defined by
\begin{equation}
\label{deltakdef}
    \delta_k = \frac{\zeta_1^k + \zeta_2^k - 1}{((r-1)(s-1))^k}
\end{equation}
for $k\geq 1$. 
\label{deltak}
\end{cor}

Observe that even though $\zeta_1$, $\zeta_2$ may be complex,
$\delta_k$ is always real.

\subsection{Preparation for small subgraph conditioning}\label{s:consequences}

Before proceeding to the second moment calculations, we establish some results
which we will be needed in order to apply Theorem~\ref{thm:janson}.
Recall the definition of $\zeta_1$, $\zeta_2$ from (\ref{sumprod}), and 
the definition of $\delta_k$ from (\ref{deltakdef}). 

The first result shows that $\delta_k > -1$ for all $k\geq 1$. 
This will be needed in the proof of the threshold result,
Theorem~\ref{main2}.

\begin{lemma}
\label{lem:deltak}
Suppose that $s\geq 3$ and  $r\geq s+1$.
 Then  
 \begin{equation}\label{b:zeta}
 		|\zeta_1|+|\zeta_2| < (r-1)^{1/2}(s-1)^{1/2}
 \end{equation}
 and  $\delta_k > -1$ for all $k\geq 1$.
\end{lemma}

\begin{proof}
For ease of notation, write $A = (r-1)(s-1)$.
First suppose that $\zeta_1$ and $\zeta_2$ are not real.
In this case, $\zeta_1$ and $\zeta_2$ form a complex conjugate pair,  so 
\[
	|\zeta_1|+|\zeta_2|  = 2 (\zeta_1\zeta_2)^{1/2} =  2 \left(\frac{(r-1)(s-1)(s-2)}{rs-r-s}\right)^{1/2}
	< A^{1/2}.
\]	
The last inequality holds since $rs - r -s \geq  s^2 - s  - 1 > 4(s-2)$.

Now assume that $\zeta_1$ and $\zeta_2$ are  real. Then (\ref{sumprod}) implies that $\zeta_1$
and $\zeta_2$ are both negative. Therefore
\begin{align*}
	|\zeta_1|+|\zeta_2| = |\zeta_1+\zeta_2| = 
	\frac{rs^2  -s^2  -2rs +r +2}{rs-r-s}  \leq  s-1 
	 < A^{1/2}.
\end{align*}
Hence \eqref{b:zeta} holds in all cases, which implies that for any $k\geq 1$,
\[
	 |\zeta_1|^k + |\zeta_2|^k \leq 
	(|\zeta_1| + |\zeta_2|)^k < A^{k/2} < A^k -1.
\]	
Rearranging shows that  $A^k + \zeta_1^k + \zeta_2^k - 1 > 0$, which implies that $\delta_k>-1$.
\end{proof}

Next, we show that condition (A3) of Theorem~\ref{thm:janson} holds.

\begin{lemma}
Let $s\geq 3$ and $r \geq s+1$, and recall the 
definitions of $\lambda_k$,  $\delta_k$ from \emph{(\ref{lambdak})} and
\emph{(\ref{deltakdef})}.
Define
\[ Q(r,s) = r^2s^2 - rs^3 - 2r^2s + 3rs^2 + s^3 + r^2 - 6rs + 4r - 4s + 4.
\]
Then $Q(r,s)>0$
and
\begin{equation} \exp\left(\sum_{k\geq 1} \lambda_k \delta_k^2\right)
 = \frac{r\, (rs-r-s)}{(r-2)\, \sqrt{Q(r,s)}},
\label{magic}
\end{equation}
so condition \emph{(A3)} of Theorem~\emph{\ref{thm:janson}}  holds.
\label{lem:A3}
\end{lemma}

\begin{proof}
We again write $A=(r-1)(s-1)$ for ease of notation. 
Using the fact that each summand in the series below is real, as $\zeta_1$ and
$\zeta_2$ are either real or complex conjugates, we have
\begin{align*}
\sum_{k\geq 1} \lambda_k \delta_k^2
  &=  \sum_{k\geq 1} \frac{1}{2k\, A^k}\, \left(\zeta_1^k + \zeta_2^k - 1\right)^2\\
  &=  \sum_{k\geq 1} \frac{1}{2k\, A^k}\, 
   \left(\zeta_1^{2k} + \zeta_2^{2k}  + 2(\zeta_1\zeta_2)^k - 2\zeta_1^k - 2\zeta_2^k + 1 
                                                \right)\\
  &=  -\dfrac{1}{2}\, \ln\left( 1 - \zeta_1^2/A\right) 
             - \dfrac{1}{2}\, \ln\left(1 - \zeta_2^2/A\right)
              - \ln(1 - \zeta_1\zeta_2/A)  \\
  & \hspace*{30mm}  {} + \ln(1 - \zeta_1/A) + \ln(1 - \zeta_2/A)
         - \dfrac{1}{2}\ln(1 - 1/A).
\end{align*}
To see the last line, note that $|\zeta_1| + |\zeta_2| < A^{1/2}$ by (\ref{b:zeta}),
and hence every series in the above summation converges absolutely. Therefore
\begin{equation}
\sum_{k\geq 1} \lambda_k \delta_k^2
  =  \dfrac{1}{2}\ln\left(\frac{A\, (A-\zeta_1)^2\, (A - \zeta_2)^2 }
   { (A - \zeta_1^2)(A - \zeta_2^2)\, (A-1)( A - \zeta_1\zeta_2)^2}\right),\label{RHS}
\end{equation}
from which it follows that
\begin{align*}
 &\exp\left(\sum_{k\geq 1} \, \lambda_k\delta_k^2\right)
  =  \left(\frac{A\, (A-\zeta_1)^2\, (A - \zeta_2)^2 }
   { (A - \zeta_1^2)(A - \zeta_2^2)\, (A-1)( A - \zeta_1\zeta_2)^2}\right)^{1/2}\\
  &\quad =  \left(\frac{A\left(A^4  - 2(\zeta_1 + \zeta_2)A^3 + ((\zeta_1 + \zeta_2)^2 + 2\zeta_1\zeta_2)A^2 - 2\zeta_1\zeta_2(\zeta_1+\zeta_2)A + (\zeta_1\zeta_2)^2\right)}
   {(A - 1)\left( A - \zeta_1\zeta_2)^2 (A^2 - ((\zeta_1 + \zeta_2)^2 - 2\zeta_1\zeta_2)A +(\zeta_1\zeta_2)^2\right)}\right)^{1/2}.
\end{align*}
Substituting for $A$ and for $\zeta_1 + \zeta_2$ and $\zeta_1\zeta_2$ leads to
(\ref{magic}) 
after much simplification, using (\ref{sumprod}). 
The expression $Q(r,s)$ in the square root must be positive, 
as it is a positive multiple of the exponential of real number.
(Alternatively, it can be proved directly that $Q(r,s)>0$ for all $s\geq 3$
and $r\geq s+1$, for example by writing $Q$ as a quadratic in $r$
for fixed $s$.)
\end{proof}

\section{The second moment}\label{s:variance}

In this section we calculate the second moment of $Y$, under the assumptions
that $s\geq 3$ and $r> \rho(s)$.
We use the framework from Section~\ref{s:preliminary}, but write $F_1$ and $F_2$
rather than $F_H$ and $F_C$, respectively, and let $H_j = G(F_j)$ for $j=1,2$.

First we provide an expression for $M_2(t,a,b,c)$, required for Step~2.
Recall that $M_2(t,a,b,c)$ counts the number of ways to choose
$(\boldsymbol{\ell},\boldsymbol{u})$ with parameters $(a,b,c)$, 
where $\boldsymbol{\ell}$ is the sequence of intersection
lengths and $\boldsymbol{u}$ is the gap lengths (around $F_2$).
Recall the definition of a valid triple from Section~\ref{s:framework}.
Here, and throughout the paper, we use the convention that 
for any nonnegative integer $p$ and integer $q$, if $q < 0$ or $q > p$ then
$\binom{p}{q}=0$.

\begin{lemma}
Suppose that $(a,b,c)$ is a valid triple.
Then the number of ways to choose $(\boldsymbol{\ell},\boldsymbol{u})$
with parameters $(a,b,c)$ is
\begin{equation}
 M_2(t,a,b,c) = 
\frac{b\, \xi_t(a,b,c)}{a}\, \binom{a}{b}\, \binom{b}{c}\, \binom{t-a-b}{c},
\label{M2}
\end{equation}
where $\xi_t(a,b,c)$ is defined by
\begin{equation}
\label{xi-def}
 \xi_t(a,b,c) = \begin{cases} \frac{c}{t-a-b} & \text{ if $a+b<t$,}\\
                     1 & \text{ if $a+b\geq t$.}
\end{cases}
\end{equation}
\label{ST2}
\end{lemma}

\begin{proof}
The result is trivially true if $b>t-a$ or if $c > t-a-b$.  
So we may assume that $b\leq t-a$ and $c\leq t-a-b$.

First, suppose that $1\leq b\leq t-a-1$.
By Lemma~\ref{lem:binom}(ii), there are
\[ \binom{b}{c}\,\binom{t-a-b-1}{c-1} = \frac{c}{t-a-b}\, \binom{b}{c}\,
  \binom{t-a-b}{c}\]
ways to select a sequence $\boldsymbol{\ell}=(\ell_1,\ldots, \ell_b)$ of
intersection lengths which add to $t-a$, such that precisely
$b-c$ of these lengths equal 1 and the rest are at least 2.
Then by Lemma~\ref{lem:binom}(i), there are
\[ \binom{a-1}{b-1} = \frac{b}{a}\, \binom{a}{b}\]
ways to choose a sequence $\boldsymbol{u} = (u_1,\ldots, u_b)$ of
gap lengths around $H_2$.
Multiplying these expressions together gives (\ref{M2}).

Next suppose that $b=t-a$. By our assumptions, it follows that
$c=0$.  Furthermore, we have $t=a+b\leq 2a$.
There is one way to choose the vector $\boldsymbol{\ell}$ of intersection
lengths,  and the number of choices for the sequence $\boldsymbol{u}$ of gap lengths
is $\frac{t-a}{a}\binom{a}{t-a}$, as above. This leads to the stated value for 
$M_2(t,a,t-a,t-a)$, using (\ref{xi-def})
and recalling that $\binom{0}{0}=1$.
\end{proof}

Next we turn to Step~5 and calculate the number of ways to complete the specification
of $(F_2,\Gamma_{H_2})$.    We also consider one of the special cases, when $F_1$ and
$F_2$ are disjoint.

\begin{lemma}
Let $a,b,t$ be integers which satisfy $1\leq b\leq a < t$.
Then
\begin{align*}
M_5(t,a,b) &=  
  (a-b)!\, ((s-2)a)!\,
 \left(\frac{(r-1)^{s-2}\, (r-2)^2}{(s-2)!}\right)^a\, 
  (r-2)^{-2b}\,
\nonum \\ & \hspace*{1cm} \times 
 \sum_{d=0}^{a-b} \binom{a-b}{d}\, \binom{(s-2)a}{a-b-d}
  \left(\frac{r-3}{r-2}\right)^d
\end{align*}
and  if $t\geq 2$,
\begin{align*}
M_0(t) &=  
  \frac{1}{2t}\, t!\, ((s-2)t)!\,
 \left(\frac{(r-1)^{s-2}\, (r-2)^2}{(s-2)!}\right)^t\, 
 \sum_{d=0}^{t} \binom{t}{d}\, \binom{(s-2)t}{t-d}
  \left(\frac{r-3}{r-2}\right)^d.
\end{align*}
\label{ST5} 
\end{lemma}

\begin{proof}
First suppose that $b\geq 1$.
In Step~5, we must identify all $H_2$-external vertices in $G(F_2\setminus F_1)$
which are not $H_2$-connection vertices (there are $a-b$ of them), 
and all $H_2$-internal vertices in $G(F_2\setminus F_1)$ (there are $(s-2)a$ of them).
As in Lemma~\ref{step5}, we call all vertices identified in this step \emph{new}. We must also
assign points to all new vertices, thereby completing $F_2\setminus F_1$. 
In Section~\ref{s:shortcycles} we approximated our number of choices at
each step by $t$ or $(s-2)t$, respectively, since we only had to identify
a constant number of new vertices. Here we must count more carefully,
and we will need a new parameter. 

Let $d$ be the number of new $H_2$-external vertices which are also $H_1$-external.
Then there are $a-b-d$ new $H_2$-external vertices which are $H_1$-internal,
and there are $a-b-d$ new $H_2$-internal vertices which are $H_1$-external.
Finally, there are $(s-2)a - (a-b-d) = (s-3)a+b+d$ new $H_2$-internal
vertices which are $H_1$-internal.  We must select identities and points
for all these new vertices.

To do this, first order all new $H_1$-external vertices 
(those not already present in $G(F_1\cap F_2)$) 
and order all new
$H_1$-internal vertices (those not already present in $G(F_1\cap F_2)$), in 
\begin{equation}
(a-b)!\, ((s-2)a)!
\label{5a} 
\end{equation}
ways. 
We will count the number of ways to choose a pair of sequences:
the first is a sequence of the $a-b$ new $H_2$-external vertices,
and the second is a sequence of the $(s-2)a$ new $H_2$-internal
vertices.  As the new $H_1$-external and $H_1$-internal vertices
have both been ordered, we can select the positions for the
$H_2$-external vertices which are $H_1$-external, in  $\binom{a-b}{d}$
ways, and then we can select the positions for the $H_2$-internal vertices
which are $H_1$-external, in $\binom{(s-2)a}{a-b-d}$ ways.  Finally, we
divide by $((s-2)!)^a$ as we need a sequence of $a$ sets of new $H_2$-internal
vertices.  Combining these gives
\begin{equation}
 \binom{a-b}{d}\, \binom{(s-2)a}{a-b-d}\, ((s-2)!)^{-a}
\label{5bandc}
\end{equation}
ways to select identities for these new vertices.

Now we must assign points to these new vertices.
The $d$ new $H_2$-external vertices which are $H_1$-external and the $a-b-d$
new $H_2$-external vertices which are $H_1$-internal must all be assigned
precisely two points, and all new vertices must be assigned precisely
one point.  There are
\begin{align}
 &((r-2)(r-3))^d\, ((r-1)(r-2))^{a-b-d}\, (r-2)^{a-b-d}\, (r-1)^{(s-3)a+b+d} \nonum\\
 &=  \left((r-1)^{s-2} (r-2)^2\right)^a\, (r-2)^{-2b}\, 
       \left( \frac{r-3}{r-2} \right)^d 
\label{5d}
\end{align}
ways to assign points to these new vertices 
(in the parts belonging to $F_2\setminus F_1$).

When $b\geq 1$, the stated expression for $M_5(t,a,b)$ is obtained by
multiplying
together (\ref{5bandc}) and (\ref{5d}), summing the resulting
expression over $d=0,\ldots, a-b$ and finally, multiplying by (\ref{5a}).

Finally, the expression for $M_0(t)$ is obtained by arguing as above and dividing by 
$2t$ in order to forget the choice of $\Gamma_{H_2}$.
\end{proof}

Define
\[
  \mathcal{D} = 
   \{ (a,b,c,d) \in\mathbb{Z}^4 \mid  \quad 0\leq c\leq b,\quad 0\leq d\leq a-b,\quad a + b + c \leq t \}
\]
and
\[ \widehat{\mathcal{D}} = \mathcal{D}\setminus 
   \{ (a,0,0,d)\in\mathcal{D}\mid 1\leq a\leq t-1\}.
\]
The set $\widehat{\mathcal{D}}$ contains all possible 4-tuples
of parameters which can arise in the second moment calculation,
recalling that when $b=0$ we must have $a=0$ or $a=t$, for
combinatorial reasons.  

The next lemma finds a combinatorial expression for $\E(Y^2)/(\E Y)^2$ as
a summation over $\widehat{\mathcal{D}}$, with the summands defined
below.  However, it will prove easier to calculate
the sum over the slightly larger set $\mathcal{D}$.
As we will see, the additional terms will have only negligible effect
on the answer.
Hence we define the summand $J_t(a,b,c,d)$ for all $(a,b,c,d)\in\mathcal{D}$, 
as follows.  First, let
\begin{align}
\begin{rcases}
 \kappa_2 &= \displaystyle{\frac{2 h(r,s)}{(r-2)^2}}\qquad\\
 \kappa_3 &= \displaystyle{\frac{(rs-r-s)^2}{h(r,s)}}\qquad\\
 \kappa_4 &= \displaystyle{\frac{r-3}{r-2}}
\end{rcases}
\label{kappa-def}
\end{align}
where, as defined in (\ref{hdef}),
\[ 
h(r,s) = (r-2)^2 + 2(s-2)(r-1)(r-2) + \dfrac{1}{2}(s-2)(s-3)(r-1)^2. 
\] 
Observe that $\kappa_2,\kappa_3,\kappa_4$ are well-defined whenever $r,s\geq 3$. 

\begin{lemma}
Suppose that $s\geq 3$ and $r > \rho(s)$ are fixed integers.
Then
\[
  \frac{\E(Y^2)}{(\E Y)^2} =  
   \sum_{(a,b,c,d)\in\widehat{\mathcal{D}}} J_t(a,b,c,d)
\]
where the summands are defined as follows:
\begin{itemize}
\item If $a=0$ then $b=c=d=0$ and we define $J_t(0,0,0,0) = \frac{1}{\E Y}$.
\item If $a\geq 1$ then we let
\begin{align*}
J_t(a,b,c,d) &= \frac{\xi_t(a,b,c)\, t}{2a^2}\, \binom{a}{b}\, \binom{b}{c}\, \binom{t-a-b}{c}\, 
  a!\,  ((s-2)a)! \non \\ 
  & \hspace*{10mm} {} \times \binom{a-b}{d}\, \binom{(s-2)a}{a-b-d}\, 
  \left(\frac{(r-1)^{s-2}\, (r-2)^2}{(s-2)!}\right)^{a}\,\kappa_2^b\, \kappa_3^c\, \kappa_4^d\non\\
 & \hspace*{10mm} {} \times \frac{p(rn-s(t+a))}{p(rn-st)}\, \frac{1}{\E Y},
\end{align*}
where
$\xi_t(a,b,c)$ is defined in \emph{(\ref{xi-def})}.
\end{itemize}
\label{combinatorial}
\end{lemma}

\begin{proof}
The set $\widehat{\mathcal{D}}$ contains all values of the
parameters $(a,b,c,d)$
which can arise from the interaction of two loose Hamilton cycles.
After dividing (\ref{unified}) by $\E Y$,  we can write the resulting
expression
as a sum over $\widehat{\mathcal{D}}$, and denote
the summand corresponding to $(a,b,c,d)\in\widehat{\mathcal{D}}$ by $J_t(a,b,c,d)$.  
(Recall that the sum over $d$ arises in the factor $M_5(t,a,b)$, see Lemma~\ref{ST5}.)

When $b\geq 1$, substituting
Lemma~\ref{ST2} and Lemma~\ref{ST5} into (\ref{unified}) and
dividing by $\E Y$ shows that the summand $J_t(a,b,c,d)$ equals
the expression given in above.

When $a=0$ we have $b=c=d=0$, corresponding to the term
\[ \frac{M_5(t,0,0)}{\E Y} = \frac{1}{\E Y}.\]
This equals the definition of $J_t(0,0,0,0)$ given above.
Finally, suppose that $a=t$ and $b=c=0$, which corresponds to
\begin{align*}
 \frac{M_5(t,t,0)}{\E Y}\cdot \frac{p(rn-2st)}{p(rn-st)}
&= \frac{t!\, ((s-2)t)!}{2t}\, \left(\frac{(r-1)^{s-2}(r-2)^2}{(s-2)!}\right)^t\,
       \sum_{d=0}^t \binom{t}{d}\, \binom{(s-2)t}{t-d}\\
  & \qquad \times \left(\frac{r-3}{r-2}\right)^d
  \, \frac{p(rn-2st)}{p(rn-st)}\, \frac{1}{\E Y}.
\end{align*}
This expression equals $\sum_{d=0}^t J_t(t,0,0,d)$,
with $J_t(t,0,0,d)$ as defined above,  noting that
$\xi_t(t,0,0)=1$.
\end{proof}

The summation in Lemma~\ref{combinatorial} will be evaluated using
Laplace summation.  The following lemma is tailored for this purpose: it is a 
restatement of~\cite[Lemma 6.3]{gjr} (using the notation of the current
paper).

\begin{lemma}\label{TA1}
Suppose the following: 
\begin{enumerate}
\item[(i)]
$\mathcal{L}\subset\mathbb{R}^m$ is a lattice with full rank $m$.
\item[(ii)] 
$K\subset\mathbb{R}^m$ is a compact convex set with non-empty interior $K^\circ$.
\item[(iii)] 
$\varphi:K\to\mathbb{R}$ is a continuous function with a unique maximum at 
some interior point $\boldsymbol{x}^*\in K^\circ$.
\item[(iv)] 
$\varphi$ is twice continuously differentiable in a neighbourhood of 
$\boldsymbol{x}^*$ and the Hessian
$H^*:=D^2\varphi(\boldsymbol{x}^*)$ is strictly negative definite.
\item[(v)] 
$\psi:K^*\to\mathbb{R}$ is a continuous function on some neighbourhood 
$K^*\subseteq K$ of $\boldsymbol{x}^*$ with $\psi(\boldsymbol{x}^*)>0$.
\item[(vi)] 
For each positive integer $t$ there is a vector $\boldsymbol{w}_t \in \mathbb{R}^m$.
\item[(vii)]
For each positive integer $t$ there is 
a function \mbox{$J_t: (\mathcal{L}+\boldsymbol{w}_t)\cap tK \to\mathbb{R}$}
and 
a real number $b_t>0$ 
such that, as $t\to\infty$,
\begin{align}
  J_t(\boldsymbol{v})&= O\left(b_t \,e^{t\varphi(\boldsymbol{v}/t)+o(t)}\right),
 && \boldsymbol{v}\in (\mathcal{L} + \boldsymbol{w}_t)\cap tK, \label{t1a} 
\intertext{and}
  J_t(\boldsymbol{v})&= b_t\left(\psi(\boldsymbol{v}/t)+o(1)\right)\, 
            e^{t\varphi(\boldsymbol{v}/t)}, &&
\boldsymbol{v}\in (\mathcal{L}+\boldsymbol{w}_t)\cap  tK^*, \label{t1b}
\end{align}
uniformly for $\boldsymbol{v}$ in the indicated sets.
\end{enumerate}
Then, as $t\to\infty$,
\begin{equation}
 \sum_{\boldsymbol{v}\in (\mathcal{L}+\boldsymbol{w}_t)\cap tK} J_t(\boldsymbol{v} )
\sim \frac{ (2\pi)^{m/2}\, \psi(\boldsymbol{x}^*)} {\det(\mathcal{L}) \, 
   \sqrt{\det\left(-H^*\right)}} \, b_t \, t^{m/2} \, e^{t\varphi(\boldsymbol{x}^*)}.
\label{final-answer}
\end{equation}
\end{lemma}

As remarked in~\cite{gjr}, the result also holds if $t$ tends to
infinity along some infinite subset of the positive integers.

\bigskip

In order to apply Lemma~\ref{TA1}, we need some more notation.
Define the scaled domain
\begin{equation}
\label{Kdef}
K = \{ (\alpha,\beta,\gamma,\delta)\in\mathbb{R}^4 \mid
   0 \leq \gamma \leq \beta,\quad
   0\leq \delta\leq \alpha-\beta,\quad 
    \alpha + \beta + \gamma \leq 1
   \}.
\end{equation}
Observe that $\mathcal{D}$ can be written as the intersection of
$\mathbb{Z}^4$ with $tK$, but it is not possible to write $\widehat{\mathcal{D}}$
in this form.  This is the reason why it is more convenient to work with
$\mathcal{D}$ when performing Laplace summation.

Let
\[
 \kappa_1 = (r-1)^{s-2}\, (r-2)^2\, (s-1)(s-2)^{2(s-2)}
\]
and recall (\ref{kappa-def}). Define the function $\varphi:K\longrightarrow\mathbb{R}$ by
\begin{align}
\varphi(\alpha,\beta,\gamma,\delta)
 &= 
g(1-\alpha-\beta) + 2(s-1) g(\alpha) + \frac{s-1}{s} g(rs-r-s-s\alpha)
   -g(\beta-\gamma)  \nonum\\
 & \qquad {} - g(\delta) - 2g(\gamma) - 2g(\alpha-\beta-\delta)
    -g((s-3)\alpha+\beta+\delta)  \nonum\\
 & \qquad {} - g(1-\alpha-\beta-\gamma)
  + \alpha\ln(\kappa_1) + \beta\ln(\kappa_2) + \gamma\ln(\kappa_3)
   + \delta\ln(\kappa_4)
\label{phidef}
\end{align}
where $g(x) = x\ln x$ for $x>0$, and   $g(0) = 0$. 
We will need the following information about the function $\varphi$.
The proof of the following crucial result is lengthy and technical, so it is deferred
to the appendix.

\begin{lemma}\label{Lemma_max}
Suppose that $s\geq 3$ and $r > \rho(s)$, where
$\rho(s)$ is defined in Theorem~\ref{main2}.
Then $\varphi$ has a unique global maximum over the domain $K$ which occurs at the point
   $\boldsymbol{x}^\ast = (\alpha^\ast,\beta^\ast,\gamma^\ast,\delta^\ast)$ defined by
  \begin{align*}
    \alpha^\ast = \frac{rs-r-s}{r(s-1)}, \qquad
    \beta^\ast = \frac{rs-s-2}{r(r-1)(s-1)},\\
    \gamma^\ast = \frac{2\, (rs-r-s)}{r(r-1)^2(s-1)^2}, \qquad 
    \delta^\ast = \frac{(r-2)(r-3)}{r(r-1)(s-1)}.
  \end{align*}
The maximum value of $\varphi$ equals 
\begin{equation}
   \varphi(\boldsymbol{x}^\ast ) = \ln(r-1) + \ln(s-1)+ \frac{(s-1)(rs-r-s)}{s} \ln\left(\frac{(rs-r-s)^2}{rs-r}\right).
\label{Lmax1}
\end{equation}
Let $Q(r,s)$ be as defined in Lemma~\emph{\ref{lem:A3}},
and denote by $H^\ast$ the Hessian of 
$\varphi(\alpha,\beta,\gamma,\delta)$ evaluated 
at the point $\boldsymbol{x}^\ast$.
Then
$H^\ast$ is strictly negative definite and
\begin{equation}
\det(-H^\ast) =\frac{r^5 (r-1)^5 (r-2) (s-1)^8 \,\, Q(r,s)}{8 (r-3) (rs-r-s)^2\,
    (rs-2r-2s+4)^2 \, h(r,s)}. \label{Lmax2}
\end{equation}
\end{lemma}
   
With this result in hand, we may establish the following asymptotic expression for the
second moment of $Y$.

\begin{lemma}
Suppose that $s\geq 3$ and $r > \rho(s)$ are fixed integers, where $\rho(s)$ is defined in Theorem~\ref{main2}.
Then as $n\to\infty$ along $\mathcal{I}_{(r,s)}$,
\[
\frac{\E(Y^2)}{(\E Y)^2}
\sim \frac{r(rs-r-s)}{(r-2)\sqrt{Q(r,s)}}
\]
where $Q(r,s)$ is defined in Lemma~\emph{\ref{lem:A3}}.
\label{A4holds}
\end{lemma}

\begin{proof}
Firstly, extend the definition of $J_t(a,b,c,d)$ to cover all $(a,b,c,d)\in
\mathcal{D}$, by defining $\xi_t(a,b,c)=1$ if $b=c=0$ and $1\leq a\leq t-1$.
We will apply Lemma~\ref{TA1} to calculate the sum of $J_t(a,b,c,d)$ over the 
domain $\mathcal{D}$.  While doing so, we will observe that the contribution
to the sum from $\mathcal{D}\setminus \widehat{\mathcal{D}}$ is negligible,
which will imply that the sum over the larger domain $\mathcal{D}$ is also 
asymptotically equal to  $\E(Y^2)/(\E Y)$.

The first six conditions of Lemma~\ref{TA1} hold
with the definitions given below.
\begin{enumerate}
\item[(i)] Let $\mathcal{L}=\mathbb{Z}^4$, a lattice with full rank $m=4$ and
with determinant 1.
\item[(ii)] 
The domain $K$ defined in (\ref{Kdef}) is compact, convex, is contained in 
$[0,1]^4$ and has non-empty interior.
Observe that $\mathcal{D} = \mathbb{Z}^4\cap t K$.  
\item[(iii)]
The function $\varphi:K\to \mathbb{R}$ defined before Lemma~\ref{Lemma_max} is 
continuous. Furthermore, $\varphi$ has a unique maximum
at the point $\boldsymbol{x}^*$, by Lemma~\ref{Lemma_max}, and
$\boldsymbol{x}^*$ belongs to the interior of $K$.
\item[(iv)] The function $\varphi$ is infinitely differentiable in the
interior of $K$. 
Let $H^*$ denote the Hessian matrix of $\varphi$ evaluated at $\boldsymbol{x}^*$.
Then $H^\ast$ is strictly negative definite, by Lemma~\ref{Lemma_max}.
\item[(v)]  
Write $\varepsilon  = \big(r(r-1)^2(s-1)^2\big)^{-1}$ and define
\[ K^* = \{ (\alpha,\beta,\gamma,\delta)\in\mathbb{R}^4 \mid
   \varepsilon <  \gamma \leq \beta - \varepsilon,\quad
   \varepsilon <  \delta <  \alpha-\beta - \varepsilon,\quad 
    \alpha + \beta + \gamma \leq 1 - \varepsilon
   \}.
\]
Then $K^*$ is contained in the interior of $K$ and $\boldsymbol{x}^* \in K^\ast$.
Furthermore, the function $\psi:K^*\longrightarrow\mathbb{R}$ defined by
\begin{align*}
\psi(\alpha,&\beta,\gamma,\delta) \\
 &= 
 \frac{1}{(\alpha-\beta-\delta)\,
  \sqrt{(\beta-\gamma)\, \delta\, (1-\alpha-\beta)\, (1-\alpha-\beta-\gamma)\, 
      ((s-3)\alpha + \beta + \delta)}}
\end{align*}
is continuous on $K^*$.  Direct substitution shows that
\begin{equation}
\psi(\boldsymbol{x}^*) = \frac{1}{2(s-2)(rs-2r-2s+4)}\, 
     \sqrt{\frac{r^{7}\, (r-1)^{5}\, (s-1)^{9}}{2(r-2)(r-3)\, h(r,s) }}\label{Lmax3}
\end{equation}
which is certainly positive when $r > \rho(s)$ and $s\geq 3$.
\item[(vi)]
Let $\boldsymbol{w}_t$ be the zero vector for each $t$.
\end{enumerate}
It remains to prove that condition (vii) of Lemma~\ref{TA1} holds.
Asymptotics are as $t\to\infty$ along the set $\{ t\in\mathbb{Z}^+ \, : \, s \text{ divides } r(s-1) t\}$.
Define 
\begin{align}
 b_t &=  \frac{(s-2)}{4\pi^2\,t^2\,\sqrt{s-1}}\, \left(\frac{1}{(r-1)(s-1)}\,
   \left(\frac{rs-r}{(rs-r-s)^2}\right)^{(s-1)(rs-r-s)/s}\right)^{t}
 \label{stuffoutfront}
\end{align}
and introduce the scaled variables
\begin{equation}
\label{scale} \alpha = a/t,\qquad \beta = b/t,\qquad \gamma= c/t,\qquad
  \delta = d/t.
\end{equation}
Observe that (\ref{t1a}) holds when $a=0$, since then
$\varphi(0,0,0,0) = \frac{s-1}{s} g(rs-r-s)$ and 
\[ J_t(0,0,0,0) = \frac{1}{\E Y} =  O(b_t)\, 
           \exp\Big(t\varphi(0,0,0,0) + \dfrac{5}{2}\ln t\Big),
\]
using Corollary~\ref{lem:looseHam}.
When $a\geq 1$,
first rewrite all binomial coefficients in $J_t(a,b,c,d)$ in terms of factorials
(except those in the factor $1/\E Y$), giving
\begin{align*}
 J_t(a,b,c,d)
&= \frac{\xi_t(a,b,c)\, t}{2a^2 }\,
  \left(s(s-1)(r-1)^{s-2}(r-2)^2\right)^{a}\, 
\kappa_2^b
  \kappa_3^c\, \kappa_4^d\, 
\non\\ &\hspace*{1cm} {} \times
  \frac{(a!)^2\, (t-a-b)!\, (((s-2)a)!)^2}{ 
  (b-c)!\, d!\, (c!)^2\, ((a-b-d)!)^2\, (t-a-b-c)!\, ((s-3)a+b+d)!}\non\\ &\hspace*{1cm} {} \times
  \frac{(rn-s(t+a)))!\, (rn/s-t)!}{ 
   (rn/s-(t+a))!\, (rn-st))!} \, \frac{1}{\E Y}.
\end{align*}
Let $x\vee y$ denote $\max(x,y)$  and apply Stirling's formula in the form
\[
\ln(N!)=N\ln N-N+\tfrac12\ln (N\vee1)+\tfrac12\ln{2\pi}+O(1/(N+1)),
\]
valid for all integers $N\geq 0$, to the above expression for $J_t(a,b,c,d)$.
Here we follow the convention that $0\ln 0=0$.
After substituting for $\E Y$ using Corollary~\ref{lem:looseHam}, this gives
\begin{align*}
& J_t(a,b,c,d)\non \\ 
& = \left(1 + O\left(\frac{1}{c+1} + \frac{1}{d+1} + \frac{1}{b-c+1} + \frac{1}{a-b-d+1}
           + \frac{1}{t-a-b-c+1}\right)\right) \nonum\\
 & \quad {} \times
 \frac{(s-2)\,t^{3/2} \, \big(\kappa_1 t^{-(s-1)}\big)^a\, \kappa_2^b\, \kappa_3^c\, \kappa_4^d}{4 \pi^{2}\,(a-b-d)\,
  \sqrt{(s-1)\, (b-c)\, d\, (t-a-b)\, (t-a-b-c)\, ((s-3)a + b + d)}} \label{sqrt}\\
 &\quad {} \times \frac{(t-a-b)^{t-a-b}\, a^{2(s-1)a}\, (rs-r-s-sa/t)^{(s-1)(rs-r-s-sa/t)t/s}}
  {(b-c)^{b-c}\, c^{2c}\, d^d\, (a-b-d)^{2(a-b-d)}\, (t-a-b-c)^{t-a-b-c}
  \, ((s-3)a+b+d)^{(s-3)a+b+d}}  \non\\
  &\quad  {} \times
   \left(\frac{1}{(r-1)(s-1)}\, \left(\frac{rs-r}{(rs-r-s)^2}\right)^{(s-1)(rs-r-s)/s}\right)^t,\nonum
\end{align*}
except that if some factor in the denominator is zero then it 
should be replaced by 1.  (Also interpret $0^0$ as 1.) 
Finally, rewriting this expression in terms of the scaled variables (\ref{scale}) proves
that (\ref{t1a}) and (\ref{t1b}) hold.
Hence condition (vii) of Lemma~\ref{TA1} is satisfied.
Therefore we may apply Lemma~\ref{TA1} to conclude that (\ref{final-answer})
holds:  that is,
\begin{equation}
\label{final-answer2}
 \sum_{(a,b,c,d)\in \mathcal{D}} J_t(a,b,c,d)\sim \frac{4\pi^2\, 
  \psi(\boldsymbol{x}^\ast)}{\sqrt{\operatorname{det}(-H^\ast)}}\, b_t\, t^2\, 
e^{t\varphi(\boldsymbol{x}^\ast)},
\end{equation}
using the fact that the lattice $\mathbb{Z}^4$ has determinant 1.
It follows from Lemma~\ref{TA1} that, up to the $1+o(1)$ relative error term, 
the right hand side of (\ref{final-answer}) depends only on the values of $J_t(\cdot)$ in
$(\mathcal{L} + \boldsymbol{w}_t)\cap tK^\ast = \mathcal{D}\cap tK^\ast$. Therefore, 
since $\mathcal{D}\cap tK^\ast = \widehat{\mathcal{D}}\cap tK^\ast$
we may replace $\mathcal{D}$ by $\widehat{\mathcal{D}}$ in
(\ref{final-answer2}).  By Lemma~\ref{combinatorial}, the 
proof is completed by
substituting (\ref{Lmax1}) -- (\ref{stuffoutfront}) into
(\ref{final-answer2}).
\end{proof}

\section{Proof of the threshold result}\label{s:proof-main2}

Recall that $Y$ is the number of subsets $F_H$ of $\mathcal{F}(n,r,s)$ consisting of $t$ parts such that
$G(F_H)$ is a loose Hamilton cycle.
To prepare for the proof of Theorem~\ref{main2}, we now find the values of $r,s$ for which 
$\E Y$ tends to infinity.  Define
\[ L(r,s) = 
    \ln(r-1) + \ln(s-1) + \frac{(s-1)(rs-r-s)}{s}\, \ln\left(1-\frac{s}{rs-r}\right)
\]
and treat $r$ as a continuous variable.
Then $L(r,s)$ is 
the natural logarithm of the base of the exponential factor in 
$\E Y$, see Corollary~\ref{lem:looseHam}.
If $L(r,s)\leq 0$ then
$\E Y = o(1)$, so a.a.s.\ there are no loose Hamilton cycles.
For example, $L(3,3)=0$, so a.a.s.\
$\cF(n,3,3)$ has no loose Hamilton cycles. 
Similarly, if $L(r,s) > 0$ then $\E Y\to\infty$.

\begin{lemma}
\label{rho-val}
For any fixed integer $s\geq 2$,  there exists a unique real number
$\rho(s) > 2$ satisfying the lower and upper bounds given in \emph{(\ref{trap})}
such that $L(\rho(s),s)=0$,
\[ L(r,s) < 0\,\,\, \text{ for\,\, $r\in [2,\rho(s))$} \quad 
\text{and} 
 \quad L(r,s)> 0\,\,\, \text{for\,\, $r\in (\rho(s),\infty)$}.
\]
If $s\geq 3$ then $\rho(s) \geq s$.
\label{Llemma}
\end{lemma}

\begin{proof}
The statements hold when $s\in \{2,3\}$, as can be verified directly.
For the remainder of the proof, assume that $s\geq 4$. 
Setting $x=(rs-r-s)/s$, we may rewrite $L(r,s)$ as $L(r,s)= f_s(x)$ where
\[ f_s(x) = \ln(sx+1) - (s-1)x\, \ln\left(1 + \dfrac{1}{x}\right).\]
First we claim that $f_s(x)$ is negative when $x\in \left[1-\frac{2}{s},\,
s-1-\frac{1}{s}\right]$. (This range of $x$ corresponds to $2 \leq r\leq s+1$.)
This can be verified directly for $s=4,\ldots, 13$, while for fixed
$s\geq 14$ we have, for $x$ in this range,
\begin{align*} 
f_s(x) &\leq \ln(sx+1) - (s-1)x\left(\frac{1}{x} - \frac{1}{2x^2}\right)\\
  &\leq \ln(s(s-1)) - (s-1) + \frac{s(s-1)}{2(s-2)} < 0.
\end{align*}
This implies that $L(r,s) < 0$ for all $s\geq 4$ and $2 \leq r\leq s+1$.
In particular, this establishes that $\rho(s) \geq s$ whenever $s\geq 3$.

Next, suppose that $x \geq s-1-\dfrac{1}{s}$. 
Then the derivative of $f_s$ with respect to $x$ satisfies
\begin{align*}
f'_s(x) = \frac{s}{sx+1}  - (s-1)\ln\left(1 + \dfrac{1}{x}\right) + \frac{s-1}{x+1}
 &\geq \frac{s}{sx+1} - \frac{s-1}{x} + \frac{s-1}{x+1}\\
 = \frac{sx(x-s) + 2sx - (s-1)}{(sx+1)x(x+1)}
  &\geq \frac{sx(-1-1/s) + 2sx - (s-1)}{(sx+1)x(x+1)}
 >0.
\end{align*}
This implies that $L(r,s)$ is monotonically increasing as a function of 
$r\geq s+1$, for any fixed $s\geq 4$, as
\[ \frac{\partial}{\partial r} L(r,s) = \frac{(s-1)}{s}\, f'_s(x). 
\]
Furthermore, $f_s(x)$ tends to infinity as $x\to\infty$.
Therefore the function $L(\cdot,s)$ has precisely one root in $(2,\infty)$,
for all $s\geq 4$. Let this root be $r=\rho(s)$.
Next we will prove that 
\begin{equation}
\label{fs-inequalities}
f_s(x_1) < 0 \quad \text{ and } \quad f_s(x_2) > 0
\end{equation}
where
\[
 x_1 = \frac{e^{s-1}}{s} - \frac{s-1}{2}  - \frac{1}{s} -
  \frac{(s^2-s+1)^2}{s e^{s-1}},\qquad
x_2 = \frac{e^{s-1}}{s} - \frac{s-1}{2} - \frac{1}{s}.
\]
Since $\rho^-(s) =s(x_1+1)/(s-1)$ and $\rho^+(s) = s(x_2+1)/(s-1)$, this will prove
that $\rho^-(s) < \rho(s) < \rho^+(s)$, as required.

When $s=4, 5$, we can verify the inequalities (\ref{fs-inequalities}) directly.
Now suppose that $s\geq 6$.
Using the inequality $a\ln(1+1/a) \geq 1 - \frac{1}{2a}$, which holds for all $a > 1$, we have
\begin{align*}
\exp\left( (s-1)x\ln\left(1+\frac{1}{x}\right)\right) &\geq \exp\left(s-1-\frac{s-1}{2x}\right)\\
   &\geq e^{s-1}\, \left(1 - \frac{s-1}{2x}\right).
\end{align*}
Note that $f_s(x) < 0$ if this expression is bounded below by $sx + 1$.
This holds if and only if
\[ 2s x^2 - 2(e^{s-1}-1)x + (s-1)e^{s-1} < 0.\]
Let $x^-$ and $x^+$ denote the smaller and larger root of this quadratic (in $x$), respectively. Then
\begin{align*} 
  x^+ &= \frac{e^{s-1}-1}{2s} + \frac{e^{s-1}}{2s}\sqrt{1 - \frac{2(s^2-s+1)}{e^{s-1}} + \frac{1}{e^{2(s-1)}}}\\
 &> \frac{e^{s-1}-1}{2s} + \frac{e^{s-1}}{2s}\sqrt{1 - \frac{2(s^2-s+1)}{e^{s-1}} }\geq x_1
\end{align*}
using the inequality $(1-a)^{1/2} \geq 1 - a/2 - a^2/2$, which holds for all $a\in (0,1)$.
Also
\[ x^- < \frac{e^{s-1}-1}{2s} < x_1,\]
proving the first statement in (\ref{fs-inequalities}).

For the upper bound, by definition of $x_2$, we have
\[ \ln(s x_2 + 1) = s-1 + \ln\left(1-\frac{s(s-1)}{2e^{s-1}}\right)
        \geq (s-1)\left(1 - \frac{s}{2e^{s-1}} - \frac{s^2(s-1)}{6e^{2(s-1)}}
         \right)
\]
since $\frac{s(s-1)}{2e^{s-1}} < 1/3$ when $s\geq 6$, and
$\ln(1-a) \geq -a -2a^2/3$ when $0 < a < 1/3$.
Next, since $x_2  > 1$ we have
\[ (s-1)x_2\ln\left(1+\dfrac{1}{x_2}\right) \leq (s-1)\left(1-\frac{1}{2x_2}
    + \frac{1}{3x_2^2}\right).
\]
Hence $f(x_2) > 0$ holds if
\begin{equation}
\label{star}
 \frac{1}{3x_2^2} 
< \frac{1}{2x_2} - \frac{s}{2e^{s-1}} - \frac{s^2(s-1)}{6e^{2(s-1)}}.
\end{equation}
Substituting the expression for $x_2$,
the left hand side of (\ref{star}) becomes
\[ \frac{4s^2}{3(2e^{s-1} - (s^2-s+2))^2}
\]
while the right hand side of (\ref{star}) becomes
\begin{align*}
\frac{s(s^2-s+2)}{2e^{s-1}(2e^{s-1}- (s^2-s+2))} - \frac{s^2(s-1)}{6e^{2(s-1)}}
  &> \frac{s(s^2-s+2)}{4e^{2(s-1)}} - \frac{s^2(s-1)}{6 e^{2(s-1)}}\\
   &= \frac{s(s^2-s+6)}{12e^{2(s-1)}}.
\end{align*}
Therefore, it suffices to prove that
\[ 16s^2\, e^{2(s-1)} \leq s(s^2-s+6)\, (2e^{s-1} - (s^2-s + 2))^2.\]
But the right hand side of this expression is bounded below by 
$4\cdot \dfrac{49}{64}\, s(s^2-s+6)\, e^{2(s-1)}$
when $s\geq 6$.
For $s\geq 6$ the inequality 
$256 s \leq 49(s^2-s+6)$ holds, and hence
(\ref{star}) holds. 
Therefore $f_s(x_2) > 0$ for $s\geq 6$,  completing the proof.
\end{proof}

 Now we are ready to complete the proof of our main result, Theorem~\ref{main2},
establishing a threshold result for existence of a loose Hamilton cycle
in $\mathcal{G}(n,r,s)$.

\begin{proof}[Proof of Theorem~\ref{main2}]
When $s=2$, the result follows immediately from Robinson and Wormald~\cite{RW92,RW94},
since $2 < \rho(2) < 3$.
Now suppose that $s\geq 3$.
Lemma~\ref{rho-val} proves that 
there is a unique value of $\rho(s)\geq 3$ such that
\[ (r-1)(s-1)\left(\frac{rs-r-s}{rs-r}\right)^{(s-1)(rs-r-s)/s} = 1.\]
Furthermore, Lemma~\ref{rho-val} proved that the upper and lower
bounds on $\rho(s)$ given in (\ref{trap}) hold.

If $r\geq 2$ is an integer with $r \leq \rho(s)$ then 
a.a.s.\ $\cG(n,r,s)$ contains no loose Hamilton cycle, using 
Corollary~\ref{lem:looseHam} and Lemma~\ref{rho-val},
since $\Pr(Y > 0 )\leq \E Y$.

Now suppose that $r > \rho(s)$ for some fixed $s\geq 3$.
As noted in Section~\ref{s:intro}, Cooper et al.~\cite{cfmr} proved that condition 
(A1) of Theorem~\ref{thm:janson}
holds, with $\lambda_k$ as defined in (\ref{lambdak}).
The remaining conditions of Theorem~\ref{thm:janson} have also been
established:
Corollary~\ref{deltak} proves that condition (A2) holds, Lemma~\ref{lem:A3}
shows that condition (A3) holds, while condition (A4) is shown to hold by
combining Lemma~\ref{lem:A3} and Lemma~\ref{A4holds}. 
Note also that $\delta_k > -1$ for all $k\geq 1$, by Lemma~\ref{lem:deltak}.
By Theorem~\ref{thm:janson}, we have $Y>0$ a.a.s.

To complete the proof, apply Lemma~\ref{translate}(a) with the function $h$ 
given by the number of loose Hamilton cycles.
\end{proof}

Finally, we provide the asymptotic distribution of the number of loose Hamilton 
cycles in $\cG(n,r,s)$, denoted by $Y_{\cG}$.

\begin{thm}
\label{distribution}
Suppose that $r$, $s$ are integers with $s\geq 2$ and $r > \rho(s)$.
Let $\zeta_1,\zeta_2\in\mathbb{C}$ be the constants defined in \emph{(\ref{sumprod})}.
Then   $Y_{\cG}/ \E Y_{\cG}$ converges  in distribution to
\begin{align*}
      \prod_{k=2}^\infty \left(1+ \frac{\zeta_1^k + \zeta_2^k - 1}{((r-1)(s-1))^k}
               \right)^{Z_k}\, 
  \exp\left( \frac{1-\zeta_1^k - \zeta_2^k}{2k}\right)
\end{align*}
where 
the variables $Z_k$ are independent Poisson variables with
\[\E Z_k = \frac{((r-1)(s-1))^k}{2k}\]
 for $k\geq 2$. 
\end{thm}

\begin{proof}
When $s=2$, the above expression matches the distribution given
by Janson~\cite[Theorem 2]{janson}.
To see this, observe that when $s=2$ we have
$\{ \zeta_1,\zeta_2\} = \{ 0,-1\}$.

Now assume that $s\geq 3$. We showed in the proof of Theorem~\ref{main2}
that conditions (A1)--(A4) of Theorem~\ref{thm:janson} hold for
$Y$.  Then Corollary~\ref{W-conditional} implies that
\[ \frac{\widehat{Y}}{\E Y} \stackrel{d}{\longrightarrow } 
e^{-\lambda_1\delta_1}\,
  \prod_{k=2}^\infty \left(1 + \delta_k\right)^{Z_k} e^{-\lambda_k \delta_k}\,\, 
                     \text{ as } \,\,n \to \infty,
\]
where $\widehat{Y}$ is the random variable obtained from $Y$ by conditioning
on the event that there are no 1-cycles.  Applying Lemma~\ref{translate}(b) with the function 
$h$ given by the number of loose Hamilton cycles shows that $\hat{Y}$ and $Y_{\cG}$ have
the same limiting distribution.  
Finally, Corollary~\ref{lem:looseHam} shows that
$\E Y_{\cG}/\E Y\sim e^{-\lambda_1\delta_1}$, and the result follows after substituting
for $\lambda_k$ and $\delta_k$ using (\ref{lambdak}) and (\ref{deltakdef}).
\end{proof}

\begin{remark}
For possible future reference we remark that our arguments prove,
as an intermediate step, that with the same threshold function $\rho(s)$,
when $s\geq 3$ and $r > \rho(s)$
we have the following analogue of Theorem~\ref{distribution}:
\begin{align*}
\frac{Y}{\E Y}\stackrel{d}{\rightarrow}  \prod_{k=1}^\infty \left(1+ \frac{\zeta_1^k + \zeta_2^k - 1}{((r-1)(s-1))^k}
               \right)^{Z_k}\, 
  \exp\left( \frac{1-\zeta_1^k - \zeta_2^k}{2k}\right).
\end{align*}
Furthermore, the analogue of Theorem~\ref{main2} holds.
\end{remark}

\subsection*{Acknowledgements}

The authors would like to express our sincere gratitude to the anonymous referees. Their 
careful reading and attention to detail has led to great improvements in this paper.



\appendix
\section{Search for the global maximum}\label{s:maximum}

We now present the proof of Lemma~\ref{Lemma_max}.  In particular, we assume that $s\geq 3$ and
$r > \rho(s)$.  Recall that $\rho(s)\geq s$ for all $s\geq 3$, as proved directly from the
definition of $\rho(s)$ in Lemma~\ref{rho-val}.

Observe that $K$, defined in (\ref{Kdef}), is a compact, convex set in $[0,1]^4$.
Furthermore, $\varphi$, defined in (\ref{phidef}),  is a continous function on $K$. 
Therefore $\varphi$ attains its maximum value at least once. Moreover, $\varphi$ is  
 infinitely differentiable in the interior of $K$.  
The first-order partial derivatives of $\varphi$ are given by
\begin{align}
\frac{\partial \varphi}{\partial\alpha} &= -\ln(1-\alpha-\beta) + 2(s-1)\ln(\alpha)
 - 2\ln(\alpha-\beta-\delta) + \ln(1-\alpha-\beta-\gamma) \non \\
  & \qquad {} - (s-3)\ln((s-3)\alpha+\beta+\delta) - (s-1)\ln(rs-r-s-s\alpha) 
  + \ln(\kappa_1), \label{a-diff}\\
\frac{\partial \varphi}{\partial\beta} &= -\ln(1-\alpha-\beta) + 
   2\ln(\alpha-\beta-\delta) - \ln(\beta-\gamma) + \ln(1-\alpha-\beta-\gamma)\non\\
  & \qquad {} - \ln((s-3)\alpha + \beta+\delta) + \ln(\kappa_2),\label{b-diff}\\
\frac{\partial\varphi}{\partial\gamma} &=  \ln(\beta-\gamma) - 2\ln(\gamma)
  + \ln(1-\alpha-\beta - \gamma) + \ln(\kappa_3),\label{c-diff}\\
\frac{\partial\varphi}{\partial \delta} &= -\ln(\delta) + 2\ln(\alpha-\beta-\delta)
 - \ln((s-3)\alpha + \beta + \delta) + \ln(\kappa_4).
 \label{d-diff}
\end{align}
For $\tau \geq 1$, let 
$$
 p_\tau =  \frac{(\tau -1)^2 + {\kappa}_3 (\tau -1)}{ (\tau -1)^2 + 2{\kappa}_3 \tau  - {\kappa}_3}, \ \ \ \ 
 q_\tau =  \frac{\kappa_4 \tau(\tau-1)^2}{ \kappa_2( (\tau-1)^2 + 2{\kappa}_3 \tau - {\kappa}_3)}
$$
and define $\boldsymbol{x}_\tau = (\alpha_\tau,\beta_\tau,\gamma_\tau,\delta_\tau)$ by
\begin{align}
\label{alpha-def}
 \alpha_\tau  &= \frac{p_\tau+  q_\tau + \frac{(s-3)}{2\kappa_4} q_\tau
   + \sqrt{\frac{(s-2)}{ \kappa_4} q_\tau(p_\tau + q_\tau)+ \frac{(s-3)^2}{4\kappa_4^2} q_\tau^2 }}
   {1+ p_\tau+  q_\tau + \frac{(s-3)}{2\kappa_4}q_\tau
   + \sqrt{\frac{(s-2)}{ \kappa_4} q_\tau(p_\tau + q_\tau)+ \frac{(s-3)^2}{4\kappa_4^2} q_\tau^2 }}
\\
  \beta_\tau &= p_\tau (1- \alpha_\tau), \qquad 
  \gamma_\tau = \tfrac{{\kappa}_3 (\tau-1)}{ (\tau-1)^2 + 2{\kappa}_3 \tau - 
    {\kappa}_3} (1- \alpha_\tau), \qquad
  \delta_\tau = q_\tau (1-\alpha_\tau).
\label{beta-delta-def}
 \end{align}
When $r\geq s\geq 3$ we have, from (\ref{hdef}),
\begin{align*}
h(r,s) &= (r-2)^2 + 2(s-2)(r-1)(r-2) + \nfrac{1}{2}(s-2)(s-3)(r-1)^2\\ 
       &\leq \left( r-2 + (s-2)(r-1)\right)^2\\
       &= (rs-r-s)^2
\end{align*}
and
\begin{align*}
h(r,s) &= \nfrac{1}{2}(rs-r-s)^2 + \nfrac{1}{2}
  \left( (r-2)^2 + (s-2)(r-1)(r-3)\right) \geq \nfrac{1}{2}(rs-r-s)^2.
\end{align*}
Therefore 
\begin{equation}
\label{kappa3-less-than-2}
1\leq \kappa_3 =\frac{(rs-r-s)^2}{h(r,s)} \leq 2.
\end{equation}
 Since $\kappa_2,\kappa_3,\kappa_4>0$ it follows immediately that  $\alpha_1 = 0$. 
Furthermore,
\begin{align}
   \lim_{\tau\rightarrow \infty} \alpha_\tau &=1, \qquad 
     \lim_{\tau\rightarrow\infty} \beta_\tau = 
     \lim_{\tau\rightarrow\infty} \gamma_\tau = 0, \nonumber \\
     \lim_{\tau\rightarrow\infty} \delta_\tau &= \delta_{\infty}:= \frac{1}{1 + \frac{s-3}{2\kappa_4} + \sqrt{\frac{s-2}{\kappa_4} 
    + \frac{(s-3)^2}{4\kappa_4^2}}} \, \in (0,1).
\label{x-infinity}
\end{align}
Next, since $\kappa_2, \kappa_3, \kappa_4>0$ it follows that
both $p_\tau$ and $q_\tau/p_\tau$ are strictly increasing, infinitely differentiable functions
of $\tau\in [1,\infty)$. 
Hence
\begin{equation}
\label{alpha-increasing}
\text{$\alpha_\tau$ is a
strictly increasing differentiable function of $\tau\in [1,\infty)$.  
}
\end{equation}
Observe also that  
 \begin{equation}\label{Lmax_inside}
  \text{$\boldsymbol{x}_\tau$ lies in the interior of the domain $K$,  for any $\tau\in (1,\infty)$}.
 \end{equation}
Since $\kappa_4 = \frac{r-3}{r-2} <1$, we have
\begin{equation}
    \sqrt{\tfrac{(s-2)}{ \kappa_4}q_\tau(p_\tau + q_\tau)+ \tfrac{(s-3)^2}{4\kappa_4^2}q_\tau^2 } \geq q_\tau + \tfrac{(s-3)}{2\kappa_4}q_\tau, \label{bound_root}
\end{equation}
and hence
\begin{equation}
 	(1-\alpha_\tau)^{-1} \geq 1+p_\tau+ \bigg(2 + \frac{s-3}{\kappa_4}\bigg)\, q_\tau \geq 1+ p_\tau + (s-1)q_\tau. \label{bound_alpha-1}
\end{equation}
We also note that
\begin{equation}
\label{we-also-note}
\tau= \frac{1-\alpha_\tau-\beta_\tau}{1-\alpha_\tau-\beta_\tau-\gamma_\tau}
\end{equation}
for all $\tau\geq 1$.

Suppose that  $\boldsymbol{x}_{\tilde{\tau}}$ is a stationary point of $\varphi$,
for some value of $\tilde{\tau}$.
Then we can solve the equations
$\frac{\partial \varphi}{\partial\alpha}(\boldsymbol{x}_{\tilde{\tau}}) = 
\frac{\partial \varphi}{\partial\beta}(\boldsymbol{x}_{\tilde{\tau}}) = 
\frac{\partial \varphi}{\partial\gamma}(\boldsymbol{x}_{\tilde{\tau}}) = 
\frac{\partial \varphi}{\partial\delta}(\boldsymbol{x}_{\tilde{\tau}}) = 0$
for 
$\ln \kappa_1$, $\ln \kappa_2$, $\ln \kappa_3$, $\ln \kappa_4$, and substitute
these expressions into  \eqref{phidef}, to obtain (after much cancellation),
\begin{equation}
\label{value-at-stat-point}
 \varphi(\boldsymbol{x}_{\tilde{\tau}}) = \ln \tilde{\tau} + \frac{s-1}{s}(rs-r-s) \ln (rs-r-s-s\alpha_{\tilde{\tau}}).
\end{equation}
Next, note that $\alpha_1=\beta_1=\gamma_1=\delta_1=0$, and therefore            
\begin{equation}
\label{value-at-1}
     	\varphi(\boldsymbol{x}_1) = \frac{(s-1)}{s} (rs-r-s) \ln(rs-r-s).
\end{equation}
It follows that
\begin{align}
\varphi(\boldsymbol{x}_{\tilde{\tau}}) - \varphi(\boldsymbol{x}_1) 
  &= \ln \tilde{\tau} + \frac{s-1}{s} (rs-r-s)  \ln\left(1 - \frac{s\tilde{\alpha}}{rs-r-s}\right) \nonum \\
  &\leq \ln \tilde{\tau} - (s-1)\tilde{\alpha}
\label{stat-minus-1}
\end{align}
for any stationary point $\boldsymbol{x}_{\tilde{\tau}}$ of $\varphi$.
This will be useful later.

\bigskip
 	
Recall the point $\boldsymbol{x}^\ast$ from the statement of Lemma~\ref{Lemma_max}.
We claim that that $\boldsymbol{x}^\ast$ belongs to the ridge: specifically, we claim that
$\boldsymbol{x}^\ast = \boldsymbol{x}_{\tau^\ast}$ for $\tau^\ast = (r-1)(s-1)$. 
To establish this, observe (by direct substitution) that
\[ p_{\tau^*} = \frac{rs-s-2}{(r-1)s},\qquad
   q_{\tau^*} = \frac{(r-2)(r-3)}{(r-1)s}
\]
and hence
\begin{align*} 
 p_{\tau*} + q_{\tau*} + \frac{s-3}{2\kappa_4}\, q_{\tau^*} &= \frac{r^2s-r^2-2rs+2r+2s-4}{2(r-1)s},\\
  \sqrt{\frac{(s-2)}{\kappa_4} q_{\tau^*}(p_{\tau^*} + q_{\tau^*}) + 
   \frac{(s-3)^2}{4\kappa_4^2} q_{\tau^*}^2} &= \frac{(r-2)(rs-r-2)}{2(r-1)s}.
\end{align*}
The sum of these two expressions is $(rs-r-s)/s$, from which it follows easily
using (\ref{alpha-def}) that
\[ 1-\alpha_{\tau^*} = \frac{s}{r(s-1)}. \]
This implies that $\alpha_{\tau^*} = \alpha^*$, and using
(\ref{beta-delta-def}), we deduce that
$(\beta_{\tau^*},\gamma_{\tau^*},\delta_{\tau^*})=(\beta^*,\gamma^*,\delta^*)$, completing the
proof of the claim.

Therefore, by~\eqref{Lmax_inside}, the point $\boldsymbol{x}^\ast$ belongs 
to the interior of $K$.

Our interest in this particular curve $\boldsymbol{x}_\tau$ is clarified by the following lemma,
which shows that $\boldsymbol{x}_\tau$ is a parameterisation of a ridge which must contain any
global maximum of $\varphi$ such that $\alpha > 1$.
We prove Lemma~\ref{claima} in Section~\ref{ss:claima}.
 
\begin{lemma}
Suppose that the assumptions of Lemma \ref{Lemma_max} hold. 
For $\tau\geq 1$,  the function $\eta_\tau(\beta,\gamma,\delta)= \varphi(\alpha_\tau,\beta,\gamma,\delta)$
has a  unique global maximum over the domain 
\[ K_\tau= \{ (\beta,\gamma,\delta)\in\mathbb{R}^3 \mid
   0 \leq \gamma \leq \beta,\
   0\leq \delta \leq \alpha_\tau - \beta,\ 
     \beta + \gamma \leq 1 - \alpha_\tau
   \}\]  
which occurs at the point $(\beta_\tau,\gamma_\tau,\delta_\tau)$.  
In particular,  $\frac{\partial}{\partial \beta} \varphi (\boldsymbol{x}_\tau) = 
                       \frac{\partial}{\partial \gamma} \varphi (\boldsymbol{x}_\tau) 
                       = \frac{\partial}{\partial \delta} \varphi (\boldsymbol{x}_\tau) = 0$ for any $\tau>1$.
\label{claima}
\end{lemma}


Let
$\boldsymbol{x}_{\infty} = (1,0,0,\delta_{\infty}) = \lim_{\tau\to\infty} (\alpha_{\tau},\beta_{\tau},
\gamma_{\tau},\delta_{\tau})$, where the value of $\delta_{\infty}$ is given in (\ref{x-infinity}).
It will be convenient to let the parameter $\tau$ range over $[1,\infty]$
to include the limiting point $\boldsymbol{x}_{\infty}$.

\begin{lemma}
For all $\delta\in [0,1]$, if  $\delta\neq \delta_\infty$ then $\varphi((1,0,0,\delta)) \leq \varphi(\boldsymbol{x}_\infty)$.
\label{alpha=1}
\end{lemma}

\begin{proof}
For a contradiction, suppose that there exists $\delta \in [0,1]$ with $\delta\neq \delta_\infty$ and
$\varphi((1,0,0,\delta)) > \varphi(\boldsymbol{x}_\infty)$.  
By continuity of $\varphi$ on $K$, there exists some
$\varepsilon > 0$ such that $\varphi((1-\varepsilon,0,0,\delta))$ is strictly greater than 
$\varphi(\boldsymbol{x}_\tau)$,
where $\tau$ is chosen so that $\alpha_{\tau} = 1-\varepsilon$.  But this contradicts Lemma~\ref{claima}. 
\end{proof}

The last two lemmas lead to the following.

\begin{lemma}
\label{lem:x-infinity}
Suppose that
$\varphi(\boldsymbol{x}^\ast)$ is strictly larger than $\varphi(\boldsymbol{x}_\tau)$
for all $\tau\in [1,\infty]$ with $\tau\neq \tau^*$.
Then the conclusion of Lemma~\ref{Lemma_max} holds.
\end{lemma}

\begin{proof}
Direct substitution shows that
$\frac{\partial}{\partial \alpha} \varphi (\boldsymbol{x}^\ast) = 0$. 
Hence, since $\boldsymbol{x}^\ast = \boldsymbol{x}_{\tau^\ast}$,
it follows from Lemma~\ref{claima} that
$\boldsymbol{x}^\ast$ is a stationary point of $\varphi$. 
Using (\ref{value-at-stat-point}) we may confirm that (\ref{Lmax1}) holds,
and direct substitution (with the help of a computer algebra package such as Maple)
shows that (\ref{Lmax2}) holds. 

Since the Hessian $H^\ast$ is real and symmetric, there exists an orthogonal basis of $\mathbb{R}^4$ 
consisting
of eigenvectors. Combining (\ref{Lmax2}) with Lemma~\ref{lem:A3} shows that $\det(-H^\ast) > 0$.
It follows that $\det(H^\ast) = (-1)^4\, \det(-H^\ast)>0$, so
$H^\ast$ has an even number of positive eigenvalues
(counting multiplicities).  If there are at least two positive eigenvalues (counting
multiplicities) then we can find a vector $\boldsymbol{w}$ which is orthogonal to
$(1,0,0,0)^t$ which is a linear combination of eigenvectors with positive eigenvalues.
By Taylor's Theorem, moving away from $\boldsymbol{x}^*$ in the direction of $\boldsymbol{w}$ would
increase the value of $\varphi$, contradicting Lemma~\ref{claima}. 
Therefore, we conclude that $H^\ast$ is negative definite, which implies that
$\boldsymbol{x}^\ast$ is a local maximum of $\varphi$. 

Recall that $\alpha_1=0$ and $\alpha_\tau$ is a strictly increasing function of $\tau$ 
with $\lim_{\tau\to\infty} \alpha_{\tau} = 1$.
By Lemma~\ref{claima}, any global maximum of $\varphi$ with $\alpha\in [0,1)$ must lie on the ridge.
But Lemma~\ref{claima} does
not rule out the possibility that $\varphi$ has a global maximum at a point in $K$ with $\alpha=1$.
By definition of $K$, any such point must satisfy $\beta = \gamma=0$.
By assumption and using Lemma~\ref{alpha=1}, we see that
$\varphi((1,0,0,\delta))\leq \varphi(\boldsymbol{x}_\infty) < \varphi(\boldsymbol{x}^\ast)$.
Therefore, we conclude that 
$\boldsymbol{x}^\ast$ is the unique global maximum of $\varphi$ on $K$,
completing the proof.
\end{proof}

Hence, it remains to prove that
$\varphi(\boldsymbol{x}^\ast)$ is strictly larger than $\varphi(\boldsymbol{x}_\tau)$
for all $\tau\in [1,\infty]$ with $\tau\neq \tau^*$.
First we consider $\tau=1$. 
Using \eqref{Lmax1} and (\ref{value-at-1}), we find that
\begin{equation}\label{compwith1}
  \varphi(\boldsymbol{x}^\ast) - \varphi(\boldsymbol{x}_1) =   
         \ln(r-1) + \ln(s-1)+ \frac{(s-1)(rs-r-s)}{s} \ln\left(\frac{rs-r-s}{rs-r}\right),
\end{equation}          
and this expression is positive since $r > \rho(s)$ (see Lemma~\ref{Llemma}).
Note that combining Lemma~\ref{claima} with  (\ref{Lmax_inside}) and (\ref{compwith1}) shows that no point on the boundary of $K$ can be a 
global maximum of $\varphi$ on $K$, except possibly points with $\alpha=1$.

Now we may suppose that $\tau > 1$.
Let $\varphi_\alpha$ denote the partial derivative of $\varphi$ with
respect to $\alpha$. 
By Lemma~\ref{claima} we have  $\frac{\partial}{\partial \delta} \varphi (\boldsymbol{x}_\tau) = 0$. 
Combining \eqref{a-diff} and \eqref{d-diff}, we find that
       \begin{align*}
       	 \varphi_\alpha(\boldsymbol{x}_\tau)
&= -\ln(1-\alpha_\tau-\beta_\tau) 
           + 2(s-1)\ln \alpha_\tau - 2\ln(\alpha_\tau-\beta_\tau-\delta_\tau) + \ln(1-\alpha_\tau-\beta_\tau-\gamma_\tau)  \\
  & \qquad {} - (s-3)\ln((s-3)\alpha_\tau+\beta_\tau+\delta_\tau) - (s-1)\ln(rs-r-s-s\alpha_\tau) 
  + \ln{\kappa}_1 \\
     &= -\ln \tau - (s-3) \ln ((s-3)\alpha_\tau+\beta_\tau+\delta_\tau) 
  -\ln \left(\frac{\delta_\tau((s-3)\alpha_\tau+\beta_\tau+\delta_\tau)}{\alpha_\tau^2}\right) 
     \\
        & \qquad {} + 2(s-2)\ln \alpha_\tau - (s-1)\ln(rs-r-s-s\alpha_\tau) + \ln \kappa_4 + \ln \kappa_1.
       \end{align*}           
We will use primes to denote differentiation with respect to $\tau$.

\begin{lemma}
Suppose that the assumptions of Lemma~\emph{\ref{Lemma_max}} hold. 
The functions 
\[ \alpha_\tau+\beta_\tau+\delta_\tau,\qquad
\frac{\delta_\tau((s-3)\alpha_\tau+\beta_\tau +\delta_\tau)}{\alpha_\tau^2} \]
are both strictly increasing with respect to $\tau \in (1,\infty)$.  
\label{claimb}
\end{lemma}

\begin{proof}
Recalling that $q_\tau/p_\tau$ increases with $\tau$, it follows that
$$
  \frac{\delta_\tau}{\alpha_\tau} =  \left(1+ p_\tau/q_\tau + \frac{(s-3)}{2\kappa_4} + \sqrt{ \frac{(s-2)(1+p_\tau/q_\tau)}{ \kappa_4}+ \frac{(s-3)^2}{4\kappa_4^2} }\right)^{-1}	
$$
and 
$$
  \frac{	\delta_\tau((s-3)\alpha_\tau+\beta_\tau +\delta_\tau)}{\alpha_\tau^2} = 
         \kappa_4 \left(1 - \frac{1}{1+
         \frac{(s-3)}{2\kappa_4 (1+p_\tau/q_\tau)} + 
         \sqrt{ \frac{(s-2)}{ \kappa_4 (1+p_\tau/q_\tau)}+ \frac{(s-3)^2 }{4\kappa_4^2(1+p_\tau/q_\tau)^2}  }
          } \right)^2
$$ 
are both increasing functions of $\tau$. Hence, by (\ref{alpha-increasing}),
it follows that $\delta_\tau$ increases. 
Finally, observe that $p_\tau < 1$ for all $\tau> 1$,  so
$$
  \frac{d}{d\tau}(\alpha_\tau + \beta_\tau + \delta_\tau) = \alpha_\tau' - \alpha_\tau' p_\tau + (1-\alpha_\tau) p_\tau'    + \delta_\tau'>0, 
$$
completing the proof. 
\end{proof}
Using Lemma~\ref{claimb} and the fact that $\beta_\tau + \delta_\tau \leq \alpha_\tau\leq 1$, we calculate that 
when $\tau > 1$,
   \begin{align}
   \varphi'_{\alpha} (\boldsymbol{x}_\tau)
   	 &\leq -\frac{1}{\tau} - (s-3) \frac{(s-3)\alpha_\tau'+\beta_\tau'+\delta_\tau'}{(s-3)\alpha_\tau+\beta_\tau+\delta_\tau} + 
   	 2(s-2) \frac{\alpha_\tau'}{\alpha_\tau} + \frac{(s-1)s \alpha_\tau'}{rs-r-s-s\alpha_\tau} 
   	 \nonum \\
   	 &\leq -\frac{1}{\tau} +\left(2(s-2)-\frac{(s-3)(s-4)}{s-2}  + \frac{(s-1)s}{rs-r-2s}\right) \frac{\alpha_\tau'}{\alpha_\tau}.\label{concave}
   \end{align}     
Observe that, by Lemma~\ref{claima}, for
$\tau\in (1,\infty)$
\begin{equation}
\label{connect-derivatives}
 \varphi'(\boldsymbol{x}_\tau) = \varphi_\alpha(\boldsymbol{x}_\tau)\, \alpha'_\tau.
\end{equation}
Using (\ref{alpha-increasing}), it follows that
$\varphi'(\boldsymbol{x}_\tau)$ and $\varphi_\alpha(\boldsymbol{x}_\tau)$ always
have the same sign.
The next lemma will be used to identify some regions in which $\varphi_\alpha(\boldsymbol{x}_{\tau})$ is monotone with respect to $\tau$ and therefore $\varphi'(\boldsymbol{x}_\tau)$
has at most one zero.
We prove Lemma~\ref{final} in Section~\ref{ss:final} below. 

\begin{lemma}
Suppose that $s\geq 3$ and $r > \rho(s)$.
Then $\varphi'_{\alpha}(\boldsymbol{x}_\tau) < 0$ if any of the following
  conditions hold:
\begin{itemize}
\item[\emph{(i)}]  $s=3$ and $\tau \geq 52$,

\item[\emph{(ii)}] $s\geq 4$ and  $\tau \geq 2(s+2)^2$,
\item[\emph{(iii)}] $(r-1)(s-1)\geq (s+1)^3$ and $\tau \geq (s+1)^{2.5}$.
\end{itemize}
\label{final}
\end{lemma}

The rest of the argument is split into two cases.  

\subsubsection*{Case 1: $r$ is sufficiently large}

Assume that $(r-1)(s-1)\geq (s+1)^3$.
Since $\varphi_{\alpha} (\boldsymbol{x}^\ast)=0$, it follows 
from Lemma~\ref{final}(iii) that
\begin{equation}\label{der_diff-alpha}
     \text{
$\varphi_{\alpha} (\boldsymbol{x}_\tau)>0$ for $(s+1)^{2.5} \leq \tau < \tau^\ast$} \ \  \text{ and }  \ \ 
\text{
$\varphi_{\alpha} (\boldsymbol{x}_\tau)<0$ for $\tau>\tau^\ast$.  }
\end{equation}
By
(\ref{connect-derivatives}), the function $\tau \mapsto \varphi'(\boldsymbol{x}_\tau)$
has at most one zero in the range $\tau \geq (s+1)^{2.5}$.  So there can
be at most one local maximum of $\varphi(\boldsymbol{x}_\tau)$ in $[(s+1)^3,\infty]$, and
we know that this local maximum occurs at $\tau =\tau^*$.
Combining this with Lemma~\ref{claima}, Lemma~\ref{alpha=1}, 
(\ref{Lmax_inside}) and \eqref{compwith1}, we conclude in particular
that there are no global maxima of $\varphi$ on the boundary of $K$. 
That is, we may restrict our attention to values of $\tau$ which 
correspond to stationary points of $\varphi$.
     
To complete the proof in this case, it suffices to show that
$\varphi(\boldsymbol{x}_\tau) 
< \varphi(\boldsymbol{x^\ast})$ 
for any $\boldsymbol{x}_\tau\neq \boldsymbol{x}^\ast$ such that 
     $\varphi_{\alpha} (\boldsymbol{x}_\tau)=0$ and 
$\tau\neq \tau^\ast$. 
For a contradiction, suppose that a global maximum is achieved at
a point $\boldsymbol{x}_{\tilde{\tau}}$, 
for some $\tilde{\tau}>1$ with $\tilde{\tau}\neq \tau^*$. 
Then
$\varphi(\boldsymbol{x}_{\tilde{\tau}})\geq 
   \varphi(\boldsymbol{x}^\ast)$
and 
$\varphi_{\alpha}(\boldsymbol{x}_{\tilde{\tau}}) = 0$. 
Furthermore, $1 < \tilde{\tau} < (s+1)^{2.5}$, by \eqref{der_diff-alpha}. 

For ease of notation, let $\tilde{\alpha} = \alpha_{\tilde{\tau}}$.
By the lower bound on $r$ we have
\begin{equation}
  r \geq \frac{(s+1)^3}{s-1} +1 \geq 33 \,\, \text{ and } \,\, 
   \frac{rs-r-s}{rs-r-2s} \leq 1+ \frac{s}{(s+1)^3 - s - 1} \leq 1.05. 
\label{claimc}
\end{equation}
By our assumption on $\boldsymbol{x}_{\tilde{\tau}}$, it follows from (\ref{stat-minus-1}) that
\begin{equation}
0 \leq \ln \tilde{\tau} - (s-1)\tilde{\alpha}, 
\label{first}
\end{equation}
while combining (\ref{Lmax1}) and (\ref{value-at-stat-point}) gives
\begin{align}
0\leq \varphi(\boldsymbol{x}_{\tilde{\tau}}) - \varphi(\boldsymbol{x}^\ast) 
    &= \ln \tilde{\tau}  - \ln \tau^\ast + \frac{s-1}{s} (rs-r-s)  
     	\ln\left(\frac{rs-r-s-s\tilde{\alpha}}{rs-r-s-s\alpha^\ast }\right)   \nonum \\ 
 	&\leq \ln \tilde{\tau} - 3\ln(s+1) +
   \frac{s-1}{s}(rs-r-s )\ln\left(1 + \frac{s(1-\tilde{\alpha})}
   {rs-r-2s }\right)\nonum \\  
 	&\leq \ln \tilde{\tau} - 3\ln(s+1) + (s-1)\frac{rs-r-s }{rs-r-2s }  
          (1-\tilde{\alpha})\nonum \\
  &\leq  \ln \tilde{\tau} - 3\ln(s+1) + 1.05\, (s-1)  (1-\tilde{\alpha}),\label{second}
\end{align}
using (\ref{claimc}) for the final inequality.
Taking a carefully chosen linear combination of the
inequalities (\ref{first}), (\ref{second}), we conclude that
\begin{align} 0 &\leq 
  1.05(1-\tilde{\alpha}) (\varphi(\boldsymbol{x}_{\tilde{\tau}}) - 
             \varphi(\boldsymbol{x}_1)) 
        + \tilde{\alpha} (\varphi(\boldsymbol{x}_{\tilde{\tau}}) - 
                  \varphi(\boldsymbol{x}^*) )\nonum\\
  &\leq 1.05 \ln \tilde{\tau} - 3 \tilde{\alpha} \ln(s+1). \label{later}
\end{align}
Since $\tilde{\tau} < (s+1)^{2.5}$, this implies that
\begin{equation}
  \tilde{\alpha} \leq \frac{1.05\, \ln(\tilde{\tau})}{3\ln(s+1)} <   
            \frac{1.05\times 2.5}{3}   = 0.875.    	      
\label{alpha-bound}
\end{equation}
Now observe that by (\ref{bound_alpha-1}),
\[ 1 - (1-\tilde{\alpha})^{-1} + p_{\tilde{\tau}} + q_{\tilde{\tau}}\leq 0.\]
In Lemma~\ref{lem:halfway}, stated later, we will prove that 
\begin{equation}
\label{halfway} 1 - (1-\tilde{\alpha})^{-1} + p_{\tilde{\tau}} + q_{\tilde{\tau}}
   \geq T(\tilde{\alpha})
\end{equation}
where
$$
 T(\alpha)= - (1-\alpha)^{-1} + 1+ \frac{(50^{ \alpha}-1)^2 + 2 (50^{ \alpha}-1)}{ (50^{ \alpha} -1)^2 + 4( 50^{ \alpha} -1) +2} + R(\alpha)
$$
and
$$
 R(\alpha) = \begin{cases} 0 & \text{ if $\alpha < 0.35$,}\\
        \frac{15}{62}\ 50^{\alpha} \frac{ (50^{ \alpha} -1)^2  }{(50^{ \alpha} -1)^2 + 4( 50^{ \alpha} -1) +2}
        & \text{ otherwise.}
      \end{cases}
      $$
From the plot of the function $\alpha\mapsto T(\alpha)$ given in Figure~\ref{Tplot},
we observe that
$T(\alpha)$ is strictly positive for $0 < \alpha \leq 0.875$ and has a jump at $\alpha = 0.35$.
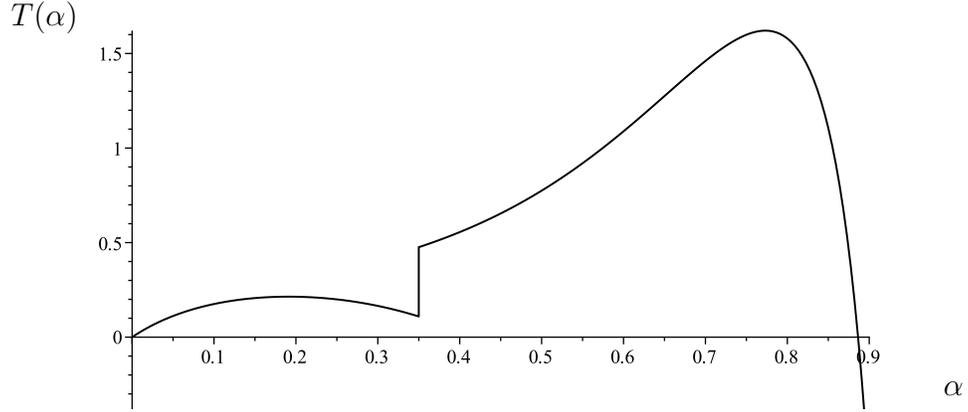
\begin{figure}[ht!]
\begin{center}
\begin{tikzpicture}
\clip (-1,-2) rectangle (12.5,3.5);
\node [right] at (0,0) {\includegraphics[width=0.7\textwidth]{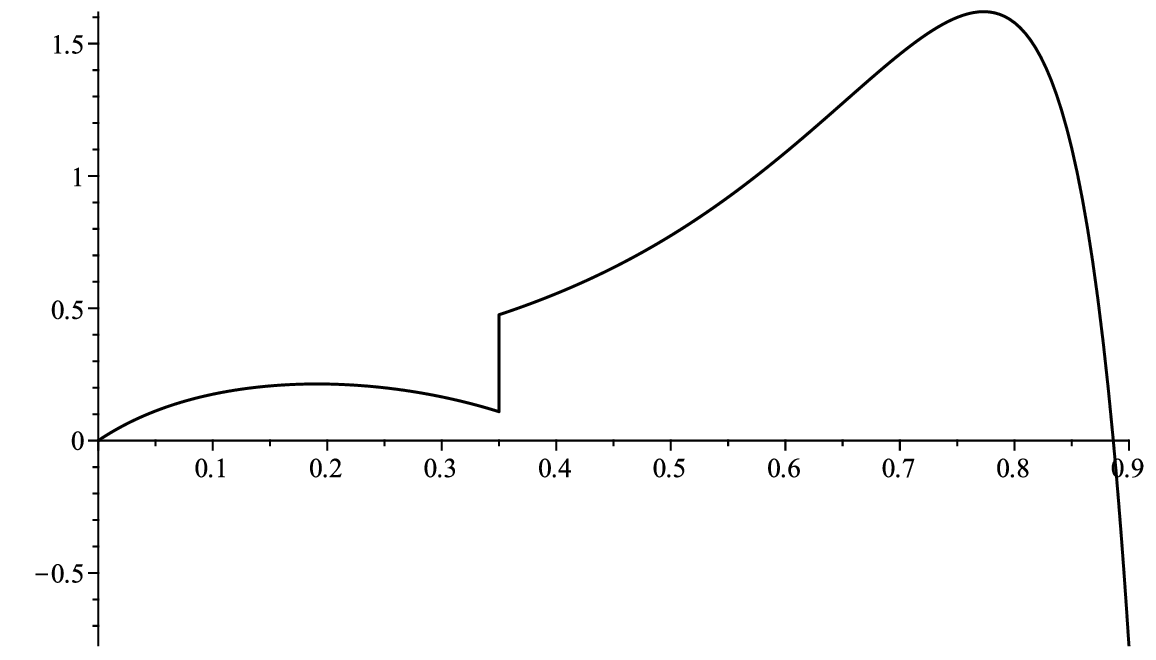}};
\node [below] at (12,-1.5) {$\alpha$};
\node [left] at (0.5,3.2) {$T(\alpha)$};
\end{tikzpicture}
\caption{The function $T(\alpha)$, which is positive for $0< \alpha\leq 0.875$.}
\label{Tplot}
\end{center}
\end{figure}

\noindent Therefore, using (\ref{alpha-bound}), if (\ref{halfway}) holds then
\[ 1 - (1-\tilde{\alpha})^{-1} + p_{\tilde{\tau}} + q_{\tilde{\tau}} > 0.\]
But this contradicts (\ref{bound_alpha-1}). 
Therefore no other global maximum $\tilde{\tau}$ 
of $\varphi(\boldsymbol{x}_\tau)$ can exist
in the interval $[1,(s+1)^{2.5}]$. 

Observe that by (\ref{kappa-def}), we may write
\begin{equation}
  \kappa_2 = s^2-s  + \frac{2(s-2)(s-1)}{r-2} +  \frac{(s-2)(s-3)}{(r-2)^2}.
\label{kappa2-expression}
\end{equation}
We use this to establish (\ref{halfway})
when $(r-1)(s-1) \geq (s+1)^3$,
which will complete the proof of Lemma~\ref{Lemma_max} in this case.

\begin{lemma}
If $s\geq 3$ and $(r-1)(s-1)\geq (s+1)^3$ and \emph{(\ref{later})} holds
then \emph{(\ref{halfway})} holds.
\label{lem:halfway}
\end{lemma}

\begin{proof}
Suppose that $s\geq 3$ and $(r-1)(s-1)\geq (s+1)^3$.  Then
\[
r-2 \geq \frac{(s+1)^3}{s-1} - 1 \geq (s+2)^2,
\]
and it follows from (\ref{kappa2-expression}) that 
\begin{align}
  \kappa_2 & \leq 
  s^2 - s + \frac{2(s-1)(s-2)}{(s+2)^2} + \frac{(s-2)(s-3)}{(s+2)^4} \nonum \\
      &   \leq s^2-1.\label{kappa2-bound}
\end{align}
By (\ref{later}), since $s\geq 3$ we have 
$$
  \tilde{\tau} \geq (s+1)^{\frac{3}{1.05} \tilde{\alpha}} 
   \geq 50^{ \tilde{\alpha}},
$$
while if $\tilde{\alpha} \geq 0.35$ then we can instead write
$$ \tilde{\tau} \geq (s+1)4^{\frac{3}{1.05} \tilde{\alpha}-1} \geq 
            \tfrac{s+1}{4}\, 50^{\tilde{\alpha}}.
$$
Using the fact that $\kappa_3\leq 2$, see (\ref{kappa3-less-than-2}), we can estimate
\[    p_{\tilde{\tau}} \geq 
   \frac{(50^{ \tilde{\alpha}}-1)^2 + 2 (50^{ \tilde{\alpha}}-1)}
  { (50^{ \tilde{\alpha}} -1)^2 + 4( 50^{ \tilde{\alpha}} -1) +2}
\]
and if $\tilde{\alpha}\geq 0.35$ we have
\[ q_{\tilde{\tau}} \geq \frac{\kappa_4 (s+1)  }{4 \kappa_2}   50^{\tilde{\alpha}} 
  \frac{ (50^{ \tilde{\alpha}} -1)^2  }{(50^{ \tilde{\alpha}} -1)^2 + 
    4( 50^{ \tilde{\alpha}} -1) +2 }.\] 
Additionally, it follows from (\ref{claimc}) that
$\kappa_4 = \frac{r-3}{r-2}\geq \frac{30}{31}$.
Using these inequalities, together with the upper bound on $\kappa_2$
given in (\ref{kappa2-bound}), we conclude that
(\ref{halfway}) holds,  as required.
\end{proof}

\subsubsection*{Case 2: $r$ is small}
      
It remains to consider the case that $(r-1)(s-1)< (s+1)^3$.
By definition of $\rho(s)$ (see Table~\ref{rhobounds}),
the only remaining pairs $(r,s)$ belong to the set
\begin{align*}
\mathcal{A} &= \{ (r,3)\mid r=4,\ldots, 32\} \cup
\{ (r,4)\mid r = 6,\ldots, 41 \}\cup  
\{ (r,5)\mid r= 12,\ldots, 54\}\\
 & \hspace*{4.5cm} {} \cup
\{ (r,6)\mid r = 28,\ldots, 69\}\cup 
\{ (r,7)\mid r= 65,\ldots,86\}. 
\end{align*}
There are 172 pairs $(r,s)$ in $\mathcal{A}$.
(It may be possible to reduce this number by refining the above analysis, 
but we have not pursued this.)

We already checked in \eqref{compwith1} 
that  $\varphi(\boldsymbol{x}_1)<\varphi(\boldsymbol{x}^*)$.  
As remarked below \eqref{compwith1}, no point on the boundary of $K$
can be a global maximum except possibly those with $\alpha=1$.
All such points have the form $(1,0,0,\delta)$ for some $\delta$.
By Lemma~\ref{alpha=1} and Lemma~\ref{lem:x-infinity}, it suffices to consider
the stationary points $\boldsymbol{x}_{\tau}$ along the ridge, and the point
$\boldsymbol{x}_\infty$.
Recalling  the definition of $\alpha_\tau$ from~\eqref{alpha-def}, 
since $p_\tau \geq (\tau-1)/(2\tau-1)$ we have
\[
	\alpha_{\tau} > \frac{p_\tau}{1+p_\tau}  \geq 
	\frac{\tau-1}{3 \tau -2}.
\]
It follows from (\ref{stat-minus-1}) that if
$\tau \in (1, (s+1)/3]$ is a stationary point of $\varphi(\boldsymbol{x}_\tau)$, then 
\[
	\varphi(\boldsymbol{x}_\tau) -
	\varphi(\boldsymbol{x}_1) \leq 	\ln \tau -  (s-1)\alpha_{\tau} <
	\tau -1 -  (s-1)  \frac{\tau-1}{3 \tau -2} \leq 0.
\]
Hence we can restrict our attention to the case when $\tau>(s+1)/3$.

For $(r,s)\in\mathcal{A}$ we consider two functions.
The first is 
\[ \tau\mapsto \frac{\tau^{1/2}}{r^{2}}\, \big( \varphi(\boldsymbol{x}^*) - \varphi(\boldsymbol{x}_\tau) \big)\]
on the interval $\tau\in [(s+1)/3,s(s-1)/2]$,
and the second is
\[ \tau\mapsto -\frac{\tau}{r^{1/2}\ln(\tau)}\, \varphi'_{\alpha} 
    (\boldsymbol{x}_\tau)
\]
on the interval $\tau\in [s(s-1)/2,\max\{52, 2(s+2)^2\}]$.
Figures~\ref{s3}--\ref{s7} show the plots of these functions for all
$(r,s)\in\mathcal{A}$, with all pairs with a given value of $s$ 
displayed together.

\begin{figure}[ht!]
\begin{tabular}{cc}
\includegraphics[width=7cm]{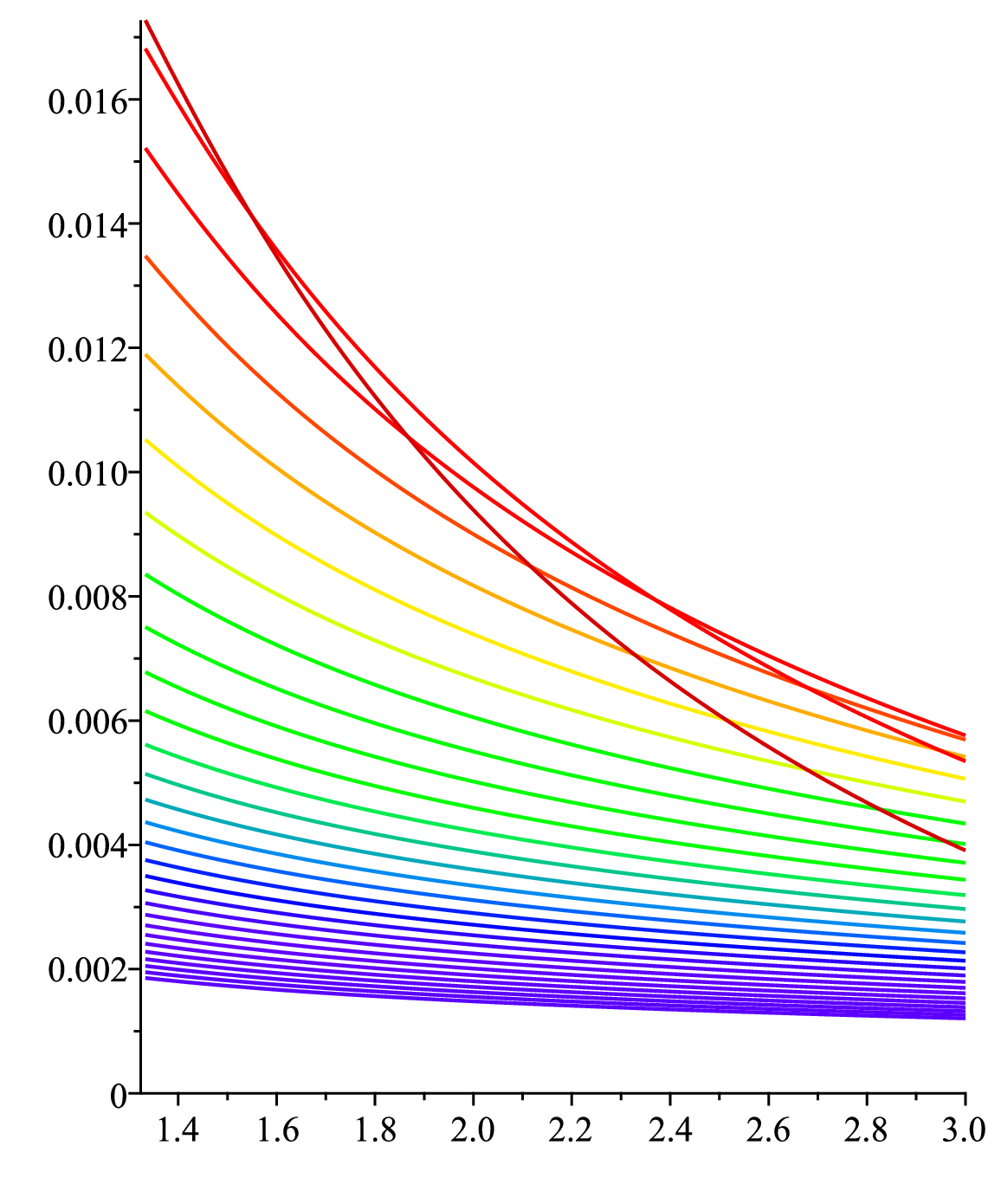} & \includegraphics[width=7cm]{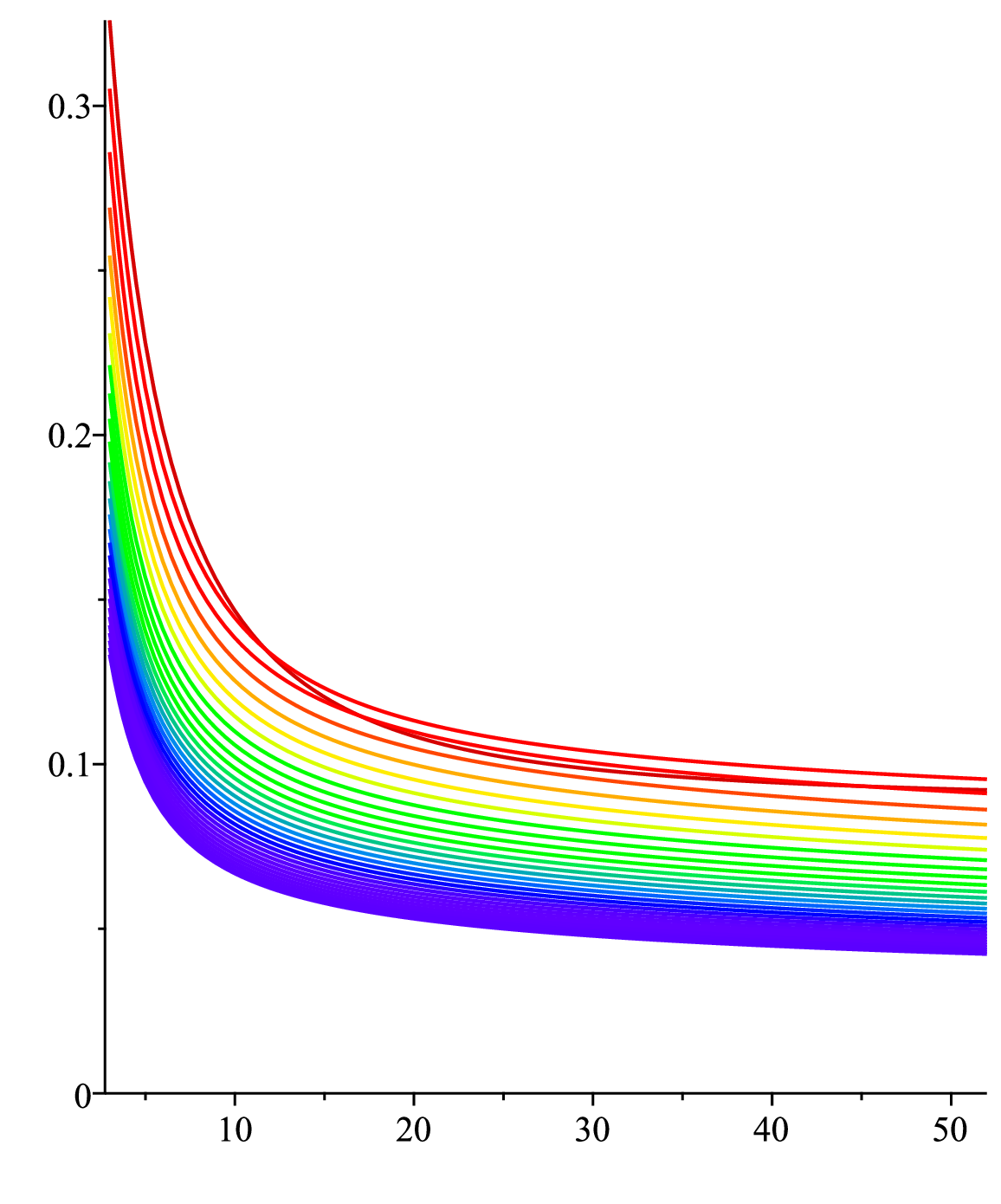}
\end{tabular}
\caption{The two plots for $s=3$, $r=4,\ldots, 32$ (from top to bottom).}
\label{s3}
\end{figure}

In each of these figures, the plot of the first function is shown on the
left, and the plot of the second function is shown on the right.
In each plot, the top lines correspond to the smallest values of $r$ and the bottom lines 
correspond to the biggest.  
Note that the scaling factors $r^{-2}\tau^{1/2}$ and $r^{-1/2}\tau/\ln(\tau)$ in the
first and second plot, respectively, do not affect the sign
of the functions, and are included to attempt to spread out the different
plots shown in each figure.

\begin{figure}[ht!]
\begin{tabular}{cc}
\includegraphics[width=7cm]{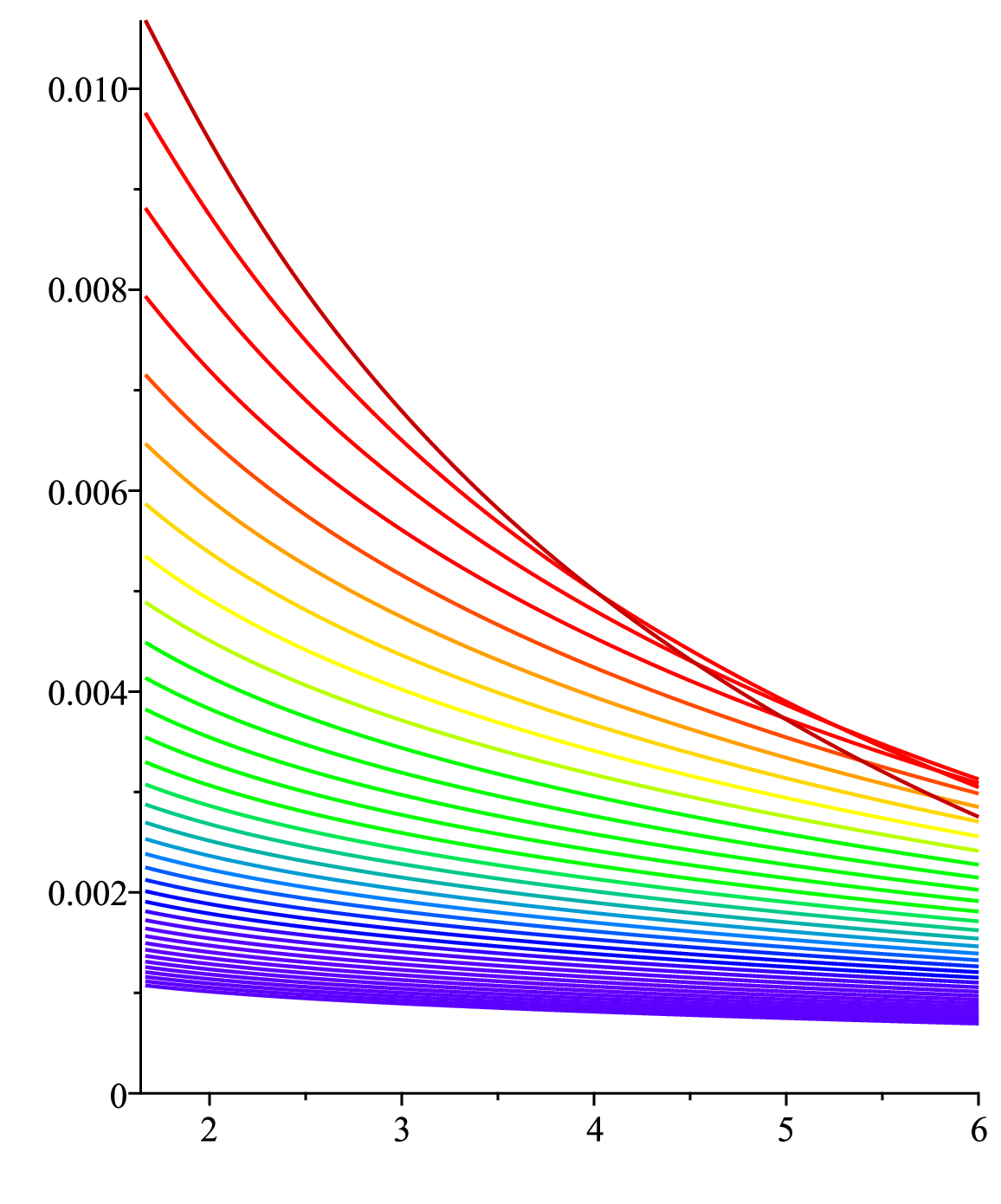} & \includegraphics[width=7cm]{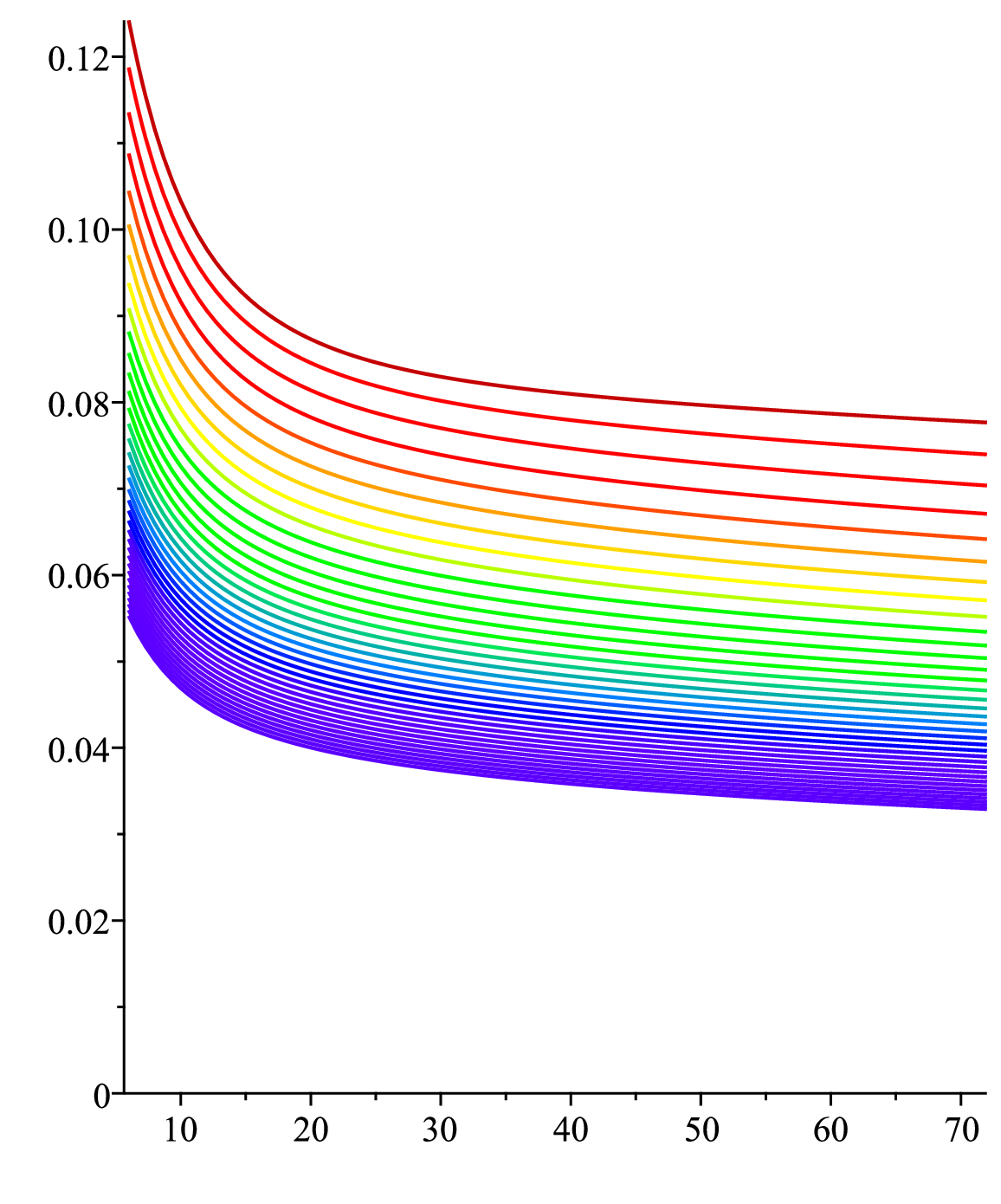}
\end{tabular}
\caption{The two plots for $s=4$, $r=6,\ldots, 41$ (from top to bottom).}
\label{s4}
\end{figure}

\begin{figure}[ht!]
\begin{tabular}{cc}
\includegraphics[width=7cm]{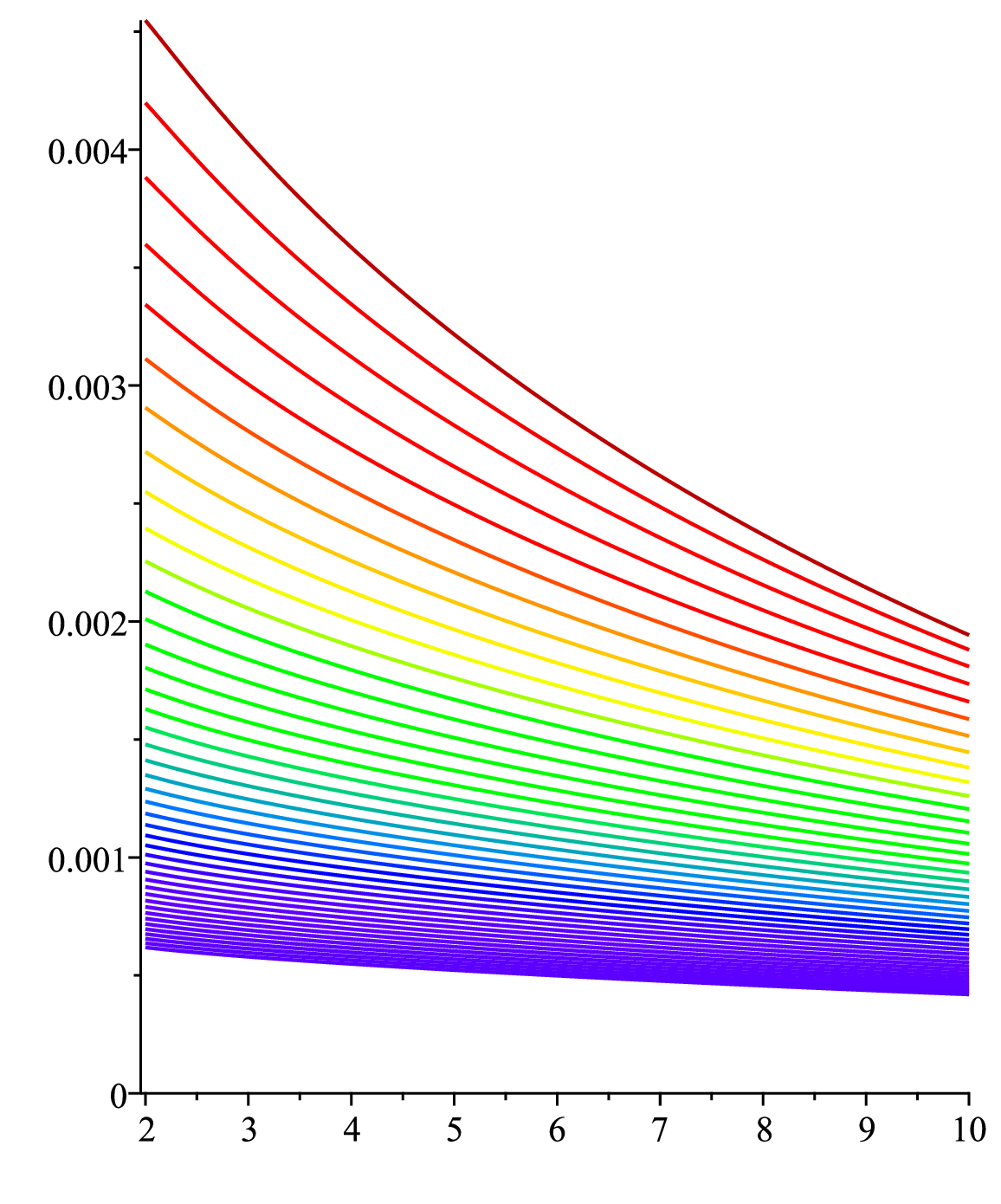} & \includegraphics[width=7cm]{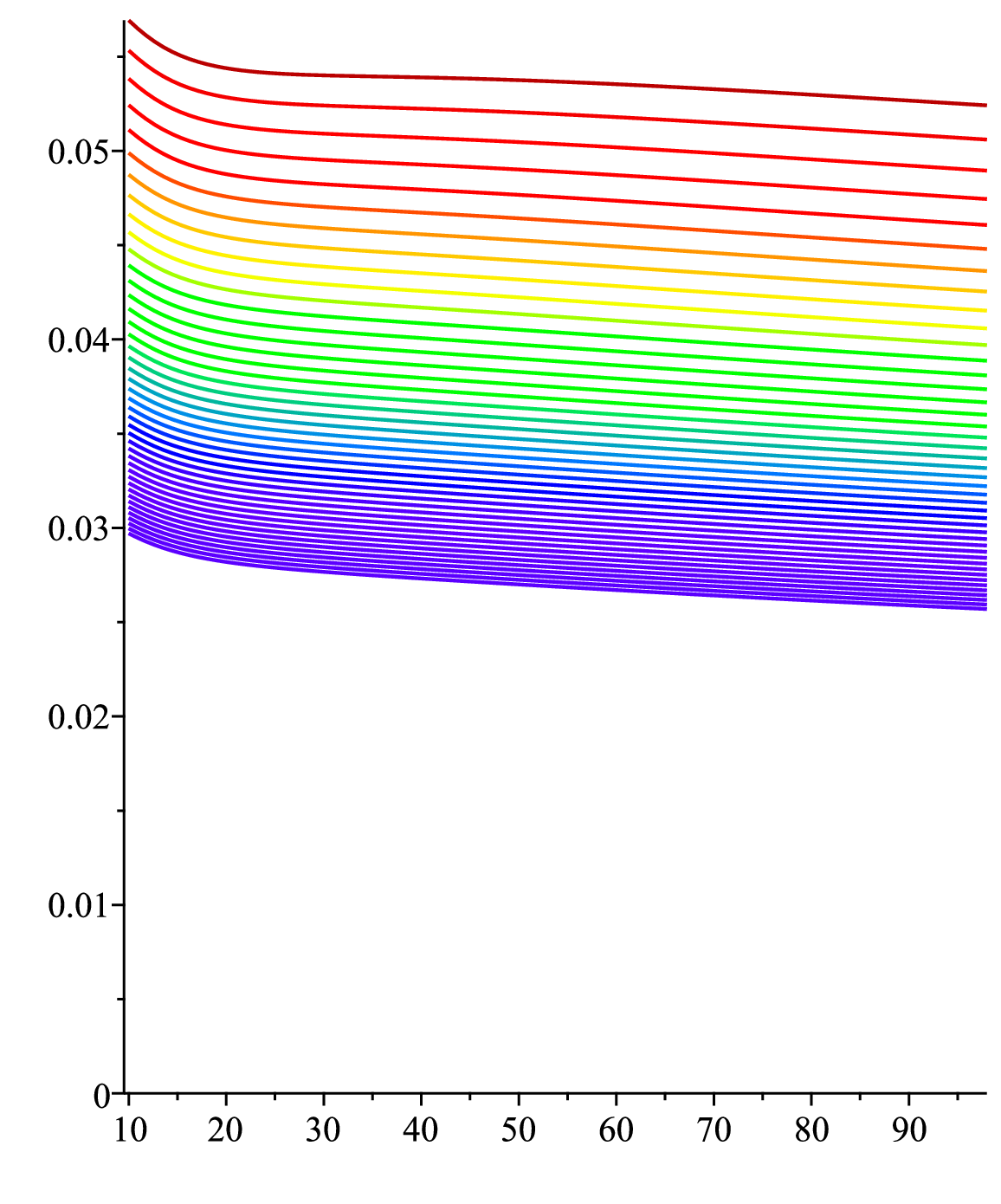}
\end{tabular}
\caption{The two plots for $s=5$, $r=12,\ldots, 54$ (from top to bottom).}
\label{s5}
\end{figure}

Consider each $(r,s)\in\mathcal{A}$ in turn: we can see that both plots are
strictly positive over the given intervals. The first plot
(on the left) shows that
$\varphi(\boldsymbol{x}_\tau)$ is strictly less than $\varphi(\boldsymbol{x}^\ast)$
for all $\tau\in [(s+1)/3,s(s-1)/2]$.  

\begin{figure}[ht!]
\begin{tabular}{cc}
\includegraphics[width=7cm]{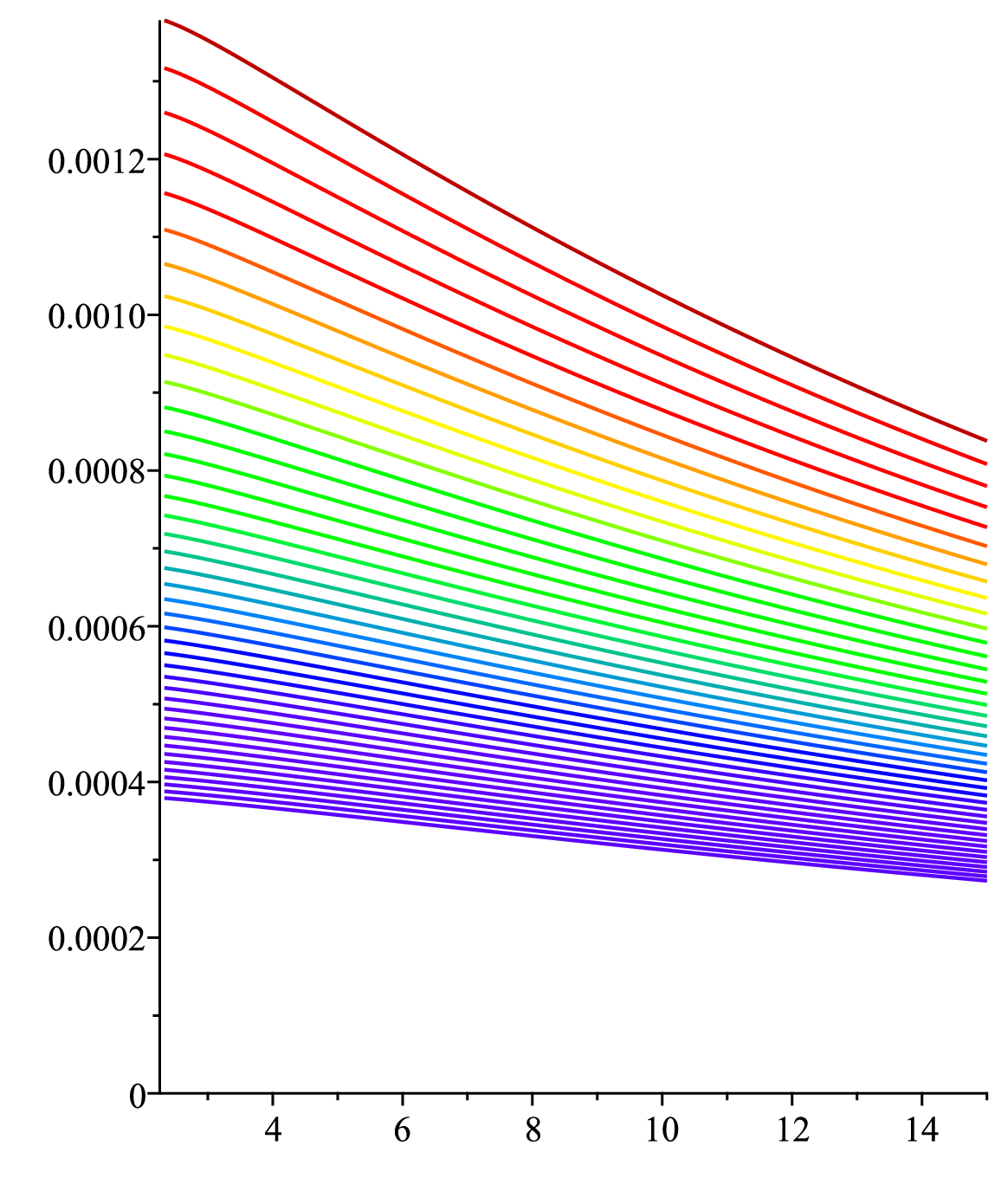} & \includegraphics[width=7cm]{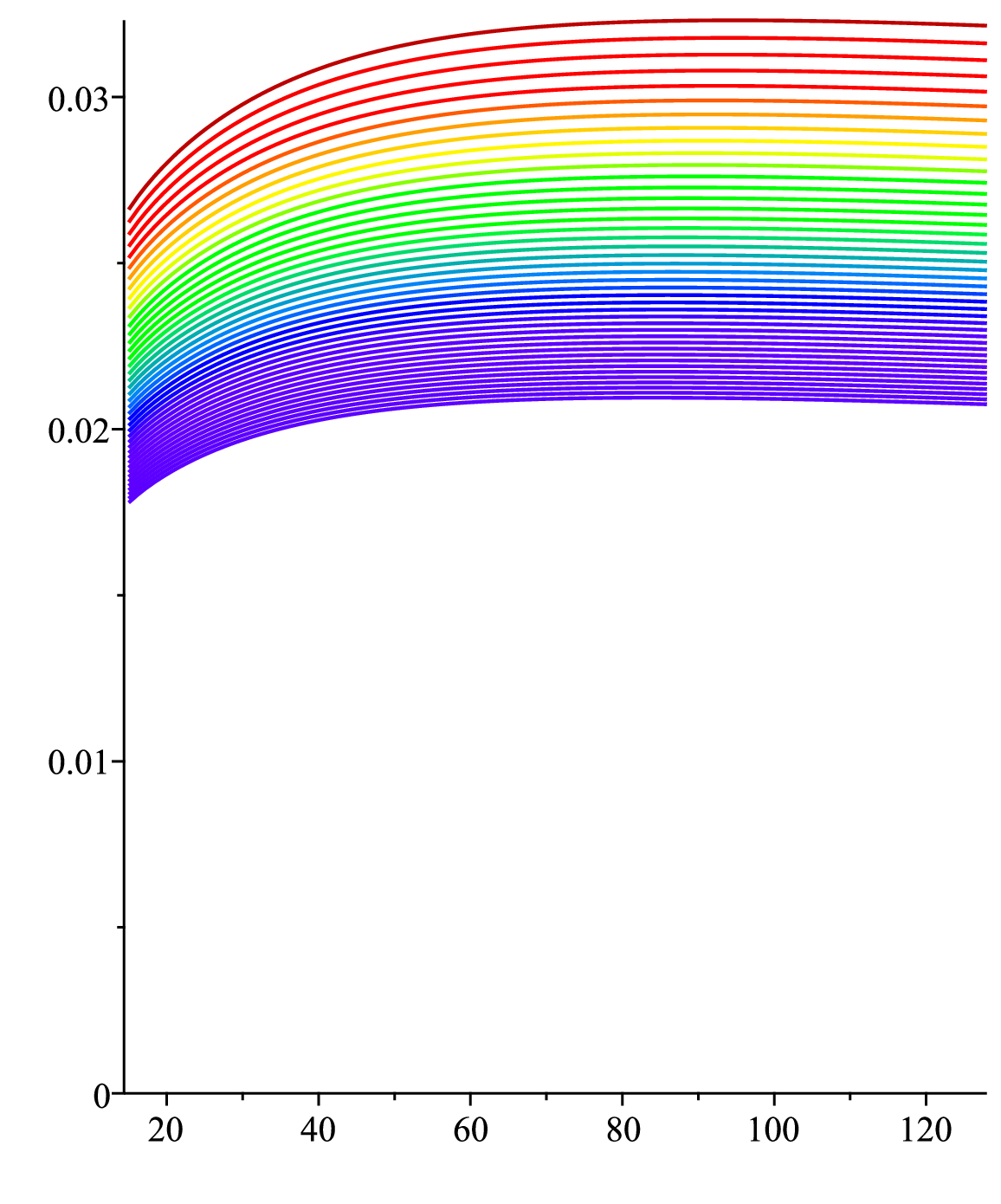}
\end{tabular}
\caption{The two plots for $s=6$, $r=28,\ldots, 69$ (from top to bottom).}
\label{s6}
\end{figure}

Combining the second plot (on the right) with Lemma~\ref{final}(i) 
or Lemma~\ref{final}(ii), 
we conclude that the function 
$\tau\mapsto \varphi'_\alpha(\boldsymbol{x}_\tau)$  
is negative for all finite $\tau\geq s(s-1)/2$.  
By (\ref{connect-derivatives}), this implies that
the function $\tau\mapsto \varphi(\boldsymbol{x}_\tau)$ has at most
one stationary point in $[s(s-1)/s,\infty)$.
But 
we know that $\boldsymbol{x}^* = \boldsymbol{x}_{\tau^*}$ is a local maximum of $\varphi(\boldsymbol{x}_\tau)$, 
and $\tau^* = (r-1)(s-1) > s(s-1)/2$ when $s\geq 3$ and $r\geq s+1$.
In particular, $\varphi'(\boldsymbol{x}_\tau) <  0$ for all
$\tau > \tau^\ast$, which proves that
the limiting point $\boldsymbol{x}_\infty$ is not a local maximum.
%
Therefore, we conclude that
$\tau^\ast$ is the unique global maximum of $\tau \mapsto\varphi(\boldsymbol{x}_{\tau})$ on $[1,\infty]$.

This argument covers the 172 remaining cases of
$(r,s)\in\mathcal{A}$ and completes the proof of Lemma~\ref{Lemma_max}.

\bigskip

\begin{figure}[ht!]
\begin{tabular}{cc}
\includegraphics[width=7cm]{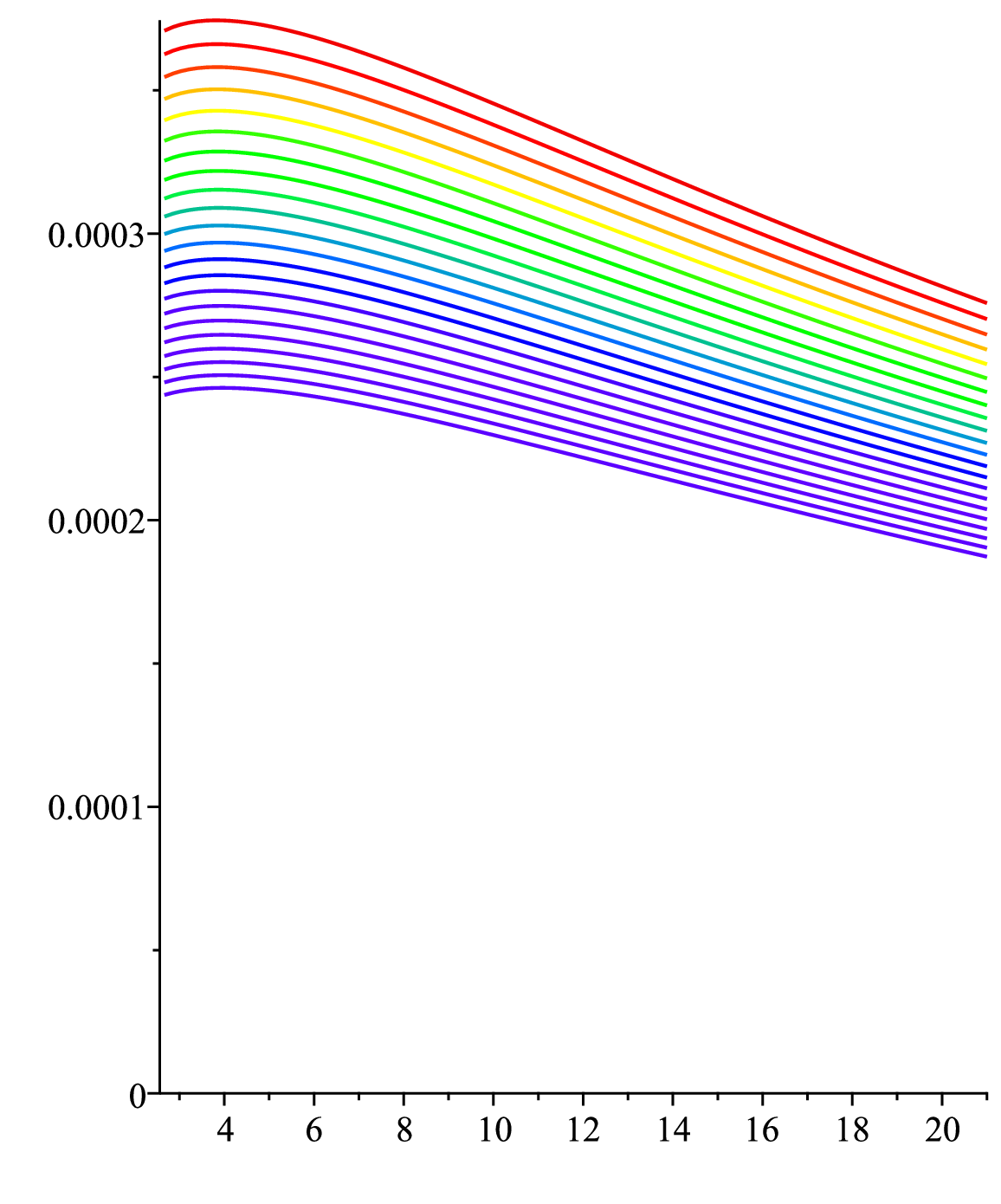} & \includegraphics[width=7cm]{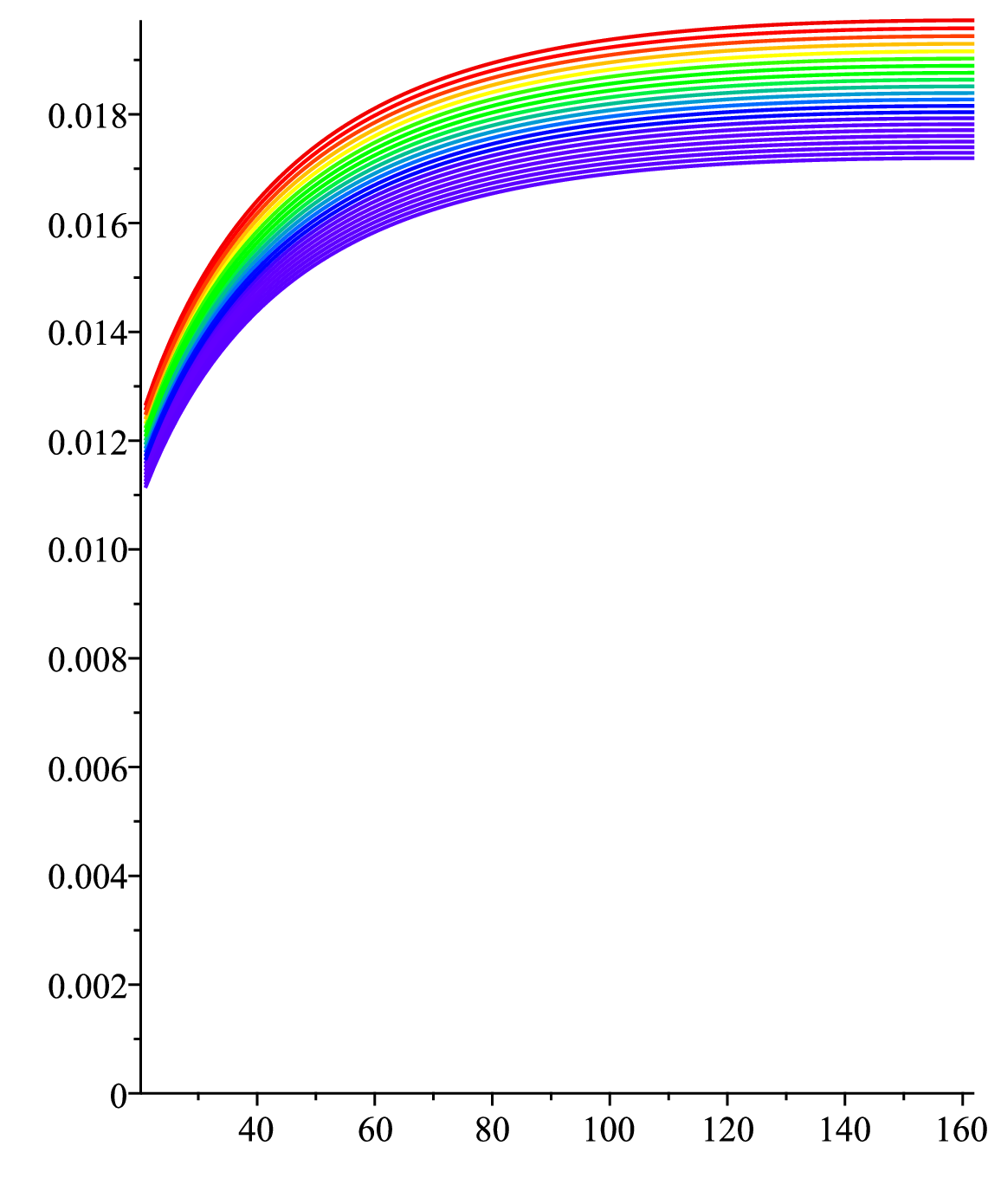}
\end{tabular}
\caption{The two plots for $s=7$, $r=65,\ldots, 86$ (from top to bottom).}
\label{s7}
\end{figure}

\begin{remark}
The above argument relies on the fact that certain explicitly-defined functions on bounded
real intervals take only positive values.  A more rigorous proof of this fact would include 
a bound on the absolute value of the derivative of each function, within the specified interval.
Using this bound, it is sufficient to approximate the value of this function on a finite set of points.
Finally, the value of the function on these points can be approximated to sufficient accuracy
using interval arithmetic.  We omit these technical details because our plots show that there is a clear
gap between each of these functions and the $x$-axis, at all points in the relevant intervals.
This gap is large enough to allow mathematical software such as Maple to compute the required values with sufficient precision.
\end{remark}

\subsection{Proof of Lemma~\ref{claima}}\label{ss:claima}

Since $\eta_\tau$ is continuous and $K_\tau$ is compact, the function $\eta_\tau$ attains
its maximum at least once on $K_\tau$.
For $\tau=1$ the region $K_\tau$ consist of one point $(\beta_\tau,\gamma_\tau,\delta_\tau)$, so the
lemma is true for this case. 
In the following we assume that $\tau>1$, which implies that $0<\alpha_\tau<1$.  

First we show that no local maximum of $\eta_\tau$ can lie on a boundary of $K_\tau$.
Recalling (\ref{c-diff}) and (\ref{d-diff}), let 
\begin{equation*}
\begin{aligned}
G_{\tau,\beta}(\gamma) &= \kappa_3\, \exp\left( - \frac{\partial\varphi}{\partial \gamma}(\alpha_{\tau},\beta,\gamma,\delta)\right)
  = \frac{\gamma^2}{(\beta-\gamma)(1-\alpha_\tau-\beta-\gamma)}, \\
D_{\tau,\beta}(\delta) &= \kappa_4\, \exp\left( - \frac{\partial\varphi}{\partial \delta}(\alpha_{\tau},\beta,\gamma,\delta)\right)
  = \frac{((s-3)\alpha_\tau + \beta + \delta)\delta}{(\alpha_\tau-\beta - \delta)^2}. 
\end{aligned}
\end{equation*}
Note that $G_{\tau,\beta}(\gamma)=\kappa_3$ if and only if 
$\frac{\partial \varphi}{\partial \gamma}(\alpha_\tau,\beta,\gamma,\delta) =0$.
For any $\beta\in (0, 1-\alpha_\tau)$, the function $G_{\tau,\beta}(\gamma)$ is 
strictly increasing with respect to $\gamma$ 
for $0 < \gamma < \min\{\beta,1-\alpha_{\tau}-\beta\}$,
and  
$$G_{\tau,\beta}(0) = 0, \qquad
\lim_{\gamma \rightarrow  \min\{\beta, 1-\alpha_\tau-\beta\} }\, G_{\tau,\beta}(\gamma) = \infty.$$
Therefore, there is a unique value of $\gamma$, say $\gamma=\tilde{\gamma}_\tau(\beta)$, which satisfies
$G_{\tau,\beta}(\gamma) = \kappa_3$
and 
$0\leq \tilde{\gamma}_\tau(\beta) < \min\{ \beta,1-\alpha_{\tau}-\beta\}$,
as $\kappa_3$ is positive and finite. 
To cover the cases $\beta=0$ and $\beta = 1-\alpha_\tau$, 
we continuously extend $\tilde{\gamma}_\tau$ and put $\tilde{\gamma}_\tau(0) = \tilde{\gamma}_\tau(1-\alpha_\tau)=0$.
Now (\ref{c-diff}) also implies that $\frac{\partial \varphi}{\partial \gamma}(\alpha_\tau,\beta,\gamma,\delta) >0$
 for $0<\gamma<\tilde{\gamma}_\tau(\beta)$  and  
 $\frac{\partial \varphi}{\partial \gamma}(\alpha_\tau,\beta,\gamma,\delta) < 0$ for 
  $\tilde{\gamma}_\tau(\beta) <\gamma< \min\{ \beta,1-\alpha_{\tau}-\beta\}.$ 
Therefore, by continuity,  for any $\beta\in [0,1-\alpha_\tau)$,
\begin{align}
\eta_\tau \, \text{ has no local maximum with } \, &
\gamma \in \{0, \min\{\beta, 1- \alpha_\tau-\beta\}\}\nonumber
\\  \text{ and } &
\min\{ \beta,1-\alpha_{\tau}-\beta\}>0.  
\label{no-1}
\end{align}

Similarly $D_{\tau,\beta}(\delta) = \kappa_4$ if and only if $\frac{\partial\varphi}{\partial\delta}(\alpha_\tau,\beta,\gamma,\delta)=0$. 
For any $0\leq \beta< \alpha_\tau$, 
the function $D_{\tau,\beta}(\delta)$ is strictly increasing with respect to $\delta$ 
for $0 < \delta < \alpha_\tau-\beta$, and  
$$D_{\tau,\beta}(0) = 0, \qquad  \lim_{\delta \rightarrow  \alpha_\tau-\beta }D_{\tau,\beta}(\delta) = \infty.$$
Therefore, there is a unique value of $\delta$, say $\delta = \tilde{\delta}_\tau(\beta)$, 
which satisfies $D_{\tau,\beta}(\delta) = \kappa_4$  and 
$0\leq \tilde{\delta}_\tau(\beta) < \alpha_\tau - \beta$ for all 
$\beta\in [0,\alpha_\tau)$, as $\kappa_4$ is positive and finite.
We continuously extend $\tilde{\delta}_\tau$ by defining $\tilde{\delta}_\tau(\alpha_\tau)=0$. 
From (\ref{d-diff}) we see  that $\frac{\partial \varphi}{\partial \delta}(\alpha_\tau,\beta,\gamma,\delta) >0$
 for $0<\delta<\tilde{\delta}_\tau(\beta)$  and  
 $\frac{\partial \varphi}{\partial \delta}(\alpha_\tau,\beta,\gamma,\delta) < 0$ for 
  $\tilde{\delta}_\tau(\beta)<\delta< \alpha_{\tau}-\beta$. 
As above, by continuity, we conclude that for any $\beta\in [0,\alpha_\tau)$,
\begin{equation}
\label{no-2}
\eta_\tau \, \text{ has no local maximum with } \,
\delta \in \{0, \alpha_\tau-\beta\} \text{ and } \alpha_\tau-\beta >0. 
\end{equation}
      
Now suppose that $(\lm{\beta},\lm{\gamma},\lm{\delta})$ is a local maximum of 
$\eta_\tau$ in $K_\tau$.  For a contradiction, suppose that this point lies on the
boundary of $K_{\tau}$.  From (\ref{no-1}), (\ref{no-2}) and
the definition of $K_{\tau}$, we know that $\lm{\beta}\in 
\{ 0,\, \min\{\alpha_\tau,\, 1-\alpha_\tau\}\}$.
\begin{enumerate}
\item[(i)] First suppose that $\lm{\beta}=0$. 
Then $\lm{\gamma}=0$, by definition of $K_{\tau}$.  
Now (\ref{no-2}) implies that $\lm{\delta} \not\in \{0,\alpha_\tau\}$,
and therefore $\lm{\delta}  \in (0,\alpha_\tau)$.  
From  \eqref{b-diff}, we find that, for $\beta$ in the neighbourhood of $0$, 
$$\frac{\partial \eta_\tau}{\partial \beta} (\beta, \lm{\gamma}, \lm{\delta}) = -\ln \beta+ O(1) > 0.$$ 
Therefore, $\lm{\beta}$ cannot be a local maximum of the function
$\beta \mapsto \eta_\tau(\beta, \lm{\gamma},\lm{\delta})$ in this case.
\item[(ii)] Next, suppose that $\lm{\beta}=\alpha_\tau<1-\alpha_\tau$. 
Then $\lm{\delta}=0$, by definition of $K_{\tau}$.
Since $\lm{\beta} < 1-\alpha_\tau$ we can apply (\ref{no-1}), which
implies that $\lm{\gamma}\not\in\{0,\min\{\alpha_\tau,1-2\alpha_\tau\}\}$ and therefore
$\lm{\gamma} \in (0,\min\{\alpha_\tau,1-2\alpha_\tau\})$.   
From  \eqref{b-diff}, we find that, for $\beta$ in the neighbourhood of $\alpha_\tau$,  
$$\frac{\partial \eta_\tau}{\partial \beta} (\beta, \lm{\gamma}, \lm{\delta}) = 
2\ln (\alpha_\tau - \beta) + O(1) < 0.$$ 
Therefore, $\lm{\beta}$ cannot be a local maximum of the function
$\beta \mapsto \eta_\tau(\beta, \lm{\gamma},\lm{\delta})$ in this case.

\item[(iii)]   
Now suppose that $\lm{\beta}=1-\alpha_\tau < \alpha_\tau$. 
Then $\lm{\gamma}=0$, by definition of $K_{\tau}$.
Since $\lm{\beta} < \alpha_\tau$, we can apply (\ref{no-2}), which
implies that $\lm{\delta}\not\in \{0, 2\alpha_{\tau}-1\}$, and hence
$\lm{\delta}  \in (0, 2\alpha_\tau-1)$. 
The map $\beta\mapsto \tilde{\gamma}_\tau(\beta)$ is continuously differentiable,
since $G'_{\tau,\beta}(\gamma) > 0$ for $0 < \gamma < \min\{\beta,1-\alpha_{\tau}-\beta\}$.
Therefore, from  \eqref{b-diff} and \eqref{c-diff},
and by definition of $\tilde{\gamma}_\tau(\beta)$,
we find that for $\beta$ in the neighbourhood of $1-\alpha_\tau$, 
\begin{align*}
     	\frac{\partial \eta_\tau}{\partial \beta} (\beta, \tilde{\gamma}_\tau(\beta), \lm{\delta}) 
 &= \frac{\partial \varphi}{\partial \beta}(\alpha_\tau, \beta, \tilde{\gamma}_\tau(\beta),\lm{\delta})
   + 0\cdot \frac{\partial \tilde{\gamma}_\tau(\beta)}{\partial \beta}\\
     	&= - \ln \left( \frac{1-\alpha_\tau-\beta}{1-\alpha_\tau-\beta - \tilde{\gamma}_\tau(\beta)} \right) 
     	 + O(1)  
     	 \\
  &= -\ln \left(1+ \frac{\kappa_3\big(\beta - \tilde{\gamma}_\tau(\beta)\big)}{\tilde{\gamma}_\tau(\beta)} \right) + O(1) < 0.
\end{align*}
Therefore, $\lm{\beta}$ cannot be a local maximum of the function
$\beta \mapsto \eta_\tau(\beta, \tilde{\gamma}_\tau(\beta),\lm{\delta})$ in this case.
\item[(iv)] Suppose that $\lm{\beta} = \alpha_\tau = 1-\alpha_\tau = \nfrac{1}{2}$.
Then $\lm{\gamma} = \lm{\delta} = 0$, by definition of $K_{\tau}$. By (\ref{b-diff}),
for $\beta$ in the neighbourhood of $\nfrac{1}{2}$, when $\alpha_\tau=\nfrac{1}{2}$,
$$\frac{\partial \eta_\tau}{\partial \beta} (\beta, 0, 0) = 
 2\ln(\nfrac{1}{2}-\beta) + O(1) < 0.
$$
Therefore the point $(\nfrac{1}{2},0,0)$ cannot be a local maximum of $\eta_{\tau}$
when $\alpha_\tau = \nfrac{1}{2}$.
\end{enumerate}

Hence there is no local maximum of $\eta_{\tau}$ on the boundary of the domain $K_t$.
Therefore the point $(\lm{\beta},\lm{\gamma},\lm{\delta})$ lies in the interior of the 
domain $K_\tau$.
and so it satisfies  the system  of equations 
$\frac{\partial \eta_\tau}{\partial \beta}=\frac{\partial \eta_\tau}{\partial \gamma} = \frac{\partial \eta_\tau}{\partial \delta}=0$. 
Define $\lm{\tau} =  \frac{1 - \alpha_\tau - \lm{\beta}}{1 - \alpha_\tau - \lm{\beta}-\lm{\gamma}}$. 
Since $\frac{\partial \eta_\tau}{\partial \gamma}=0$, using (\ref{c-diff}) we find that 
$$
     	  \lm{\tau}-1 = \frac{\lm{\gamma}}{1 - \alpha_\tau - \lm{\beta}-\lm{\gamma}} =\frac{\kappa_3 (\lm{\beta} - \lm{\gamma})}{\lm{\gamma}}.
$$
This implies that $\lm{\beta}  = ( \frac{\lm{\tau}-1}{\kappa_3} +1) \lm{\gamma}$. 
Substituting this back into the expression for $\lm{\tau}$ and 
solving with respect to $\lm{\gamma}$, we obtain
$$
 \lm{\gamma} = (1-\alpha_\tau) \frac{\lm{\tau}-1}{ \frac{(\lm{\tau}-1)^2}{\kappa_3} +2\lm{\tau}-1}.
$$
Next, since $\frac{\partial \eta_\tau}{\partial \beta}=0$ and 
$\frac{\partial \eta_\tau}{\partial \delta}=0$, using (\ref{b-diff}) and (\ref{d-diff})
implies that
$$
  \lm{\delta} = \frac{\kappa_4 (\lm{\beta}-\lm{\gamma})(1-\alpha_\tau-\lm{\beta})}
  { \kappa_2 (1-\alpha_\tau-\lm{\beta} -\lm{\gamma})}
  =  \frac{\kappa_4}{\kappa_2 \kappa_3} \lm{\tau}(\lm{\tau}-1)  \lm{\gamma} = 
  (1-\alpha_\tau) \frac{\kappa_4\, \lm{\tau}(\lm{\tau}-1)^2}{\kappa_2 \kappa_3 
       \left(\frac{(\lm{\tau}-1)^2}{\kappa_3} +2  \lm{\tau} -1\right)}.
$$  
Recalling the definitions of $p_\tau$ and $q_\tau$, we can write
$$
  \lm{\beta} = (1-\alpha_\tau) p_{\lm{\tau}}, \ \ \ \  \lm{\delta} = (1-\alpha_\tau) q_{\lm{\tau}}.
$$
Substituting 
these expressions into the equation $\frac{\partial \eta_\tau}{\partial \delta}=0$,
using (\ref{d-diff}), leads
to the identity
\[ \kappa_4\Big( \frac{1}{1-\alpha_\tau}- (1 + p_{\lm{\tau}} + q_{\lm{\tau}})\Big)^2
    = q_{\lm{\tau}}\, 
        \Big(\frac{s-3}{1-\alpha_\tau} - (s-3-p_{\lm{\tau}} - q_{\lm{\tau}})\Big), 
\]
and solving this for $(1-\alpha_\tau)^{-1}$ gives
$$
  (1-\alpha_\tau)^{-1} =  1+ p_{\lm{\tau}}+  q_{\lm{\tau}} + \frac{(s-3)q_{\lm{\tau}}}{2\kappa_4}
   \pm \sqrt{\frac{(s-2)q_{\lm{\tau}}(p_{\lm{\tau}} + q_{\lm{\tau}})}{ \kappa_4}+ \frac{(s-3)^2q_{\lm{\tau}}^2}{4\kappa_4^2} }.
$$
Since $\lm{\beta}+\lm{\delta} \leq \alpha_\tau$, which is equivalent to
$(1-\alpha_\tau)^{-1} \geq 1 + p_{\lm{\tau}} + q_{\lm{\tau}}$, 
we must take the positive sign outside of the radical.
Therefore, it follows that
$$(1-\alpha_\tau)^{-1} = (1-\alpha_{\lm{\tau}})^{-1}.$$
By monotonicity of $\alpha_\tau$ we conclude that $\lm{\tau} = \tau$, and hence 
$(\lm{\beta},\lm{\gamma}, \lm{\delta}) = (\beta_\tau, \gamma_\tau, \delta_\tau)$.  
Therefore, the point $(\beta_\tau,\gamma_\tau,\delta_\tau)$ is the only point where 
the global maximum of $\eta_\tau$ is attained on $K_\tau$.
This completes the proof of Lemma~\ref{claima}.

\subsection{Proof of Lemma~\ref{final}}\label{ss:final}

We will use (\ref{concave}). Observe that the
factor multiplying $\frac{\alpha'_\tau}{\alpha_\tau}$ in (\ref{concave})
can be bounded above by
\begin{equation}
\label{factoroutfront}
  2s-4 - \frac{(s-3)(s-4)}{s-2} + \frac{s(s-1)}{rs-r-2s}
 \leq s+2
\end{equation} 
when $r\geq s+1\geq 4$.  

We now work towards an upper bound on $\frac{\alpha'_\tau}{\alpha_\tau}$.
Write $\alpha_\tau = \frac{N_\tau}{1+N_\tau}$ where
\[ N_\tau = p_\tau + q_\tau + \frac{s-3}{2\kappa_4} q_\tau +
            \sqrt{\frac{s-2}{\kappa_4} q_\tau(p_\tau + q_\tau) + \frac{(s-3)^2}{4\kappa_4^2} q_\tau^2}. 
\]
Then
\begin{equation}
\label{Ntau}
 \frac{\alpha_\tau'}{\alpha_\tau} = \frac{N_\tau'}{N_\tau(1+N_\tau)}.
\end{equation}
Using (\ref{kappa3-less-than-2}), we see that
\[
1 - \dfrac{2}{\tau} \leq p_\tau\leq 1, \qquad \
                q_\tau \geq   \frac{\kappa_4}{\kappa_2} (\tau-4),
\]
\begin{align*}
     p_\tau' = \frac{\kappa_3(\tau^2 + \kappa_3 - 1)}{\left(\tau^2+(\kappa_3-1)(2\tau-1)\right)^2}
          \leq \frac{\kappa_3}{\tau^2+(\kappa_3-1)(2\tau-1)}
          &\leq \frac{2}{\tau^2},\\
 q_\tau' =\frac{\kappa_4\, (\tau -1)\Big((\tau -1)^3 + \kappa_3(4\tau ^2-3\tau +1)\Big)}{\kappa_2
        \left( (\tau -1)^2 + \kappa_3(2\tau -1)\right)^2} &\leq \frac{\kappa_4}{\kappa_2}.
\end{align*}
Recall that $\mu = 2\kappa_4 +s-3$.
Applying (\ref{bound_root}) and the above inequalities, we have
\[
 N_\tau \geq p_\tau + 2q_\tau + \frac{s-3}{\kappa_4} q_\tau
         \geq 1-\frac{2}{\tau} + \frac{\mu (\tau-4)}{\kappa_2}
        = \frac{\mu \tau}{\kappa_2}\left( 1 + \frac{\kappa_2}{\mu \tau}
             - \frac{4}{\tau}
           - \frac{2\kappa_2}{\mu \tau^2} \right)
\]
and
\begin{align*}
N'_\tau &\leq p'_\tau + q'_\tau + \frac{s-3}{2\kappa_4} q'_\tau +
                 \frac{(s-2)\left( q_\tau p'_\tau + q'_\tau p_\tau + 2q_\tau q'_\tau\right)
             + \frac{(s-3)^2}{2\kappa_4} q_\tau q'_\tau}{\mu  q_\tau}\\
 &=
p'_\tau + 2q'_\tau + \frac{s-3}{\kappa_4} q'_\tau +
       \frac{(s-2)p'_\tau +  2(1-\kappa_4) q'_\tau}{\mu } 
                    + \frac{(s-2) q'_\tau p_\tau}{\mu q_\tau}\\
   & \leq \frac{2}{\tau^2} + \frac{\mu}{\kappa_2} + 
      \frac{2(s-2)}{\mu \tau^2} + \frac{2\kappa_4(1-\kappa_4)}{\kappa_2 \mu}
           + \frac{s-2}{\mu(\tau-4)}\\
  &= \frac{\mu}{\kappa_2}\left(1 
  + \frac{2\kappa_4(1-\kappa_4)}{\mu^2} + \frac{(s-2)\kappa_2}{\mu^2(\tau-4)}
  + \frac{2\kappa_2}{\mu \tau^2} + \frac{2(s-2)\kappa_2}{\mu^2\tau^2} 
   \right).
\end{align*}
Substituting these bounds into (\ref{Ntau}), we conclude that 
\begin{align}
\frac{\alpha_\tau'}{\alpha_\tau}
       &\leq 
  \frac{\frac{\kappa_2 }{\mu\tau^2}
  \left(1 
	 + \frac{2\kappa_4(1-\kappa_4)}{\mu^2} 
        + \frac{(s-2)\kappa_2}{\mu^2 (\tau-4)}
   + \frac{2\kappa_2}{ \mu \tau^2} + \frac{2(s-2)\kappa_2}{\mu^2 \tau^2} 
	\right)}
    { \left(1 + \frac{\kappa_2}{\mu \tau} - \frac{4}{\tau} 
          -\frac{2\kappa_2}{\mu \tau^2} 
	             \right)
  \left(1 + \frac{2\kappa_2}{\mu \tau} - \frac{4}{\tau} 
           -\frac{2\kappa_2}{\mu \tau^2}
	             \right)
	              }.
\label{alpha_ineq}
\end{align}
Using (\ref{kappa-def}),
we claim that whenever $s\geq 3$ and $r > \rho(s)$, we have
\begin{align*}
\dfrac{1}{2} &\leq \kappa_4\leq 1 \, \text{ and } \, \kappa_4(1-\kappa_4)\leq \nfrac{1}{4},\\
s-2&\leq \mu\leq s-1,\\
s^2-s &\leq \kappa_2\leq s^2+s-3,\\
s & \leq \frac{\kappa_2}{\mu} \leq \begin{cases} 8 & \text{ if $s=3$,}\\
                                                s+3 & \text{ if $s\geq 4$}.
             \end{cases}
\end{align*}
When $s\geq 3$ we have $r > \rho(s) \geq s$, by Lemma~\ref{Llemma},
which implies that $r\geq s+1$.
This is sufficient to prove almost all of the above inequalities.
The bounds on $\kappa_4$ are clear, and lead immediately to the bounds on
$\mu$. 
The bounds on $\kappa_2$ follow from (\ref{kappa2-expression}),
using the inequality $r-2\geq s-1$.  
The lower bound on $\kappa_2/\mu$ then follows,
while for the upper bound we must be a little more precise.  
If $s\geq 5$ then by definition of $\mu$ and using the bound $r-2\geq s-1$,
we find that
\[
\mu = s-3 + 2\kappa_4
    \geq s-3 + \frac{2(s-2)}{s-1} 
    = s-1 - \frac{2}{s-1}.
\]
Therefore
\[ \frac{\kappa_2}{\mu} \leq \frac{s^2+s-3}{s-1-\frac{2}{s-1}} \leq s+3.\]
Next, suppose that $s=4$.  Then $r\geq 6$ since $r > \rho(4)$,  from Table~\ref{rhobounds}.
Hence when $s=4$ we have $\mu \geq \dfrac{5}{2}$ and
\[ \frac{\kappa_2}{\mu} \leq \frac{34}{5} < 7 = s+3.\]
Finally, if $s=3$ then $r\geq 4$ and $\mu\geq 1$, while
(\ref{kappa2-expression}) implies that $\kappa_2 \leq 8$. 


Using the inequalities proved above, the denominator of (\ref{alpha_ineq}) 
is bounded below by
\begin{align}
 \left(1 + \frac{\kappa_2}{\mu\tau}  - \frac{4}{\tau} \right.&\left. - \frac{2\kappa_2}{\mu\tau^2}
\right)
\left(1 + \frac{2\kappa_2}{\mu\tau} - \frac{4}{\tau} - \frac{2\kappa_2}{\mu\tau^2}
\right) \nonumber\\
&\geq
\left(1 + \frac{s-4}{\tau} - \frac{2s}{\tau^2}\right)
\left(1 + \frac{2(s-2)}{\tau} - \frac{2s}{\tau^2}\right) \nonumber\\
&\geq 1 + \frac{3s-8}{\tau} + \frac{2(s^2-8s+8)}{\tau^2} 
   - \frac{2s(3s-8)}{\tau^3}. \label{denomA24}
\end{align}
For the first inequality we use the lower bound for $\kappa_2/\mu$
everywhere, since the coefficients $1/\tau - 2/\tau^2$ and $2/\tau-2/\tau^2$ 
are positive.
We claim that the expression given on the right hand side of~(\ref{denomA24}) 
is bounded below by~1
 if any of conditions (i), (ii) or (iii) holds, and 
hence the same is true for the
denominator of (\ref{alpha_ineq}). 
First suppose that $s\geq 7$.  
Then $s^2-8s+8\geq 0$, and $\tau^2 - 2s\geq 0$ whenever 
$\tau \geq \min\{ 2(s+2)^2,\, (s+1)^{2.5}\}$.
It follows that (\ref{denomA24}) is bounded below by~1.
For $s=3,4,5,6$, direct substitution into (\ref{denomA24}) confirms
that the given expression is bounded below by~1
when $\tau\geq \min\{2(s+2)^2,\, (s+1)^{2.5}\}$.
This establishes the claim.

Now we consider the numerator of (\ref{alpha_ineq}).
First suppose that $s\geq 4$. Applying the bounds on
$\kappa_4$ and $\kappa_2/\mu$ gives
\[
\frac{\alpha_{\tau}'}{\alpha_{\tau}} \leq
  \frac{s+3}{\tau^2}\left( 1 + \frac{1}{2(s-2)^2} + \frac{s+3}{\tau-4}
   + \frac{4(s+3)}{\tau^2}\right).
\]
Substituting this and (\ref{factoroutfront}) into (\ref{concave}), we find
that
$\varphi_\alpha'(\boldsymbol{x}_\tau) < 0$ if
\begin{equation} (s+2)(s+3)\left(1 +  \frac{1}{2(s-2)^2} + \frac{s+3}{\tau-4} + 
       \frac{4(s+3)}{\tau^2}\right) < \tau.
\label{tau-sufficient}
\end{equation}
Assuming that $\tau\geq \min\{ 2(s+2)^2,\, (s+1)^{2.5}\}$,
since $s\geq 4$, the left hand side of (\ref{tau-sufficient}) is bounded
above by
\[ (s+2)(s+3)\left(1 + \dfrac{1}{8} + \dfrac{7}{51}
  + \dfrac{28}{55^2}\right) < 1.3 (s+2)(s+3),\]
which is bounded above by $\min\{ 2(s+2)^2,\, (s+1)^{2.5}\}$.
This uses the value $55$ as a convenient lower bound 
for $5^{2.5} \approx 55.9017$.)
Hence, when $s\geq 4$ and $r > \rho(s)$, if either (ii) or (iii) 
holds then $\varphi_\alpha'(\boldsymbol{x}_\tau)<0$.

Finally, suppose that $s=3$.  
Then $\frac{\kappa_2}{\mu^2} \leq \frac{\kappa_2}{\mu} \leq 8$, so the numerator
of (\ref{alpha_ineq}) is bounded above by
\begin{equation}
 \frac{8}{\tau^2}\, \left(\dfrac{3}{2} + \dfrac{8}{\tau-4} + \dfrac{32}{\tau^2}\right).
\label{numers3}
\end{equation}
(This also uses the inequalities $\mu\geq 1$ and $\kappa_4(1-\kappa_4)\leq \nfrac{1}{4}$
stated earlier.)
We proved earlier that the denominator
of (\ref{alpha_ineq}) is bounded below by 1. 
Since the left hand side of (\ref{factoroutfront}) is $2+\frac{3}{r-3}$ when $s=3$,
using (\ref{concave}) we see that a sufficient condition 
for $\varphi'_\alpha(\boldsymbol{x}_\tau)$ when $s=3$ is
\begin{equation}
\label{s3sufficient}
 8\left(2 + \dfrac{3}{r-3}\right)\left(\dfrac{3}{2} + \dfrac{8}{\tau-4} 
   + \dfrac{32}{\tau^2}\right) < \tau. 
\end{equation}
For (iii) we also assume that $2(r-1) \geq 4^3$, which means that $r\geq 33$,
and that $\tau\geq 32$. 
Then the left hand side of (\ref{s3sufficient}) can be bounded above by
\[
8\times \dfrac{21}{10}\times \left( \dfrac{3}{2} + \dfrac{8}{28} + \dfrac{1}{32}\right) < 
32,
\]
which proves that $\varphi'_\alpha(\boldsymbol{x}_\tau)< 0$  
when (iii) holds and $s=3$.

It remains to consider (i).
Assume that $s=3$, $r\geq 5$ and $\tau \geq 52$.
Then the left hand side of (\ref{s3sufficient}) is bounded above by
\[ 8\times \dfrac{7}{2} \times\left(\dfrac{3}{2} + \dfrac{8}{48} + \dfrac{32}{52^2}\right)
< 52,
\]
which proves that $\varphi_\alpha'(\boldsymbol{x}_\tau)<0$ whenever (i) holds with
$s=3$ and $r\geq 5$.

Next, suppose that $s=3$ and $r = 4$.  Here we must use the denominator of
(\ref{alpha_ineq}) directly, as the usual lower bound of~1 gives away too much.
Now
\[ \kappa_4 = \nfrac{1}{2},\quad \mu = 1,\quad \kappa_2 = 8.\]
Substituting these values directly into (\ref{alpha_ineq}), 
we find using (\ref{concave}) that 
$\varphi_{\alpha}'(\boldsymbol{x}_\tau) < 0$ if 
\[
 40\left(\dfrac{3}{2} + \dfrac{8}{\tau-4} + \dfrac{32}{\tau^2}\right)
   < \tau\left(1 + \dfrac{16}{\tau} + \dfrac{16}{\tau^2} - \dfrac{256}{\tau^3}
   \right),
\]
which holds for $\tau \geq 52$.
Therefore $\varphi_\alpha'(\boldsymbol{x}_\tau)<0$ whenever (i) holds with
$s=3$ and $r=4$. 
This completes the proof of Lemma~\ref{final}. \hfill $\square$


\section{More than one cycle}\label{SomeRoutineUsuallyIgnored}

In this section we prove that (\ref{omit}) holds, restated as
Lemma~\ref{more-than-one-cycle} below.   

Let $r,s\geq 2$ be fixed integers such that $\Omega(n,r,s)$ is nonempty.
For any subpartition $F$ of a partition from $\Omega(n,r,s)$, let
$\|F\|$  denote the number of parts of $F$ and let $|F|$ 
be the number of cells with points in $F$. Alternatively, 
$\|F\|$ is the number of edges in $G(F)$ and $|F|$ is the number of
vertices with positive degree in $G(F)$.

\begin{lemma}\label{Lemma_cycles}
For any partition $F\in\Omega(n,r,s)$, let
 $F_1, \ldots, F_{\ell}$ be distinct subpartitions of  $F$ 
  (possibly with some common parts) such that, for each $j=1,\ldots, 
  \ell$, the hypergraph
  $G(F_j)$ is a 1-cycle, when $\|F_j\|=1$, or a loose $\|F_j\|$-cycle,
  when $\|F_j\|\geq 2$.  
 Then
   \begin{equation}\label{ind_cycle}
    |F_1 \cup\cdots \cup F_\ell | \leq 
        (s-1)\|F_1 \cup\cdots \cup F_\ell \|,
   \end{equation}
    and equality holds if and only if all cycles
    $G(F_j)$ are loose and pairwise disjoint. 
\end{lemma} 

\begin{proof}
    We prove the result by induction on $\ell$. When $\ell=1$, the statement follows
immediately by the definition of loop and loose cycle. 
For the inductive step, assume that the statement holds for $\ell=k\geq 1$. 
We will prove that the statement holds for $\ell = k+1$.

Let $F_1,\ldots, F_{k+1}$ be subpartitions of some partition in $\Omega(n,r,s)$.
First suppose that subpartitions $F_{k+1}$ and $F_1 \cup \cdots \cup F_{k}$ do not have any 
parts in common.  
Then, using the induction hypothesis, 
     \begin{align*}
      |F_1 \cup\cdots \cup F_{k+1} | &\leq  |F_1 \cup\cdots \cup F_{k} | + |F_{k+1}|\\
      &\leq 
      (s-1) \|F_1 \cup\cdots \cup F_{k} \| + (s-1) \|F_{k+1}\|\\
     &= (s-1)\|F_1 \cup\cdots \cup F_{k+1} \|.
     \end{align*}
The first inequality becomes equality precisely when $G(F_{k+1})$ is disjoint 
from $G(F_1 \cup\cdots \cup F_{k})$. The second inequality becomes equality if and only if
all cycles $G(F_1),\ldots, G(F_{k+1})$ are loose and pairwise disjoint. 
This proves the statement 
in the case that $\left(F_1 \cup \cdots \cup F_{k}\right) \cap F_{k+1}=\emptyset$. 
Note that, in particular, this condition holds whenever $F_{k+1}$ is a 1-cycle,
since the $F_j$ are distinct and no cycle of length at least 2 can contain a loop. 
     
Next, suppose that $F_{k+1} \subseteq  F_1 \cup \cdots \cup F_{k}$. 
Then the induction hypothesis implies that \eqref{ind_cycle} holds. 
Suppose that (\ref{ind_cycle}) holds with equality with $\ell=k+1$.
Then (\ref{ind_cycle}) holds with equality for $\ell=k$, which implies by the inductive
hypothesis that $F_1,\ldots, F_k$ are loose and disjoint.  Now since $F_{k+1}\subseteq F_1\cup \cdots \cup F_k$,
it follows that $F_{k+1}=F_j$ for some $j\in [k]$.  This contradicts the assumptions of the lemma.
Hence we do not have equality in \eqref{ind_cycle} in this case.

Finally, suppose that
\[ \left(F_1 \cup \cdots \cup F_{k}\right) \cap F_{k+1} \not\in\{ \emptyset, F_{k+1}\},\]
and note that $\| F_{k+1}\| \geq 2$ in this case.
Then, by assumptions, $G(F_{k+1})$ is a loose cycle, 
and so $G(F_{k+1} \setminus \left(F_1 \cup \cdots \cup F_{k}\right))$ is 
a union of disjoint loose paths. Let $P$ be any of these paths and let 
$\ell(P)$ denote the number of edges in $P$. Observe that $P$ contains at least two 
vertices of $G(F_1 \cup \cdots \cup F_{k})$, so the number of vertices in $P$ 
which are not in $G(F_1 \cup \cdots \cup F_{k})$ is bounded above by $(s-1)\ell(P)-1$. 
Thus, we conclude that
\begin{align*}
	|F_{1}\cup\cdots\cup F_{k+1}| - 
	|F_{1}\cup\cdots\cup F_{k}| &\leq \sum_{P} \left( (s-1)\ell(P)-1\right)\\
        & \leq  -1 + (s-1)\sum_{P} \ell(P)\\
	&\leq -1+ (s-1)(\|F_{1}\cup\cdots\cup F_{k+1}\| - \|F_{1}\cup\cdots\cup F_{k}\|),
\end{align*}
where the sum is over all maximal paths $P$ in $G(F_{k+1} \setminus \left(F_1 \cup \cdots \cup F_{k}\right))$.
By the inductive hypothesis, \eqref{ind_cycle} holds for $\ell = k$, and adding
this inequality to the above gives
\[
   |F_1 \cup\cdots \cup F_{k+1} | <    (s-1)\|F_1 \cup\cdots \cup F_{k+1} \|.
\]
In particular, we see that the inequality is strict in this case, completing the proof.
\end{proof}

For convenience, we say that $F$ is a \emph{subpartition of} $\Omega(n,r,s)$ whenever
$F$ is a subpartition of some partition from $\Omega(n,r,s)$.
For two subpartitions $F_1,F_2$ of $\Omega(r,n,s)$ 
let $\Pr(F_1\! \mid\! F_2)$ denote the probability that $\mathcal{F}(n,r,s)$ contains $F_1$ 
as a subpartition, conditioned on the event that $\mathcal{F}(n,r,s)$ contains $F_2$
as a subpartition. Equivalently, $\Pr(F_1\! \mid\! F_2)$ equals the probability that a uniform 
random element of the set
\[ \{F \subseteq \Omega(n,r,s) \mid F_2 \subseteq F\}\]
 contains $F_1$. 
If $F_2 =\emptyset$ then we simply write $\Pr(F_1)$.

We say that two subpartitions $F_1$ and $F_2$ of $\Omega(n,r,s)$ are \textit{isomorphic} if 
there exists a permutation of the $rn$ points 
which induces a hypergraph isomorphism between $G(F_1)$ and $G(F_2)$. 
We will call this permutation an \textit{isomorphism} between $F_1$ and $F_2$.

The next result is proved using a generalisation of the
framework described in Section~\ref{s:framework}, but with less precision:
here, the exact constants in the asymptotic expressions do not matter.

\begin{lemma}\label{Lemma2}
Given a subpartition $F_0$ of $\Omega(n,r,s)$,
let $D_0$ denote the set of all subpartitions of $\Omega(n,r,s)$  which are
isomorphic to $F_0$. Let $F_H$ be a subpartition of $\Omega(n,r,s)$  such that 
$G(F_H)$ is a loose Hamilton cycle.
   Then the following holds as $n\rightarrow \infty$. 
   \begin{itemize}
    \item[\emph{(i)}]     
Uniformly over all subpartitions $F_0$ with $\|F_0\| = \Theta(1)$, 
    \[ \sum_{F\in D_0} \Pr(F \mid F_H) = \Theta\left(n^{|F_0|-(s-1)\|F_0\|}\right).\] 
     \item[\emph{(ii)}]   
Uniformly over all subpartitions $F_0,\, F_{\mathrm{bad}}$ with $\|F_0\|,\,\|F_{\mathrm{bad}}\| = \Theta(1)$,
        \[ \sum_{\substack{F\in D_0 \\ F \cap F_{\mathrm{bad}} = \emptyset}} 
 		 \Pr(F \mid  F_{\mathrm{bad}}\cup F_H)  = (1+ O(1/n))\, \sum_{F \in D_0} \Pr(F \mid F_H).
        \]
   \end{itemize}
\end{lemma}
\begin{proof}
For any subpartition $U$  of $F_0$, let $D_U$  denote the set of all 
subpartitions $F \in D_0$  such that there exists an isomorphism from $F$ to $F_0$ which  maps
$F \cap F_H$ to $U$. There are 
$\Theta(1)$ classes $D_U$, including the class $D_\emptyset$ corresponding to $U =\emptyset$,
and the union of these classes equals $D_0$. Note that the class $D_\emptyset$ is disjoint
from all other classes $D_U$.
Observe also that the class $D_U$ is empty unless $G(U)$ is a 
union of disjoint loose paths. 
Let $b=b(U)$ be the number of these paths, noting that $b=0$ when $U=\emptyset$.

Recalling that the length of a loose Hamiltonian cycle equals $t = n/(s-1)$, we find that
 \begin{equation}
\label{PrFmidFH}
    \Pr(F \mid F_H) =  
    \frac{p(rn - st - s\|F_0 \setminus U\|)}{p(rn - st)} 
 = \Theta\left(n^{-(s-1)(\|F_0\|-\|U\|)}\right).
 \end{equation}
 Now we bound the number of elements in $D_U$.  There are $\Theta(n^{b})$ ways to
select $b$ disjoint paths of specified lengths in the cycle $F_H$.
These paths form $F\cap F_H$.
 Next, there are $\Theta(n^{|F_0| - |U|})$ ways to choose the 
 parts in $F$ that are not determined by $F \cap F_H$. 
 (First, choose the vertices of $G(F)$ which are not contained in $G(U)$,
  in $\Theta(n^{|F_0|- |U|})$ ways, and then assign points to these
  vertices, in $\Theta(1)$ ways.)
Since $G(U)$ is a union of $b$ disjoint loose paths, we find that
	 \[
	    |U|  =  (s-1)\|U\| + b.
	 \] 
 Therefore,
\begin{equation}
\label{sizeDU} |D_U| =
    \Theta\left(n^{b} \cdot n^{|F_0| - (s-1)\|U\|- b} \right) 
\end{equation}
and hence
 \[
    \sum_{F \in D_U } \Pr(F \mid F_H) = 
    \Theta\left(n^{b}\cdot n^{|F_0| - (s-1)\|U\|- b}\cdot n^{-(s-1)(\|F_0\|-\|U\|)}\right)
    = \Theta\left(n^{|F_0|-(s-1)\|F_0\|}\right).
 \]
 Summing over all classes $D_U$ proves that
    \[ \sum_{F\in D_0} \Pr(F \mid F_H)  \leq
    \sum_{U} \sum_{F\in D_U} \Pr(F \mid F_H) 
= O\left(n^{|F_0|-(s-1)\|F_0\|}\right),\] 
as there are $\Theta(1)$ classes.  For the matching lower bound,
    \[ \sum_{F\in D_0} \Pr(F \mid F_H) \geq \sum_{F\in D_\emptyset} \Pr(F\mid F_H) = 
\Omega\left(n^{|F_0|-(s-1)\|F_0\|}\right).\] 
This completes the proof of (i).
	  
To prove (ii), let $F_{\mathrm{bad}}$ be any subpartition of $\Omega(n,r,s)$ with
$\| F_{\mathrm{bad}}\| = \Theta(1)$.
Next, for any subpartition $U$ of $F_0$, the number of $F\in D_U$ such 
that $F\cap F_{\mathrm{bad}}\neq \emptyset$ is bounded above by
$O(n^{|F_0| - (s-1)\| U\| - 1})$.
We wish to bound the number of 
subpartitions $F\in D_U$ which contain at least one part of $F_{\mathrm{bad}}$.
Arguing as in the proof of (\ref{sizeDU}),
either the number of choices of the $b$ paths in $F_H$ is reduced by a factor
of $\Omega(n)$,  or the number of choices of the
  remaining vertices of $G(F)$ (which are not covered by $G(U)$) 
 is reduced by a factor of $\Omega(n)$. 
Therefore, by (\ref{PrFmidFH}) and (\ref{sizeDU}),  we have
\[ \sum_{\substack{F\in D_0\\ F\cap F_{\mathrm{bad}}\neq \emptyset}} \Pr(F\mid F_H)
  = O(1/n)\, \sum_{F\in D_0} \Pr(F\mid F_H).
\]
Now for any $F \in D_0$ such that $F \cap F_{\mathrm{bad}} = \emptyset$, we have
\begin{align*}
       \Pr(F \mid F_{\mathrm{bad}} \cup F_H) &=
       \frac{p(rn - s \| F \cup F_{\mathrm{bad}} \cup F_H\|)}
       {p(rn - s \|F_{\mathrm{bad}} \cup F_H\|)} \\
       &= \big(1 + O(1/n)\big) \,\frac{p(rn - s \| F  \cup F_H\|)}{p(rn - s \| F_H\|)} \nonumber \\
         	&= \big(1+ O(1/n)\big)\,  \Pr(F \mid  F_H), 
\end{align*}
and hence
\begin{align*}
\sum_{F \in D_0} \Pr(F \mid F_H) &=  \big(1+O(1/n)\big)\, 
\sum_{\substack{F\in D_0 \\ F \cap F_{\mathrm{bad}} = \emptyset}} \Pr(F\mid F_H)\\
  &= \big(1+O(1/n)\big)
           \sum_{\substack{F\in D_0 \\ F \cap F_{\mathrm{bad}} = \emptyset}} 
 		 \Pr(F \mid  F_{\mathrm{bad}}\cup F_H),
\end{align*}
completing the proof.
\end{proof}

 Recall that $X_k$ denotes the number of subpartitions $F$ in $\mathcal{F}(n,r,s)$ such that $G(F)$ is a loose $k$-cycle, for $k\geq 2$, and 
 $X_1$ is the number of parts $F$ containing more than one point from some cell.
 The random variable $Y=X_t$ corresponds to the number of loose Hamilton cycles.
 We are ready to prove (\ref{omit}), which enables us to check 
condition (A2) of Theorem~\ref{thm:janson} by simply estimating the
 quantities $\E(YX_k)/\E Y$ for (fixed) positive integers $k$.
 
\begin{lemma} 
\label{more-than-one-cycle}
 Let $k, j_1,\ldots,j_k$ be fixed nonnegative integers. Then,  as $n\rightarrow \infty$,
   \[
    \frac{\E (Y (X_1)_{j_1} \cdots (X_k)_{j_k} )}{\E Y}  \sim \prod_{i=1}^{k}      \left( \frac{\E (Y X_{i} )}{\E Y} \right)^{j_i}.
   \]
 \end{lemma} 
 
 \begin{proof} 
Note that $\E(Y (X_1)_{j_1} \cdots (X_k)_{j_k})$ is the expected number of
ways to choose, in $\mathcal{F}(n,r,s)$, subpartitions corresponding to the following 
structures:
a loose Hamilton cycle, and $j_1$ distinct 1-cycles, and $j_2$
distinct loose 2-cycles, and so on, up to $j_k$ distinct loose $k$-cycles.
 Let $\ell = j_1+\cdots+j_k$ be the total number of short cycles, and let 
$k_1,\ldots,k_\ell$ be the sequence of cycle lengths in non-decreasing order.
That is, for $q \in\{1, \ldots, \ell\}$, the value of $k_q$ equals $\rho$ defined by
  $\sum_{i=1}^{\rho-1}j_i <  q \leq \sum_{i=1}^{\rho} j_i$.
 By linearity of expectation, observe that 
\[
		\frac{\E (Y (X_1)_{j_1} \ldots (X_k)_{j_k} )}{\E Y} =
		\frac{1}{\E Y}\sum_{F_H} \sum_{F_1,\ldots, F_\ell}
		 \Pr\left( F_{1} \cup \cdots \cup F_{\ell} \cup F_H\right),
\]
where the first sum is over subpartitions $F_H$ of $\Omega(n,r,s)$ such that 
$G(F_H)$ is a loose Hamilton cycle and the second 
sum is over distinct subpartitions $F_1, \ldots, F_\ell$ of $\Omega(n,r,s)$ such that  
for $q=1,\ldots, \ell$, the graph $G(F_q)$ is a 1-cycle when $k_q=1$,
 or a loose $k_q$-cycle when $k_q\geq 2$. By symmetry, any  
$F_H$ gives the same contribution to the first sum. Fixing one 
subpartition $F_H$ and multiplying by $\E Y /\Pr(F_H)$ (which equals the number of choices of $F_H$), we find that
\[
		\frac{\E (Y (X_1)_{j_1} \ldots (X_k)_{j_k} )}{\E Y} =
		 \sum_{F_1,\ldots, F_\ell}
		 \Pr\left(F_{1} \cup \cdots \cup F_{\ell} \mid F_H \right).
\]
Let  $D_0$ be the set of subpartitions isomorphic to a fixed subpartition $F_0$ of $\Omega(n,r,s)$. 
If $G(F_0)$ is a union of $\ell$ cycles of lengths 
$k_1,\ldots, k_\ell$ then  we have, by Lemma \ref{Lemma2}(i),
   	\[ 
   			\sum_{\substack{F_1,\ldots,F_\ell\\F_1\cup \cdots \cup F_\ell \in D_0}} 
   			\Pr (F_1 \cup \cdots \cup F_{\ell} \mid F_H) = \Theta(n^{|F_0|-(s-1)\|F_0\|}).
   	\]
   	 Using Lemma \ref{Lemma_cycles}, we obtain that 
   	 \[
   	     \sum_{F_1,\ldots, F_\ell} \Pr (F_1 \cup \cdots \cup F_{\ell} \mid F_H) 
   	      \sim  \sum_{{\rm disjoint}\, F_1,\ldots, F_\ell} \Pr (F_1 \cup \cdots \cup F_{\ell} \mid F_H), 
   	 \]
where the second sum  is over summands of the first sum with the additional restriction that
subpartitions $F_1,\ldots,F_\ell$ are pairwise disjoint. 
In fact, by Lemma \ref{Lemma_cycles} and Lemma \ref{Lemma2}(i), the first sum is asymptotically equal to the sum over the summands
such that all cycles $G(F_1), \ldots, G(F_k)$ are loose and pairwise disjoint. 
Hence we can include or omit any other summands without changing the asymptotic
value of the sum.

For $q=1,\ldots, \ell$, let $D_q$ denote the set of all subpartitions $F$
of $\Omega(n,r,s)$  that $G(F)$ is a 1-cycle, when $k_q=1$ or a loose $k_q$-cycle,
when $k_q\geq 2$. 
It remains to prove that
\begin{equation}\label{for_ind}
    \sum_{{\rm disjoint}\, F_1,\ldots, F_\ell} \Pr (F_1 \cup \cdots \cup F_{\ell} \mid F_H)
     \sim \left(\sum_{F\in D_1} \Pr (F \mid F_H) \right) \cdots  
     \left(\sum_{F\in D_\ell} \Pr (F \mid F_H) \right).
   \end{equation}
 We  prove \eqref{for_ind} by induction on $\ell$. For $\ell=1$ the 
 left hand side and right hand side coincide. 
Now suppose that (\ref{for_ind}) holds for some $\ell=p\geq 1$.
 To prove that (\ref{for_ind}) holds with $\ell=p+1$, we use Lemma \ref{Lemma2}(ii) with 
$F_{\mathrm{bad}} = F_1\cup\cdots \cup F_{p}$ and a subpartition $F_0 \in D_{p+1}$. 
 Then, we have  $D_0 = D_{p+1}$ and 
 \begin{align*} 
     \sum_{{\rm disjoint}\, F_1,\ldots, F_{p+1}} 
     &\Pr (F_1 \cup \cdots \cup F_{p+1} \mid F_H) 
     \\&= 
     \sum_{{\rm disjoint}\, F_1,\ldots, F_p} \Pr (F_1 \cup \cdots \cup F_{p} \mid F_H) 
     \sum_{\substack{F \in D_{p+1}\\ F \cap F_{\mathrm{bad}} = \emptyset}} 
     \Pr (F  \mid F_{\mathrm{bad}}\cup F_H)\\
     &\sim \sum_{{\rm disjoint}\, F_1,\ldots, F_p} \Pr (F_1 \cup \cdots \cup F_p \mid F_H)
         \sum_{F \in D_{p+1}} \Pr (F \mid F_H). 
 \end{align*}
Applying the induction hypothesis completes the proof.
    \end{proof}

\end{document}